\documentclass{amsart}

\usepackage{amsmath,amssymb}
\usepackage{bbold,mathrsfs,amscd,lineno,hyperref}
\usepackage{mathtools,afterpage}
\usepackage{tikz}
\usepackage{mathrsfs}
\usepackage{yfonts,amsbsy}
\usepackage{epsfig,amsfonts,color,amscd,url}
\usepackage{graphicx}

\DeclareFontFamily{U}{mathb}{\hyphenchar\font45}
\DeclareFontShape{U}{mathb}{m}{n}{ <5> <6> <7> <8> <9> <10> gen * mathb <10.95> mathb10 <12> <14.4> <17.28> <20.74> <24.88> mathb12 }{}
\DeclareSymbolFont{mathb}{U}{mathb}{m}{n}
\DeclareFontSubstitution{U}{mathb}{m}{n}
\DeclareMathSymbol{\selfmap}{3}{mathb}{"FD}

\newcommand\full{{\operatorname{full}}}

\newcommand{\wC}{\widehat{\mathbb{C}}}
\newcommand{\Disk}{\mathbb{D}}

\newcommand{\bfilled}{{\boldsymbol{\mathfrak K} }}

\newcommand{\filled}{{\mathfrak K }}
\newcommand{\Julia}{{\mathfrak J }}
\newcommand{\Post}{{\mathfrak P }}
\newcommand{\Hub}{{\mathfrak T}}

\newcommand{\per}{{\operatorname{per}}}

\newcommand{\ExpN}{{\operatorname{EN}}}

\newcommand{\AD}{{\operatorname{AD}}}
\newcommand{\WAD}{{\operatorname{WAD}}}

\newcommand{\bush}{{\mathfrak B }}
\newcommand{\bbush}{{\boldsymbol {\mathfrak B} }}
\newcommand{\bUpsilon}{{\boldsymbol {\Upsilon} }}

\newcommand{\balpha}{{\ba}}

\newcommand{\bPi}{{\boldsymbol \Pi}}

\newcommand{\can}{{\operatorname{can}}}

\renewcommand{\sp}{{\ \ }}

\newcommand{\hide}[1]{}

\renewcommand{\SS}{{\mathcal S}}
\newcommand{\Fam}{{\mathcal F}}
\newcommand{\FamG}{{\mathcal G}}
\newcommand{\FamS}{{\mathcal S}}

\newcommand{\Width}{{\mathcal W}}

\newcommand\new{{\operatorname{new}}}

\newcommand{\rrto}{\rightrightarrows}

\renewcommand{\AA}{{\mathcal A}}

\newcommand{\bbt}{{\mathbf t}}
\newcommand{\rn}[1]{\text{\romannumeral#1}}%

\everymath{\displaystyle}

\usepackage{enumitem}

\newtheorem{mthm}{Theorem}

\renewcommand{\max}{{\operatorname{max}}}

\begin{document}

\newtheorem{thm}{Theorem}[section]
\newtheorem{cor}[thm]{Corollary}
\newtheorem{lem}[thm]{Lemma}
\newtheorem{prop}[thm]{Proposition}
\newtheorem{obs}[thm]{Observation}
\newtheorem{rem}[thm]{Remark}
\newtheorem{notation}[thm]{Notation}
\newtheorem{exc}[thm]{Exercise}
\newtheorem{exe}[thm]{Exercise}
\newtheorem{problem}[thm]{Problem}

\newtheorem{defn}[thm]{Definition}

\theoremstyle{plain}
\newtheorem*{thmA}{Main Theorem (physical version)}
\newtheorem*{thmB}{Main Theorem (dynamical version)}
\newtheorem*{thmC}{Theorem (Global Lee-Yang-Fisher Current)}
\newtheorem*{thmRIG}{Local Rigidity Theorem}
\newtheorem*{LY Theorem}{Lee-Yang Theorem}
\newtheorem*{LY Theorem2}{General Lee-Yang Theorem}
\newtheorem*{LY TheoremBC}{Lee-Yang Theorem with Boundary Conditions}
\newtheorem*{LY Theorem2BC}{General Lee-Yang Theorem with Boundary Conditions}
\newtheorem*{conj}{Conjecture}

\numberwithin{equation}{section}
%\numberwithin{figure}{section}

\newcommand{\figref}[1]{Figure~\ref{#1}}
\newcommand{\Rmig}{R} %the Migdal-Kadinoff renormalization in U,V,W coords.
\newcommand{\Cmig}{C} %the invariant ``cylinder'' for \Rmig
\newcommand{\Cmigbl}{C_0}
\newcommand{\TOPmig}{\mathrm{T}} %the invariant ``cylinder'' for \Rmig
\newcommand{\BOTTOMmig}{\mathrm{B}} %the invariant ``cylinder'' for \Rmig
%fixed points:
\newcommand{\FIXmig}{b} %the fixed points on the invariant ``cylinder'' for \Rmig
\newcommand{\CFIXmig}{e} %the complex attacting fixed points for \Rmig
%indeterminant points:
\newcommand{\INDmig}{a} %the indeterminant points for \Rmig
\newcommand{\Smig}{S} %the principal locus for \Rmig, i.e. the line U+2V+W=0

\newcommand{\thin} {{\mathrm{thin}}}
\newcommand{\thick} {{ \mathrm{thick}}}
\newcommand{\loc}{{\mathrm {loc}}}
\newcommand{\hor}{{\mathrm {hor}}}
\newcommand{\ver}{{\mathrm {ver}}}
\newcommand{\ess}{{\mathrm{ess}}}
\newcommand{\ness}{{\mathrm{ne}}}
\newcommand{\pro}{{\mathrm{pr}}}

\newcommand{\locflux} {\WW^\perp_\loc}

\newcommand{\re}{{\mathrm{(r)}}}
\newcommand{\essre}{{\mathrm{ess(r)}}}
\newcommand{\nere}{{\mathrm{ne(r)}}}
\newcommand{\verre}{{\mathrm{ver(r)}}}

\newcommand{\Cyl}{{\mathrm{Cyl}}}
\newcommand{\hsubset}{{ \,\underset{h}{\subset}\, }    }
\newcommand{\hsupset}{{ \, \underset{h}{\supset} \, }     }

\newcommand{\lp}{{\,\mathrm {dominates}\,}}

\newcommand{\Rphys}{\mathcal R} %the Physical renormalization in Z,T,S coords.
\newcommand{\Cphys}{\mathcal C} %the invariant cylinder for \Rphys
\newcommand{\Solid}{\mathcal{SC}}
\newcommand{\Solidmig}{SC}
\newcommand{\Secphys}{\Pi}
\newcommand{\Secmig}{P}
\newcommand{\TOPphys}{\mathcal T} %Top the invariant cylinder for \Rphys
\newcommand{\Cphystl}{{\mathcal C}_1}
\newcommand{\Cphysbl}{{\mathcal C}_0}
\newcommand{\Cphyslow}{{\mathcal C}_*}
\newcommand{\BOTTOMphys}{\mathcal B} %Top the invariant cylinder for \Rphys
%fixed points:
\newcommand{\FIXphys}{\beta} %the fixed points on the invariant cylinder for \Rphys
\newcommand{\CFIXphys}{\eta} %the complex attacting fixed points for \Rphys
%indeterminant points:
\newcommand{\INDphys}{\alpha} %the indeterminant points for \Rmig
\newcommand{\Sphys}{\mathcal S} %the principal locus for \Rphys, i.e. the curve line z^2+2zt+1=0
\newcommand{\Icurve}{\mathcal G} %image of the indeterminacy pts
\newcommand{\Imig}{G} 
\newcommand{\PI}{{\mathcal A}}
\newcommand{\crittemps}{{\mathscr C}} %The set of all critical tempuratures
\newcommand{\epoints}{{\mathscr E}} %The set of all endpoints of high-tempurature hairs

\newcommand{\Line}{L}
\newcommand{\Lzero}{L_0}
\newcommand{\Lone}{L_1}
\newcommand{\Ltwo}{L_2}
\newcommand{\Lthree}{L_3}
\newcommand{\Lfour}{L_4}

\newcommand{\LLzero}{\LL_0}
\newcommand{\LLone}{\LL_1}

\newcommand{\Div}{{\rm Div}}
\newcommand{\Tongue}{\Upsilon}

\newcommand{\Par}{{{\mathcal P}}}
\newcommand{\inv}{{\iota}}

\newcommand{\leg}{{\mathrm {leg}}}
\newcommand{\il}{{\mathrm {il}}}

\newcommand{\Mconv}{{\rm M}}
\newcommand{\Hconv}{{\rm H}}
\newcommand{\tconv}{{\rm t}}

\newcommand{\Weight}{W}

\newcommand{\KSQRT}{\widehat \KK}

\newcommand{\ex}{{\mathrm{exc}}}
\newcommand{\out}{{\mathrm{out}}}

\newcommand{\vertbondpp}{\, \underset{\oplus}{\overset{\oplus}{|}} \, }
\newcommand{\vertbondpm}{\, \underset{\ominus}{\overset{\oplus}{|}} \,}
\newcommand{\vertbondmp}{\, \underset{\oplus}{\overset{\ominus}{|}} \, }
\newcommand{\vertbondmm}{\, \underset{\ominus}{\overset{\ominus}{|}} \, }

\font\nt=cmr7

\def\note#1
{\marginpar
{\nt $\leftarrow$
\par
\hfuzz=20pt \hbadness=9000 \hyphenpenalty=-100 \exhyphenpenalty=-100
\pretolerance=-1 \tolerance=9999 \doublehyphendemerits=-100000
\finalhyphendemerits=-100000 \baselineskip=6pt
#1}\hfuzz=1pt}

\newcommand{\bignote}[1]{\begin{quote} \sf #1 \end{quote}}

% \long\def\bignote#1{}

% \centerline{NEW MACROS (2019) } 

\newcommand{\lobr}{{\mathrm{lo-br}}}
\newcommand{\sh}{{\mathrm{sh}}}
\newcommand{\sei}{{\mathrm{si}}}
\newcommand{\pe}{{\mathrm{pe}}}

\newcommand{\ol}{\overline}
\newcommand{\ul}{\underline}

\newcommand{\pp}{{\mathfrak{p}}}
\newcommand{\qq}{{\mathfrak{q}}}
\newcommand{\bb}{{\mathfrak{b}}}

\newcommand{\Kfilled}{{\mathcal {K} }}
\newcommand{\Sec} {{{S}}}
\newcommand{\Zec} {{{Z}}}
% \newcommand{\Zay}{{\mathcal{?}}}
% \newcommand{\Hub}{\TT}

% \centerline{OLD MACROS continued} 

%Commands specific to the renormalization operators:

%% %% NEW MACROS (originated in Dec 2020)

\newcommand{\Jul}{{\mathrm{J}}}
\newcommand{\Mandel}{{\mathcal{M}}}
\newcommand{\KK}{{\mathcal T}}
\newcommand{\Bouq}{{B}}

\newcommand{\base}{{\circ}}
\newcommand{\com}{{\mathrm{com}}}
\newcommand{\perb}{{q}}
\newcommand{\lo}{{\mathrm{lo}}}

%%%%%%%%%%%%%%%%%%%%%%%

\newcommand{\QED}{\rlap{$\sqcup$}$\sqcap$\smallskip}

\def\sss{\subsubsection}

\newcommand{\correspond}{\Psi}
\newcommand{\conjugacy}{\psi}

\newcommand{\rank}{\rm rank}
\newcommand{\di}{\partial}
\newcommand{\dibar}{\bar\partial}
\newcommand{\hookra}{\hookrightarrow}
\newcommand{\ra}{\rightarrow}
\newcommand{\hra}{\hookrightarrow}
\newcommand{\imply}{\Rightarrow}
\def\lra{\longrightarrow}
\newcommand{\wc}{\underset{w}{\to}}
\newcommand{\tu}{\textup}

\def\ssk{\smallskip}
\def\msk{\medskip}
\def\bsk{\bigskip}
\def\noi{\noindent}
\def\nin{\noindent}
\def\lqq{\lq\lq}
\def\sm{\setminus}
\def\bolshe{\succ}
\def\ssm{\smallsetminus}
\def\tr{{\text{tr}}}
\def\Crit{{\mathrm{Crit}}}

\newcommand{\ctg}{\operatorname{ctg}}
\newcommand{\diam}{\operatorname{diam}}
\newcommand{\dist}{\operatorname{dist}}
\newcommand{\Hdist}{\operatorname{H-dist}}
\newcommand{\cl}{\operatorname{cl}}
\newcommand{\inter}{\operatorname{int}}
\renewcommand{\mod}{\operatorname{mod}}
\newcommand{\card}{\operatorname{card}}
\newcommand{\tl}{\tilde}
\newcommand{\ind}{ \operatorname{ind} }
\newcommand{\Dist}{\operatorname{Dist}}
\newcommand{\Graph}{\operatorname{Graph}}
\newcommand{\len}{\operatorname{\l}}
\newcommand{\vol}{\operatorname{vol}}

\renewcommand{\Re}{\operatorname{Re}}
\renewcommand{\Im}{\operatorname{Im}}

\newcommand{\orb}{\operatorname{orb}}
\newcommand{\HD}{\operatorname{HD}}
\newcommand{\supp}{\operatorname{supp}}
\newcommand{\id}{\operatorname{id}}
\newcommand{\length}{\operatorname{length}}
\newcommand{\dens}{\operatorname{dens}}
\newcommand{\meas}{\operatorname{meas}}
\newcommand{\area}{\operatorname{area}}
\renewcommand{\Im}{\operatorname{Im}}

\renewcommand{\d}{{\diamond}}

\newcommand{\lef}{{\mathrm{left}}}
\newcommand{\righ}{{\mathrm{right}}}

\newcommand{\Dil}{\operatorname{Dil}}
\newcommand{\Ker}{\operatorname{Ker}}
\newcommand{\tg}{\operatorname{tg}}
\newcommand{\codim}{\operatorname{codim}}
\newcommand{\isom}{\approx}
\newcommand{\comp}{\circ}
\newcommand{\esssup}{\operatorname{ess-sup}}
\newcommand{\Rat}{{\mathrm{Rat}}}
\newcommand{\hot}{{\mathrm{h.o.t.}}}
\newcommand{\Conf}{{\mathrm{Conf}}}

\newcommand{\SLa}{\underset{\La}{\Subset}}

\newcommand{\const}{\mathrm{const}}
\def\loc{{\mathrm{loc}}}
\def\fib{{\mathrm{fib}}}
\def \br{{\mathrm{br}}}

\newcommand{\eps}{{\epsilon}}
\newcommand{\epsi}{{\epsilon}}
\newcommand{\veps}{{\varepsilon}}

\newcommand{\Ga}{{\Gamma}}
\newcommand{\De}{{\Delta}}
\newcommand{\de}{{\delta}}
\newcommand{\la}{{\lambda}}
\newcommand{\La}{{\Lambda}}
\newcommand{\si}{{\sigma}}
\newcommand{\Si}{{\Sigma}}
\newcommand{\Om}{{\Omega}}
\newcommand{\om}{{\omega}}
\newcommand{\Ups}{{\Upsilon}}

\newcommand{\al}{{\alpha}}
\newcommand{\ba}{{\mbox{\boldmath$\alpha$} }}
\newcommand{\be}{{\beta}}
\newcommand{\bbe}{{\mbox{\boldmath$\beta$} }}
\newcommand{\bk}{{\boldsymbol{\kappa}}}
\newcommand{\bg}{{\boldsymbol{\gamma}}}

\newcommand{\bare}{{\bar\eps}}

\newcommand{\Ray}{{\mathcal R}}
\newcommand{\Eq}{{\mathcal E}}
\newcommand{\PR}{PR}

\newcommand{\AAA}{{\mathcal A}}
\newcommand{\BB}{{\mathcal B}}
\newcommand{\CC}{{\mathcal C}}
\newcommand{\DD}{{\mathcal D}}
\newcommand{\EE}{{\mathcal E}}
\newcommand{\EEE}{{\mathcal O}}
\newcommand{\II}{{\mathcal I}}
\newcommand{\FF}{{\mathcal F}}
\newcommand{\GG}{{\mathcal G}}
\newcommand{\JJ}{{\mathcal J}}
\newcommand{\HH}{{\mathcal H}}
\newcommand{\LL}{{\mathcal L}}
\newcommand{\MM}{{\mathcal M}}
\newcommand{\NN}{{\mathcal N}}
\newcommand{\OO}{{\mathcal O}}
\newcommand{\PP}{{\mathcal P}}
\newcommand{\QQ}{{\mathcal Q}}
\newcommand{\QM}{{\mathcal QM}}
\newcommand{\QP}{{\mathcal QP}}
\newcommand{\QL}{{\mathcal Q}}

\newcommand{\RR}{{\mathcal R}}
\newcommand{\SSS}{{\mathcal S}}
\newcommand{\TT}{{\mathcal T}}
\newcommand{\TTT}{{\mathcal P}}
\newcommand{\UU}{{\mathcal U}}
\newcommand{\VV}{{\mathcal V}}
\newcommand{\WW}{{\mathcal W}}
\newcommand{\XX}{{\mathcal X}}
\newcommand{\YY}{{\mathcal Y}}
\newcommand{\ZZ}{{\mathcal Z}}

\newcommand{\AS}{{\mathcal{AS}}}
\newcommand{\SAS}{{\mathcal{SAS}}}

\newcommand{\A}{{\Bbb A}}
\newcommand{\BBB}{{\Bbb B}}
\newcommand{\C}{{\Bbb C}}
\newcommand{\bC}{{\bar{\Bbb C}}}
\newcommand{\D}{{\Bbb D}}
\newcommand{\Hyp}{{\Bbb H}}
\newcommand{\J}{{\Bbb J}}
\newcommand{\Ll}{{\Bbb L}}
\renewcommand{\L}{{\Bbb L}}
\newcommand{\M}{{\Bbb M}}
\newcommand{\N}{{\Bbb N}}
\newcommand{\Q}{{\Bbb Q}}
\newcommand{\R}{{\Bbb R}}
\newcommand{\T}{{\Bbb T}}
\newcommand{\V}{{\Bbb V}}
\newcommand{\U}{{\Bbb U}}
\newcommand{\W}{{\Bbb W}}
\newcommand{\X}{{\Bbb X}}
\newcommand{\Z}{{\Bbb Z}}

\newcommand{\tT}{{\mathrm{T}}}
\newcommand{\tD}{{D}}
\newcommand{\hyp}{{\mathrm{hyp}}}
\newcommand{\cusp}{{\mathrm{cusp}}}

\newcommand{\fix}{{b}}
\newcommand{\cxfix}{{\xi}}
\newcommand{\LINV}{L_{\rm inv}}
\newcommand{\LLINV}{{\mathcal L}_{\rm inv}}
\newcommand{\f}{{\bf f}}
\newcommand{\g}{{\bf g}}
\newcommand{\h}{{\bf h}}
\renewcommand{\i}{{\bar i}}
\renewcommand{\j}{{\bar j}}

\newcommand{\geoD}{{\boldsymbol{\gamma}}}

\newcommand{\Bf}{{\mathbf{f}}}
\newcommand{\Bg}{{\mathbf{g}}}
\newcommand{\Bh}{{\mathbf{h}}}
\newcommand{\Bi}{{\mathbf{i}}}
\def\BJ{{\mathbf{J}}}
\def\Bl{{\mathbf{l}}}
\def\Bm{{\mathbf{m}}}
\def\Bn{{\mathbf{n}}}

\def\Bj{{\mathbf{j}}}
\def\BJ{{\mathbf{J}}}
\def\Bphi{{\mathbf{\Phi}}}
\def\Bpsi{{\mathbf{\Psi}}}
\newcommand\Bom{\boldsymbol{\om}}

\newcommand{\BD}{{\boldsymbol{D}}}
\newcommand{\BE}{{\boldsymbol{E}}}
\newcommand{\BF}{{\boldsymbol{F}}}
\newcommand{\BG}{{\mathbf{G}}}
\def\BH{{\mathbf{H}}}
\newcommand{\BI}{{\boldsymbol{I}}}
\newcommand{\BK}{\mathbf{K}}
\newcommand{\BL}{\mathbf{L}}
\def\B0{{\mathbf{0}}}
\newcommand{\BP}{{\boldsymbol{P}}}
\newcommand{\BS}{{\boldsymbol{S}}}
\def\BT{{\mathbf{T}}}
\newcommand{\BW}{{\mathbf{W}}}
\newcommand{\BV}{{\mathbf{V}}}
\newcommand{\BOM}{{\boldsymbol{\Om}}}

\newcommand{\BPi}{{\boldsymbol{\Pi}}}
\def\BUps{{\boldsymbol{\Upsilon}}}
\def\BLa{{\boldsymbol{\La}}}
\def\BGa{{\boldsymbol\Gamma}}
\def\BDe{{\boldsymbol\Delta}}
\def\BUps{{\boldsymbol\Upsilon}}
\def\BThe{{\boldsymbol\Theta}}
\def\BOm{{\boldsymbol \Om}}
\def\BPsi{{\boldsymbol\Psi}}

\def\Baleph{{\boldsymbol\aleph}}

\newcommand{\Bw}{{\mathbf{w}}}
\newcommand{\Bal}{{\boldsymbol{\alpha}}}
\newcommand{\Bde}{{\boldsymbol{\delta}}}
\newcommand{\Bga}{{\boldsymbol{\gamma}}}
\newcommand{\Bsi}{{\boldsymbol{\sigma}}}
\newcommand{\Bla}{{\boldsymbol{\lambda}}}
\def\Be{\mathbf{e}}
\def\Dia{{\Diamond}}

\def\SB{{\boldsymbol{\BB}}}

\def\ext{{\mathrm{ex}}}
\def\mouth{\operatorname{mouth}}
\def\tail{\operatorname{tail}}

\newcommand{\Comb}{{\it Comb}}
\newcommand{\Top}{{\operatorname{Top}}}
\newcommand{\Bottom}{{\operatorname{Bot}}}
\newcommand{\QC}{\mathcal QC}
\newcommand{\Def}{\mathcal Def}
\newcommand{\Teich}{\mathcal Teich}
\newcommand{\PPL}{{\mathcal P}{\mathcal L}}
\newcommand{\Jac}{\operatorname{Jac}}
\newcommand{\Homeo}{\operatorname{Homeo}}
\newcommand{\AC}{\operatorname{AC}}
\newcommand{\Dom}{\operatorname{Dom}}
\newcommand{\ord}{\operatorname{ord}}

\newcommand{\Hol}{{\rm Hol}}

\newcommand{\Aff}{\operatorname{Aff}}
\newcommand{\Euc}{\operatorname{Euc}}
\newcommand{\MobC}{\operatorname{M\ddot{o}b}({\Bbb C}) }
\newcommand{\PSL}{ {\mathcal{PSL}} }
\newcommand{\SL}{ {\mathcal{SL}} }
\newcommand{\CP}{ {\Bbb{CP}}   }

\newcommand{\hf}{{\hat f}}
\newcommand{\hz}{{\hat z}}
\newcommand{\hM}{{\hat M}}

\renewcommand{\lq}{``}
\renewcommand{\rq}{''}

\newcommand{\Ch}{\textrm{Ch}}
% \newcommand{\Ch}{{\textrm{\cyr CH} }}

%&&&&&&&&&&    Content   &&&&&&&&

\catcode`\@=12

\def\Empty{}
\newcommand\oplabel[1]{
  \def\OpArg{#1} \ifx \OpArg\Empty {} \else
  	\label{#1}
  \fi}
		
%%%%%%%%%%%%%%%%%%%%%%%%%%%%%%%%%%%%%%%%%%%%%%%%%%%%%%%%%%%%%%%%%%%%%
% Insert a postscript figure using psfig.
% Usage:	\realfig{label}{filename}{caption}
%
% uses psfig macros: must have \input{psfig} in the preamble to use
% it. 
%%%%%%%%%%%%%%%%%%%%%%%%%%%%%%%%%%%%%%%%%%%%%%%%%%%%%%%%%%%%%%%%%%%%%

\long\def\realfig#1#2#3#4{
\begin{figure}[htbp]
%%%\centerline{\psfig{figure=#3,height=#2}}
\centerline{\psfig{figure=#2,width=#4}}
\caption[#1]{#3}
\oplabel{#1}
\end{figure}}

%&&&&&&&&&&&&       List of figures              &&&&&&&&&
%
%&&&&&&&&&&&&&&&&&&&&&&&&&&&&&&&&&&&&&&&&&&&&&&&&&&&&&

\newcommand{\comm}[1]{}
\newcommand{\comment}[1]{}

\title[MLC at Feigenbaum points]{MLC at Feigenbaum points}

\author {Dzmitry Dudko and Mikhail Lyubich }

\bigskip\bigskip

%\thanks{
%   This work was supported in part by Sloan Research Fellowship
% and NSF grants DMS-8920768 and DMS-9022140.}
%\date{\today}

\begin{abstract} 
  We prove {\em a priori} bounds for Feigenbaum quadratic polynomials,
  i.e., infinitely renormalizable polynomials $f_c: z\mapsto z^2+c$  of bounded type.
  It implies local connectivity of the corresponding Julia sets
  $J(f_c)$ and  MLC (local connectivity of the Mandelbrot set $\Mandel$) at the
  corresponding parameters~$c$. It also yields the scaling
  Universality, dynamical and parameter, 
  for the corresponding combinatorics.  The MLC Conjecture was open for the most
  classical period-doubling Feigenbaum parameter  as well as for the complex tripling renormalizations.
  Universality for the latter was conjectured by Goldberg-Khanin-Sinai
  in the early 1980s.
\end{abstract}

\setcounter{tocdepth}{1}
 
\maketitle
\tableofcontents

\section{Introduction}
\subsection{Brief history} The MLC Conjecture on the local connectivity of the Mandelbrot set
$\Mandel$,
put forward by Douady and Hubbard in the mid 1980s \cite{DH:Orsay},  is the central problem
of contemporary Holomorphic Dynamics. It would imply a precise
topological model for $\Mandel$ and
the Fatou Conjecture on the density of hyperbolic maps in $\Mandel$,
and it is intimately related to the Mostow-Thurston Rigidity phenomenon
in 3D Hyperbolic Geometry.

Around 1990 Yoccoz proved MLC at any parameter $c\in \Mandel$ that is
not infinitely renormalizable (in the quadratic-like sense)
thus linking  the problem tightly to the Quadratic-like
Renormalization Theory
(see \cite{H,M}).
First advances in this direction appeared in \cite{puzzle}, where MLC
was established for infinitely renormalizable parameters of  {\em high type}
and the general problem was reduced, under some circumstances,  to the problem of
{\em a priori} bounds.

Quadratic-like renormalization appears in two flavors, {\em primitive}
and {\em satellite}. The above results were concerned with the
primitive case. Next breakthrough in this direction appeared in the
work of Jeremy Kahn \cite{K} who established {\em a priori} bounds,
and hence MLC, for all infinitely renormalizable parameters of  {\em bounded
primitive} type (using the Covering Lemma \cite{covering lemma}).
It followed up with the work \cite{decorations,molecules}
handling the {\em definitely primitive case}. 

It took more than  10 years to prove MLC at  some infinitely renormalizable parameters of
{\em bounded satellite} type \cite {DL} (based upon the Pacman
Renormalization Theory developed in \cite{DLS}).  However, the most
prominent parameter, the {\em period-doubling Feigenbaum point} corresponding to the
cascade of doubling renormalizations (see Figure~\ref{Fig:FeigPar}) was not covered by this result.
In this paper we are filling in this gap by proving MLC at the Feigenbaum
point, and in fact, at {\em all infinitely renormalizable parameters of bounded type} (for which we will still preserve the same name). 

\begin{rem} The universality of the period-doubling cascade was discovered in the mid 1970s by Feigenbaum~\cite{F1,F2} (the parameter part, see Figure~\ref{Fig:FeigPar}) and independently by Coullet and Tresser~\cite{TC,CT} (the dynamical part). It is intimately related to the renormalization phenomenon in the Quantum Field Theory and Statistical Mechanics; the discovery opened up a new universality paradigm in Dynamical Systems.
\end{rem}

\begin{figure}
\begin{tikzpicture}%background rectangle/.style={fill=green}]

\filldraw[green] (-3,-4) -- (6.95,-4) --(6.95,4.4) --(-3,4.4);
\begin{scope}[shift={(1,2)}]
  \node[scale=0.95] at(0,0){ \includegraphics[width=5cm]{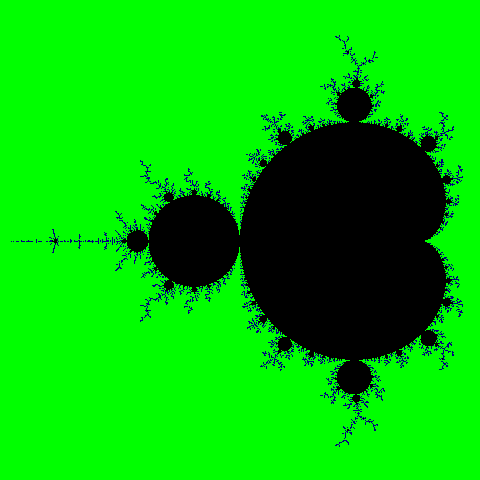}};
  
\draw[scale=0.95][dashed,red, line width=0.3mm] (-1.2,0) circle (0.4cm);
  \coordinate  (a1) at (-1.2,-0.352);
  
\end{scope}

\begin{scope}[shift={(-3,-4)},scale=0.88,every node/.style={scale=0.88}, scale=2.25]
 \node at (1,1) {\includegraphics[width=4.5cm]{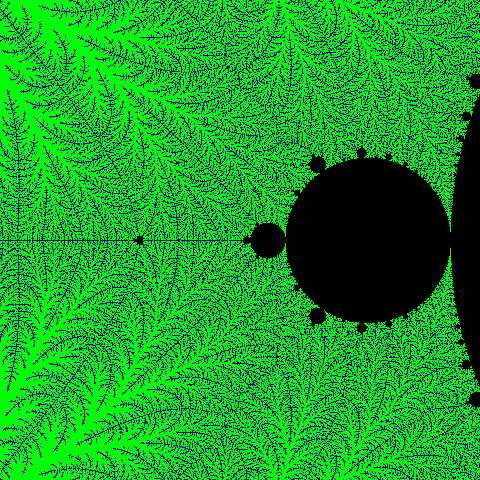}};
 % \draw[red, line width=0.5mm] (0, 0) --(0, 2) --(2, 2) --(2, 0)--(0,0);
  
  \draw[red, dashed, line width=0.3mm] (a1)--(1,2); 

  \draw[red, dashed, line width=0.3mm,shift={(0.05,0)}] (0.8, 0.8) --(0.8, 1.2) --(1.2, 1.2) --(1.2,0.8)--(0.8,0.8);

\coordinate  (w1) at (1.25,1.15);

\end{scope}

\begin{scope}[shift={(3,-4)},scale=0.88,every node/.style={scale=0.88},scale=2.25]
  \draw[red, dashed,line width=0.3mm] (w1)--(0,1.5); 
 \node at (1,1) {\includegraphics[width=4.5cm]{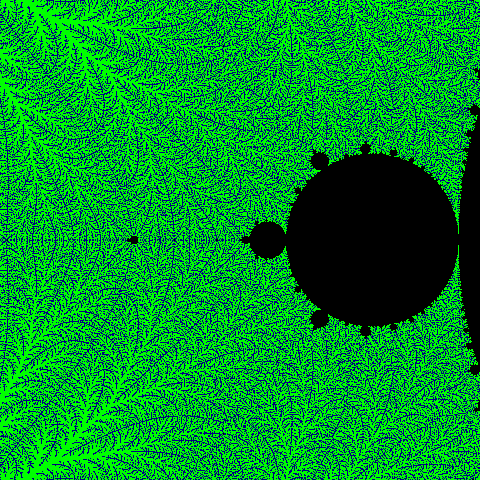}};
 % \draw[red, line width=0.5mm] (0, 0) --(0, 2) --(2, 2) --(2, 0)--(0,0);
\end{scope}

 \end{tikzpicture}
 \caption{Self-similarity of the Mandelbrot set at the period-doubling Feigenbaum parameter: $\MM$ scales almost linearly under $w\mapsto 4.6692\dots w+o(|w|^{1+\varepsilon})$.}
 \label{Fig:FeigPar}
\end{figure}
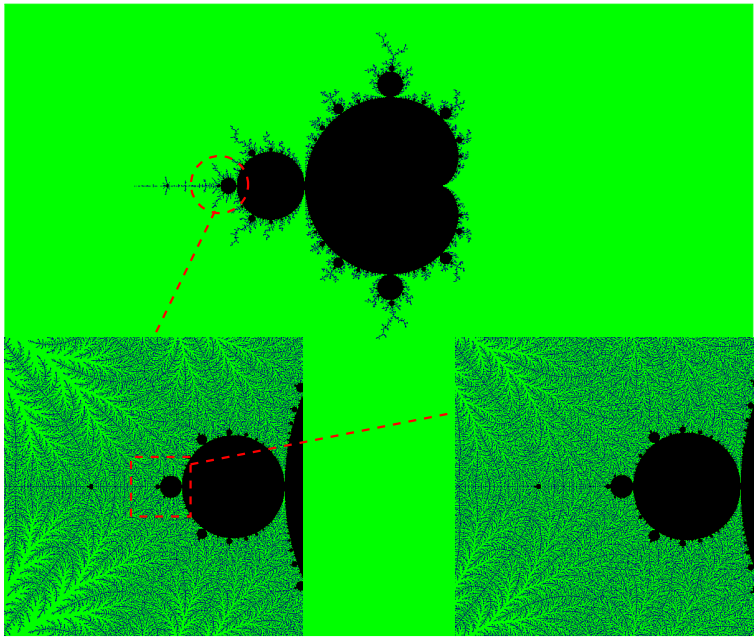

\subsection{Statement of the  main result and its consequences} A {\em Feigenbaum map} is an infinitely renormalizable  quadratic-like map $f: U\ra V$ with bounded combinatorics,
i.e., all renormalization periods of $f$ are bounded by some $\bar p$
(see \S \ref{sss:infin renorm} for precise definitions).
One says that such an $f$ has {\em a priori} bounds if
its quadratic-like renormalizations $\RR^n f : U^n\ra V^n$ can be selected so that
\[\mod (V^n\sm U^n)\geq \mu>0.\] The bounds are called {\em beau} if $\mu$ depends only on the
combinatorial bound $\bar p$ as long as $n$ is big enough (depending on $\mod (V\sm U)$).

\begin{figure}
\begin{tikzpicture}
\filldraw[green] (-3,-0.4) -- (6.95,-0.4) --(6.95,7.96) --(-3,7.96);

\begin{scope}[shift={(1,2)}]
  \node[scale=0.95] at(0,0){ \includegraphics[width=5cm]{1.png}};

  \begin{scope}[shift={(2.45,1.45)}]
\draw[scale=0.95][dashed,red, line width=0.15mm] (-1.2,0) circle (0.2cm);
  \coordinate  (a1) at (-1.21,0.17);
  
\end{scope}  
  
\end{scope}

\begin{scope}[shift={(-3,4)},scale=0.88,every node/.style={scale=0.88}, scale=2.25]
 \node at (1,1) {\includegraphics[width=4.5cm]{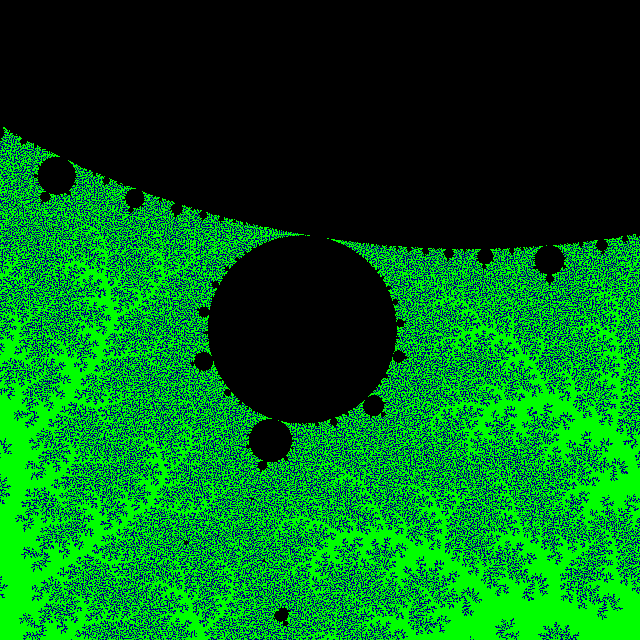}};
  %\draw[red, line width=0.5mm] (0, 0) --(0, 2) --(2, 2) --(2, 0)--(0,0);
  
  \draw[red,dashed, line width=0.3mm] (a1)--(2,0.3);

\begin{scope}[shift={(-0.30,-0.12)}]
  \draw[red, dashed, line width=0.3mm,shift={(-0.05,0.02)}] (0.8, 0.8) --(0.8, 1.12) --(1.12, 1.12) --(1.12,0.8)--(0.8,0.8);
\coordinate  (w1) at (1.08,1.08);
\end{scope}

\end{scope}

\begin{scope}[shift={(3,4)},scale=0.88,every node/.style={scale=0.88},scale=2.25]
  \draw[red, dashed,line width=0.3mm] (w1)--(0,1.5); 
 \node at (1,1) {\includegraphics[width=4.5cm]{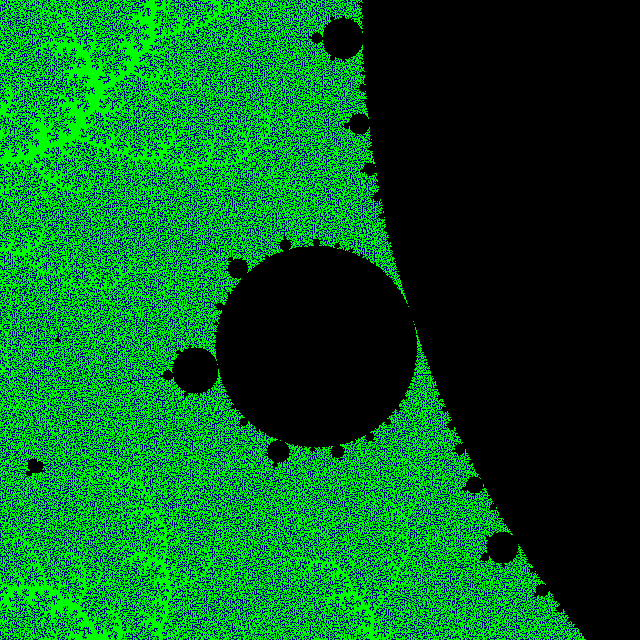}};
 % \draw[red, line width=0.5mm] (0, 0) --(0, 2) --(2, 2) --(2, 0)--(0,0);
\end{scope}

 \end{tikzpicture}
 \caption{Illustration to the Goldberg-Khanin-Sinai Conjecture: the Mandelbrot set scales almost linearly at the parameter of the period-tripling bifurcation. (The first zoom magnifies a deep renormalization level.)}
 \label{Fig:conj:GSK}
\end{figure}

\begin{mthm}
\label{mainthm:A}
Any Feigenbaum quadratic-like map has a priori beau bounds.
\end{mthm}

Together with the Rigidity Theorem of \cite{puzzle}, it implies MLC at the corresponding parameters:
 
\begin{mthm}
\label{mainthm:B}
The Mandelbrot set is locally connected at any Feigenbaum parameter. 
\end{mthm}

The {\em Renormalization Conjecture} asserts that for any combinatorial bound  $\bar p$ the renormalization transformation $\RR$
has a renormalization horseshoe $\AAA_{\bar p}$ of Feigenbaum maps of type  $\bar p$ on which it acts hyperbolically
with one-dimensional unstable foliation. Together with \cite{L:universality} (see also \cite{S,McM2,AL}) we obtain:

\begin{mthm}
\label{mainthm:C}
For any combinatorial bound $\bar p$,
 the Renormalization Conjecture is valid in the space of quadratic-like maps. 
\end{mthm}

Let us consider the set $\II_{\bar p}$  of Feigenbaum parameters $c\in \Mandel$ of type $\bar p$.
Theorem~\ref{mainthm:B} implies that each $c\in \II_{\bar p}$ is the intersection of the nest $\MM^n(c)$ of little
$\MM$-copies corresponding to the  $n$-fold renormalizations of $f_c$.  
Let us say that Mandelbrot set is \emph{self-similar} over $\II_{\bar p}$ if
the straightening map $\chi:\MM^{n+1} (c) \ra \MM^n(\chi(c)) $  (see \S \ref{renorm sec})
is $C^{1+\de}$-conformal at any  $c\in \II_{\bar p}$
(with uniform constants).

The Renormalization Conjecture (Theorem~\ref{mainthm:C}) implies: 

\newtheorem*{Theorem D}{Theorem D}
\begin{mthm}
  For any $\bar p$, the Mandelbrot set is self-similar over $\II_{\bar p}$.
\end{mthm}

For the {\em satellite triplings}, the last two  theorems confirm the
Goldberg-Sinai-Khanin Conjecture \cite{GSK} from the 1980s, see Figure~\ref{Fig:conj:GSK}.

\subsection{Rough outline of the proof} The historical developments of ideas behind the proof are summarized in~\cite{L21}. 
 For a ql map $f\colon U\to V$, its width is
\[\Width(f) = \Width(V\setminus \filled_f)=\frac{1}{\mod (V\setminus \filled_f)} ,\]
 where $\filled_f$ is the filled Julia set of $f$. Our goal is to bound $\Width(f)$ from above. We will first review the main ideas for the bounded primitive case, then we will discuss how to complete the argument in the bounded satellite case.

\begin{figure}[t!]
\[\begin{tikzpicture}[]

\draw[fill=red,opacity=0, fill opacity=0.2]
(0.2,-0.2)--(0.2,0.2) -- (5.8,0.2)--(5.8,-0.2);

\node[red,] at (3,0.5) {$\FamG_1$};

\draw[fill=blue,opacity=0, fill opacity=0.2]
(6.2,-0.2)--(6.2,0.2) -- (8.8,0.2)--(8.8,-0.2);

\node[blue] at (7.5,0.5) {$\FamG_2$};

\begin{scope}[rotate =0]
\draw[line width=0.8mm] (-0.5,0)--(0.5,0); 
\draw[line width=0.8mm] (-0.2,-0.2)--(-0.2,0.2);
\draw[line width=0.8mm] (0.2,-0.2)--(0.2,0.2);
  \end{scope}

\begin{scope}[shift={(6,0)}]
\draw[line width=0.8mm] (-0.5,0)--(0.5,0); 
\draw[line width=0.8mm] (-0.2,-0.2)--(-0.2,0.2);
\draw[line width=0.8mm] (0.2,-0.2)--(0.2,0.2);
  \end{scope}
  
  \begin{scope}[shift={(9,0)}]
\draw[line width=0.8mm] (-0.5,0)--(0.5,0); 
\draw[line width=0.8mm] (-0.2,-0.2)--(-0.2,0.2);
\draw[line width=0.8mm] (0.2,-0.2)--(0.2,0.2);
  \end{scope}

\begin{scope}[shift={(0,-3)}]

\draw[fill=blue,opacity=0, fill opacity=0.2]
(0.2,-0.2)--(0.2,0.2) -- (2.8,0.2)--(2.8,-0.2);

\draw[fill=red,opacity=0, fill opacity=0.2]
(6.2,-0.2)--(6.2,0.2) -- (8.8,0.2)--(8.8,-0.2);

\draw[fill=red,opacity=0, fill opacity=0.2]
(3.2,-0.2)--(3.2,0.2) -- (5.8,0.2)--(5.8,-0.2);

\draw[line width=0.3mm] (4.1,0.5) edge[bend left=10,->] (3,2.5); 

\node[left,red] at (4.8,0.5) {$\FamG'_1$};

\draw[line width=0.3mm] (7.5,0.5) edge[bend left=15,->] (3.5,2.5); 

\node[right,red] at (7.7,0.5) {$\FamG''_1$};

\draw[line width=0.3mm] (1.5,0.5) edge[,->] (7.5,2.5); 
\node[blue,left] at (1.3,0.5) {$\FamG'_2$};

\begin{scope}[rotate =0]
\draw[line width=0.8mm] (-0.5,0)--(0.5,0); 
\draw[line width=0.8mm] (-0.2,-0.2)--(-0.2,0.2);
\draw[line width=0.8mm] (0.2,-0.2)--(0.2,0.2);
  \end{scope}

\begin{scope}[shift={(3,0)},gray]
\draw[line width=0.8mm] (-0.5,0)--(0.5,0); 
\draw[line width=0.8mm] (-0.2,-0.2)--(-0.2,0.2);
\draw[line width=0.8mm] (0.2,-0.2)--(0.2,0.2);
  \end{scope}

\begin{scope}[shift={(6,0)}]
\draw[line width=0.8mm] (-0.5,0)--(0.5,0); 
\draw[line width=0.8mm] (-0.2,-0.2)--(-0.2,0.2);
\draw[line width=0.8mm] (0.2,-0.2)--(0.2,0.2);
  \end{scope}
  
  \begin{scope}[shift={(9,0)}]
\draw[line width=0.8mm] (-0.5,0)--(0.5,0); 
\draw[line width=0.8mm] (-0.2,-0.2)--(-0.2,0.2);
\draw[line width=0.8mm] (0.2,-0.2)--(0.2,0.2);
  \end{scope}

\end{scope}
\end{tikzpicture}\]

\caption{Primitive pull-off in the airplane combinatorics. If $\Width_\loc^\ver(\filled_j) \ll \Width_\loc(\filled_j)$, then most curves of the horizontal rectangles $\FamG_1,\FamG_2$ overflow their preimages $\FamG'_1,\FamG'_2,\FamG''_1$ as illustrated. Such a configuration is impossible by the Gr\"otzsch inequality;~i.e., \eqref{eq:outline:3} holds. Peripheral preimages are omitted.}
\label{Fg:Pull off:Air Comb}
\end{figure}
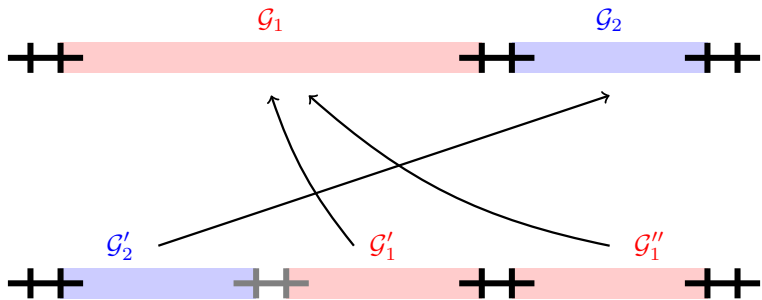
\subsubsection{Pull-off in the primitive case \cite{K,covering lemma}}
\label{sss:intro:PrimPullOff}
Consider the $n$th renormalization $f_n=\RR^n f$ of $f=f_0$, and let $\filled_0, \filled_1,\dots ,\filled_{p-1}$ be the periodic cycle of little filled Julia sets of $f_n$ in the dynamical plane of $f_0\colon U\to V$. Assuming that the renormalization is primitive, all $\filled_i$ are pairwise disjoint.

Consider the thin-thick decomposition of $V'\coloneqq V\setminus \bigcup_i \filled_i\,$: there are finitely many pairwise disjoint rectangles $\FamG_i$ between components of $\partial V'$ such that, up to $O_p(1)$ of width, the $\FamG_i$ contain all non-peripheral curves in $V'$ with endpoints in $\partial V'$, see Figure~\ref{Fg:TTD} for an illustration of a thin-thick decomposition. A rectangle $\FamG_i$ is called 
\begin{itemize}
\item \emph{vertical} if it connects the outer boundary $\partial^\out V'\coloneqq \partial V$ of $V'$ to one of its inner boundary components;
\item \emph{horizontal} if it connects two inner (i.e., non-outer) boundary components of $V'$.
\end{itemize}

Below we assume that $\Width(f_n)\gg_p 1$. We have:
\begin{equation}
\label{eq:outline:1}
\Width(f)\ge \sum_{\text{ vertical }\FamG_i} \Width(\FamG_i) 
\end{equation} 
and, for every $\filled_j$:  
\begin{equation}
\label{eq:outline:2}
\Width(f_n) \asymp  \Width_\loc(\filled_j) \  \coloneqq
\sum_{\substack{ \FamG_i \text{ adjacent}\\
        \text{to }\filled_j }}\Width(\FamG_i) +O_p(1), 
\end{equation}
i.e.~the sum in~\eqref{eq:outline:2} is taken over all rectangles
$\FamG_i$ adjacent to $\filled_j$.
(A ``$\psi$-ql renormalization''  is used to
achieve~\eqref{eq:outline:2}, see \S \ref{psi-renorm},~\ref{sss:WADs}, and~\eqref{eq:Width is Loc WAD}.) 
 Set also
\begin{equation}
\Width_\loc^\ver(\filled_j)\coloneqq \sum_{\substack{ \text{vertical }\FamG_i \\ \text{adjacent to }\filled_j   }}\Width(\FamG_i).
\end{equation}

A fundamental fact is that the horizontal $\FamG_i$ are eventually
(after several restrictions of the domain)
aligned with the Hubbard tree $T_f$ of $f$ as shown on~Figure~\ref{Fg:Pull off:Air Comb}.
Applying the Gr\"otzsch inequality to the horizontal rectangles and their pullbacks
(see the caption of Figure~\ref{Fg:Pull off:Air Comb}),
we experience a definite loss of the horizontal weight in favor of the vertical one.
The created definite vertical weight can then be pushed forward by means of the
Covering Lemma \cite{covering lemma}. Since all the local weights are comparable\footnote{This observation goes back to McMullen
  \cite[Theorem 9.3]{McM1}.}, 
we obtain:
\begin{equation}
\label{eq:outline:3}
\Width_\loc^\ver(\filled_j) \asymp \Width_\loc(\filled_j) \sp\sp\sp\text{ for every }\filled_j,
\end{equation}
where ``$\asymp$'' is independent of $p$.

We stress that a key combinatorial ingredient used for~\eqref{eq:outline:3} is the absence of 
periodic horizontal rectangles; i.e. a rectangle that has an iterated lift homotopic to itself rel small Julia sets, see~\S\ref{ss:intro:SatCase},~\S\ref{sss:InvArcDiagr}, \S\ref{sss:PeriodRect}.

Combining ~\eqref{eq:outline:1}, \eqref{eq:outline:3}, \eqref{eq:outline:2},  and assuming $p\gg 1$, we obtain:
\begin{equation}
\label{eq:intro:comp}
{
\begin{array}{c}
\Width(f)\ge \sum_{\text{ vertical }\FamG_i} \Width(\FamG_i)\ = \ \sum_{j=1}^p  \Width_\loc^\ver(\filled_j)\ \asymp \\ \\  \sum_{j=1}^p  \Width_\loc(\filled_j)\ \asymp p\  \Width(f_n)\ \gg\  2 \Width(f_n);
\end{array}
}\end{equation}
i.e:
\begin{equation}
\label{eq:W f vs W f_n}
\Width(f) \ge 2 \Width(f_n)
\end{equation}
implying  a priori bounds in the primitive case.
(For instance, by the {\em Record Argument}: by selecting ``record levels'', on which the degeneration   exceeds all the preceding ones, we immediately arrive at a contradiction.)

\begin{figure}[t!]
\[\begin{tikzpicture}[scale=0.8]

\draw[fill=red,opacity=0, fill opacity=0.2]
(0.2,-0.2)--(0.2,0.2) -- (2.8,0.2)--(2.8,-0.2);

\node[red,] at (1.5,0.5) {$\RR_0$};

\draw[fill=blue,opacity=0, fill opacity=0.2]
(3.2,-0.2)--(3.2,0.2) -- (10.8,0.2)--(10.8,-0.2);

\node[blue] at (7,0.5) {$\FamG$}; 

\draw[fill=red,opacity=0, fill opacity=0.2]
(11.2,-0.2)--(11.2,0.2) -- (13.8,0.2)--(13.8,-0.2);

\node[red,] at (12.5,0.5) {$\RR_1$};

\begin{scope}[rotate =0]
\draw[line width=0.8mm] (-0.5,0)--(0.5,0); 
\draw[line width=0.8mm] (-0.2,-0.2)--(-0.2,0.2);
\draw[line width=0.8mm] (0.2,-0.2)--(0.2,0.2);

\node[above] at (-0.2,0.2) {$\bush_0$};
  \end{scope}

\begin{scope}[shift={(3,0)}]
\draw[line width=0.8mm] (-0.5,0)--(0.5,0); 
\draw[line width=0.8mm] (-0.2,-0.2)--(-0.2,0.2);
\draw[line width=0.8mm] (0.2,-0.2)--(0.2,0.2);

\node[above] at (-0.,0.2) {$\bush_1$};
  \end{scope}

\begin{scope}[shift={(11,0)}]
\draw[line width=0.8mm] (-0.5,0)--(0.5,0); 
\draw[line width=0.8mm] (-0.2,-0.2)--(-0.2,0.2);
\draw[line width=0.8mm] (0.2,-0.2)--(0.2,0.2);
\node[above] at (-0.,0.2) {$\bush_2$};
  \end{scope}
  
  \begin{scope}[shift={(14,0)}]
\draw[line width=0.8mm] (-0.5,0)--(0.5,0); 
\draw[line width=0.8mm] (-0.2,-0.2)--(-0.2,0.2);
\draw[line width=0.8mm] (0.2,-0.2)--(0.2,0.2);
\node[above] at (0.2,0.2) {$\bush_3$};
  \end{scope}

\begin{scope}[shift={(0,-3)}]

\draw[fill=red,opacity=0, fill opacity=0.2]
(0.2,-0.2)--(0.2,0.2) -- (2.4,0.2)--(2.4,-0.2);

\node[red,] at (1.5,0.5) {$\RR'_0$}; 

\draw (2,0.5) edge[line width=0.2mm,->,bend right]  node[right]{ $f^2$}(2, 2.5);

\begin{scope}[shift={(2.5,0)}]
\draw[fill=red,opacity=0, fill opacity=0.2]
(0.7,-0.2)--(0.7,0.2) -- (2.3,0.2)--(2.3,-0.2);

\node[red,] at (1.5,0.5) {$\RR''_0$};

\end{scope}

\draw[fill=blue,opacity=0, fill opacity=0.2]
(5.2,-0.2)--(5.2,0.2) -- (8.8,0.2)--(8.8,-0.2);

\node[blue] at (7,0.5) {$\FamG'$}; 

\draw (7.7,0.5) edge[line width=0.2mm,->,bend right]  node[right]{ $f^2$}(7.7, 2.5);

\draw[fill=red,opacity=0, fill opacity=0.2]
(11.6,-0.2)--(11.6,0.2) -- (13.8,0.2)--(13.8,-0.2);

\node[red,] at (12.5,0.5) {$\RR'_1$}; 

\draw (13,0.5) edge[line width=0.2mm,->,bend right]  node[right]{ $f^2$}(13, 2.5);

\begin{scope}[shift={(-2.5,0)}]
\draw[fill=red,opacity=0, fill opacity=0.2]
(11.7,-0.2)--(11.7,0.2) -- (13.3,0.2)--(13.3,-0.2);

\node[red,] at (12.5,0.5) {$\RR''_1$}; 
\end{scope}

\begin{scope}[rotate =0]
\draw[line width=0.8mm] (-0.5,0)--(0.5,0); 
\draw[line width=0.8mm] (-0.2,-0.2)--(-0.2,0.2);
\draw[line width=0.8mm] (0.2,-0.2)--(0.2,0.2);
%\draw[line width=0.9mm,gray] (0.5,0) -- (0.8,0); 
%\draw[line width=0.9mm,gray] (0.6,-0.2) -- (0.6,0.2); 
  \end{scope}

\begin{scope}[shift={(3,0)}]
\draw[line width=0.9mm,gray] (-0.6,-0.2) -- (-0.6,0.2); 
\draw[line width=0.9mm,gray] (-0.9,0) -- (-0.5,0); 
\draw[line width=0.8mm] (-0.5,0)--(0.5,0); 
\draw[line width=0.8mm] (-0.2,-0.2)--(-0.2,0.2);
\draw[line width=0.8mm] (0.2,-0.2)--(0.2,0.2);
  \end{scope}

\begin{scope}[gray,shift={(5,0)}]
\draw[line width=0.8mm] (-0.5,0)--(0.5,0); 
\draw[line width=0.8mm] (-0.2,-0.2)--(-0.2,0.2);
\draw[line width=0.8mm] (0.2,-0.2)--(0.2,0.2);
  \end{scope}

\begin{scope}[gray,shift={(9,0)}]
\draw[line width=0.8mm] (-0.5,0)--(0.5,0); 
\draw[line width=0.8mm] (-0.2,-0.2)--(-0.2,0.2);
\draw[line width=0.8mm] (0.2,-0.2)--(0.2,0.2);
  \end{scope}

\begin{scope}[shift={(11,0)}]
\draw[line width=0.8mm] (-0.5,0)--(0.5,0); 
\draw[line width=0.8mm] (-0.2,-0.2)--(-0.2,0.2);
\draw[line width=0.8mm] (0.2,-0.2)--(0.2,0.2);
\draw[line width=0.9mm,gray] (0.5,0) -- (0.8,0); 
\draw[line width=0.9mm,gray] (0.6,-0.2) -- (0.6,0.2); 
  \end{scope}
  
  \begin{scope}[shift={(14,0)}]
\draw[line width=0.8mm] (-0.5,0)--(0.5,0); 

\draw[line width=0.8mm] (-0.2,-0.2)--(-0.2,0.2);
\draw[line width=0.8mm] (0.2,-0.2)--(0.2,0.2);

%\draw[line width=0.9mm,gray] (-0.6,-0.2) -- (-0.6,0.2); 
%\draw[line width=0.9mm,gray] (-0.9,0) -- (-0.5,0); 
  \end{scope}

\end{scope}
\end{tikzpicture}\]

\caption{Pull-off in the Feigenbaum combinatorics for non-periodic
  rectangles. Either most of the horizontal degeneration is within periodic rectangles $\RR_0,\RR_1$
  such that a $(1-\delta)$-portion of each $\RR_i$ overflows its periodic lifts $\RR'_i$. Or we have $\Width^\ver_\loc(\bush_j)\succeq_\delta \Width_\loc(\bush_j).$}
\label{Fg:Pull off:Feigen Comb}
\end{figure}
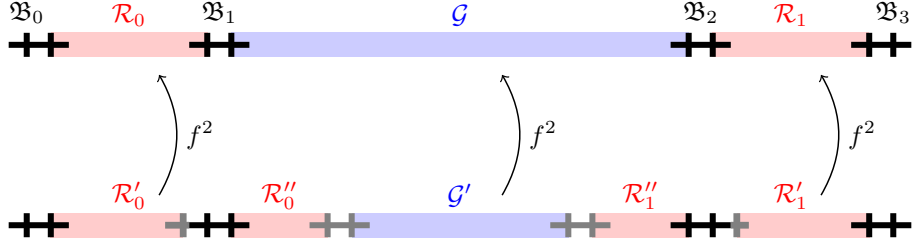

\subsubsection{Satellite case}\label{ss:intro:SatCase} Let us now assume that $f$ is an infinitely renormalizable map with the period-doubling combinatorics; i.e., $f$ is hybrid equivalent to a map in the period-doubling Feigenbaum combinatorial class. The general bounded case is similar.

Let $\bush\coloneqq \filled^1_0\cup \filled^1_1$ be the bouquet consisting of two level $1$ little Julia sets of $f$. Instead of $\Width(f)$, we will measure
\[\Width_\bullet(f)\coloneqq \Width(V\setminus \bush).\]

In the dynamical plane of $f=f_0$, consider the periodic cycle of little bouquets  $\bush_0,\bush_1,\dots, \bush_{p-1}$ (where $p=2^n$) associated with $f_n=\RR^n f$. We enumerate these bouquets linearly from left-to-right. Consider the thick-thin decomposition of $V'\coloneqq V\setminus \bigcup_i \bush_i$. The following properties are established similarly as in the primitive case considered above:
\begin{itemize}
\item Estimates \eqref{eq:outline:1} and~\eqref{eq:outline:2}, where little Julia sets $\filled_i$ are replaced with bouquets $\bush_i$;
\item alignment of the  horizontal rectangles
% in the thin-thick decomposition of $V'$
with the Hubbard tree of $f$ (after several restrictions of the domain).
\end{itemize}

However, unlike in the primitive case~\S\ref{sss:intro:PrimPullOff},
$f$ may have periodic horizontal rectangles $\RR_k$ aligned with the Hubbard tree between $\bush_{2k}$ and $\bush_{2k+1}$ as shown on Figure~\ref{Fg:Pull off:Feigen Comb}. This leads us to the following alternative: either most of the total degeneration is within the $\RR_k$ or a substantial part of the total degeneration is vertical (i.e.~\eqref{eq:outline:3} holds). More precisely, choose a small $\delta>0$. We have: 
\begin{enumerate}[label=\text{(\Roman*)},font=\normalfont,leftmargin=*]
\item \label{C1:intro}either there is a periodic rectangle $\RR_k$ connecting $\bush_{2k},\bush_{2k+1}$ such that a $(1-\delta)$ part of $\RR_k$ overflows its iterated lift $\RR'_k$ under $f^{2^{n-1}}$; 
\item \label{C2:intro} or else we have:
\[
\Width_\loc^\ver(\bush_j) \asymp_\delta \Width_\loc(\bush_j) \sp\sp\sp\text{ for every }\bush_j
\]
\end{enumerate}
If~\ref{C2:intro} holds, then assuming $p\gg_\delta 1$ and $\Width_\bullet(f_n)\gg_{p,\delta} 1$ and repeating the argument of~\eqref{eq:intro:comp}, we obtain 
\begin{enumerate}[label=\text{(\Roman*$'$)},start=2, font=\normalfont,leftmargin=*]
\item \label{C2':intro}\hspace{3cm} $\Width_\bullet(f)>2 \Width_\bullet(f_n)$;
\end{enumerate}
the Record Argument leads to a contradiction in case~\ref{C2':intro}. 

The dichotomy ``Case~\ref{C1:intro} vs Case~\ref{C2':intro}'' is stated as Theorem~\ref{thm:prim pull-off}. This is a refinement of~\cite{K}.

\emph{Consider now Case~\ref{C1:intro}.} As Figure~\ref{Fg:Sat Pull
  off:Feig Comb} illustrates, bouquets $\bush_0,\bush_1$ grow under
lifting, thus the ``combinatorial distance'' between $\widetilde
\bush_0,\widetilde \bush_1$ shrinks. This allows us to detect
laminations $\FamS_i$ (the ``differences'' between $\RR=\RR_k$ and its
lift $\RR'$) that are much wider than $\Width(\RR)\succeq
\Width_\bullet(f_n)$; i.e.~$\Width(\FamS_i)\ge C_\delta \Width(\RR)$
with $C_\delta\to \infty$ as $\delta\to 0$. There are two
possibilities: either a substantial part of $\FamS_i$ hits preperiodic
bouquets associated with $f_{n+1}$ or a substantial part of $\FamS_i$
creates a ``wave'' (see \S \ref{waves sec})  above such a preperiodic
bouquet. In the latter case, the Wave Lemma implies that
$\Width_\bullet(f_{n-1}) \gg \Width_\bullet(f_n)$
and the Record Argument is applicable.

In the former case, we obtain $\Width_\bullet (f_{n+1}) \succeq \Width(\FamS_i) \succeq C_\delta \Width_\bullet(f_n)$; i.e.:
\[\Width_\bullet (f_{n+1}) \ge C'_\delta \Width_\bullet(f_n)\sp\sp\text{ with }\sp C'_\delta\to \infty\text\sp { as }\sp \delta\to 0.\]
 For a sufficiently small $\delta$, this eventually contradicts the Teichm\"uller contraction within hybrid classes~\cite{S:Berkeley} (stated as Proposition~\ref{prop:Teichm contr}):  
\[ \Width_\bullet (f_n) = O\left( \Delta^n\right) \sp\sp\sp\text{ for some } \Delta >1.\]

Finally, we will convert in Theorem~\ref{thm:psi bounds} bounds for $\Width_\bullet(f_m)$ into bounds for $\Width(f_m)$.

\begin{rem} In the paper, we will be measuring the degeneration around ``bushes'' \S\ref{ss:bushes}, which are Hubbard continua enhanced with the periodic cycle of level one little Julia sets. An alternative would be to keep the Julia sets for primitive renormalization and use the satellite flowers in the satellite case. We believe that it is also possible to work with Hubbard continua; however, the Wave Lemma~\S\ref{waves sec} and the transition toward classical QL-bounds will be more subtle.
\end{rem}

\subsection*{Acknowledgments}
The first author was partially supported by the NSF grant DMS $2055532$ and the ERC grant ``HOLOGRAM''. The second author has been partly supported by the NSF, the Hagler and Clay Fellowships, the Institute for Theoretical Studies at ETH (Zurich), and MSRI (Berkeley).

We would like to thank Jeremy Kahn for many stimulating discussions. We are grateful to the referee for a comprehensive reading of the manuscript. A number of our pictures are made with W. Jung's program \emph{Mandel}.

Results of this paper were first announced in 2021 at the \href{https://sites.google.com/g.ucla.edu/quasiworld/}{Quasiworld Seminar} and at the CIRM conference \href{https://www.i2m.univ-amu.fr/events/advancing-bridges-in-complex-dynamics/}{``Advancing Bridges in Complex Dynamics''}, see the talk by the second author~\cite{L21}.

\begin{figure}[t!]
\[\begin{tikzpicture}[xscale=2,yscale=1.5]

\draw[fill=red,opacity=0, fill opacity=0.2]
(0.2,-0.17)--(0.2,0.17) -- (5.8,0.17)--(5.8,-0.17);

\node[red,scale=1.2] at (3,0.5) {$\RR$};

\begin{scope}[rotate =0]
\draw[line width=0.8mm] (-0.5,0)--(0.5,0); 
\draw[line width=0.8mm] (-0.2,-0.2)--(-0.2,0.2);
\draw[line width=0.8mm] (0.2,-0.2)--(0.2,0.2);
\node[above,scale=1.2] at(0,0.2){$\bush_0$};
  \end{scope}

\begin{scope}[shift={(6,0)}]
\draw[line width=0.8mm] (-0.5,0)--(0.5,0); 
\draw[line width=0.8mm] (-0.2,-0.2)--(-0.2,0.2);
\draw[line width=0.8mm] (0.2,-0.2)--(0.2,0.2);
\node[above,scale=1.2] at(0,0.2){$\bush_1$};
  \end{scope}

\begin{scope}[shift={(0,-2)}]

\draw[fill=red,opacity=0, fill opacity=0.2]
(1.1,-0.17)--(1.1,0.17) -- (4.9,0.17)--(4.9,-0.17);

\node[red,scale=1.2] at (3,0.5) {$\RR'$}; 

\draw[fill=orange,opacity=0, fill opacity=0.3]
(1.1,-0.17)--(1.1,0.17) -- (0.6,0.17)--(0.6,-0.17);

\node[orange,scale=1.2] at (0.85,0.5) {$\FamS_1$};

\draw[fill=orange,opacity=0, fill opacity=0.3]
(5.4,-0.17)--(5.4,0.17) -- (4.9,0.17)--(4.9,-0.17);

\node[orange,scale=1.2] at (5.15,0.5) {$\FamS_2$};

\begin{scope}[rotate =0]
\draw[line width=0.8mm] (-0.5,0)--(0.5,0); 
\draw[line width=0.8mm] (-0.2,-0.2)--(-0.2,0.2);
\draw[line width=0.8mm] (0.2,-0.2)--(0.2,0.2);
\draw[line width=0.9mm,gray] (0.5,0) -- (1.4,0); 
\draw[line width=0.9mm,gray] (0.6,-0.2) -- (0.6,0.2); 
\draw[line width=0.9mm,gray] (1.1,-0.2) -- (1.1,0.2); 

\node[above,scale=1.2] at(0,0.2){$\widetilde \bush_0$};
 \end{scope}

\begin{scope}[shift={(6,0)}]
\draw[line width=0.9mm,gray] (-1.1,-0.2) -- (-1.1,0.2); 
\draw[line width=0.9mm,gray] (-0.6,-0.2) -- (-0.6,0.2); 
\draw[line width=0.9mm,gray] (-1.4,0) -- (-0.5,0); 
\draw[line width=0.8mm] (-0.5,0)--(0.5,0); 
\draw[line width=0.8mm] (-0.2,-0.2)--(-0.2,0.2);
\draw[line width=0.8mm] (0.2,-0.2)--(0.2,0.2);
\node[above,scale=1.2] at(0,0.2){$\widetilde \bush_1$};
  \end{scope}

\end{scope}
\end{tikzpicture}\]

\caption{Pull-off in the Feigenbaum combinatorics for periodic rectangles. Assume that a $(1-\delta)$ part of a periodic rectangle $\RR$ overflows its iterated lift $\RR'$. Since $\RR$ also overflows $\FamS_1,\FamS_2$ (that are disjoint from $\RR'$), we obtain from the Gr\"otzsch inequality that $\Width(\FamS_i)> C_\delta\Width(\RR)$ with $C_\delta\to \infty$ as $\delta\to 0$.}
\label{Fg:Sat Pull off:Feig Comb}
\end{figure}
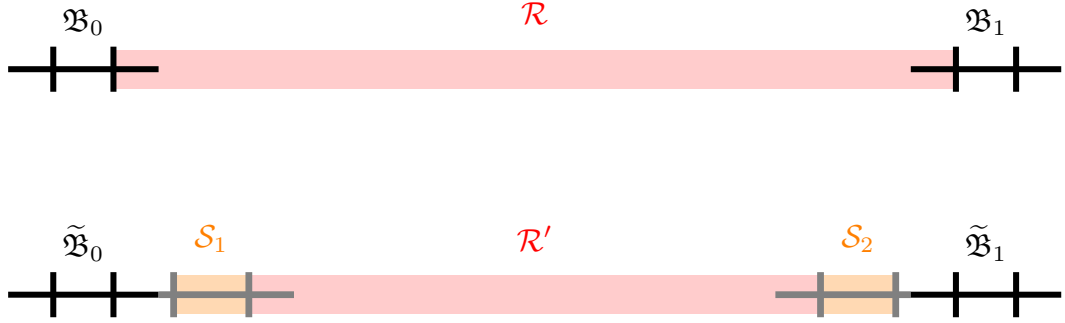

 \section{QL renormalization, bushes, and invariant arc diagrams} \label{ql maps}
The reader can find most of the needed background material in the book \cite{L:book}. Let us briefly outline the context of this section with emphasis on the notation.

\subsection{Outline} Basics aspects of quadratic-like maps are summarized in~\S\ref{ss:QL maps}. We will usually denote by
\begin{itemize}
\item $f\colon X\to Y$ a ql map,
\item $\filled_f$ its (filled) Julia set, 
\item $\filled^{[n]}_i\subset \filled_f$ its level $n$ little Julia sets,
\item $\bfilled^{[n]}\coloneqq \bigcup_i \filled^{[n]}_i$ the union of periodic level $n$ little Julia sets, see~\eqref{eq:bfilled},
\item $\Hub_f\subset \filled$ the Hubbard continuum of $f$, see~\S\ref{sss:CombHT},
\item $\bush_f\coloneqq \Hub_f \cup \bfilled^{[1]}$ the \emph{bush} of $f$, see \eqref{eq:dfn:bush},
\item $\bush_f^{(m)}\coloneqq f^{-m}(\bush_f)$ the little bush of height $m$,
\item $\bush^{[n]}_i \equiv \bush^{[n],(0)}_i\subset \bush^{[n],(m)}_i\subset \filled^{[n]}_i$ the associated little bushes,
\item $\bbush^{[n]}\coloneqq \bigcup_i \bush^{[n]}_i\subset \bfilled^{[n]}$ the union of periodic level $n$ little bushes,
\item $\bbush^{[n],(m)}\coloneqq f^{-m} \left(\bbush^{[n]}\right)$ the union of level $n$ and height $m$ little periodic and preperiodic bushes,
\item $\bbush^{[n],(m)}_\per\coloneqq\bbush^{[n],(m)} \cap \bfilled^{[n]}$ the union of periodic level $n$ and height $m$ little bushes.
\end{itemize}
If $\bush_f =\bfilled^{[1]}$, i.e., $\Hub_f \subset \bfilled^{[1]}$, then we will also refer to $\bush_f $ as the \emph{satellite flower} of $f$. Preperiodic little Julia sets and bushes are defined accordingly, \S\ref{sss:LittleBushes}. We will usually assume that
\begin{enumerate}[label=\text{(\Alph*)},start=15, font=\normalfont,leftmargin=*]
\item \label{CrtPoint} $\filled^{[n]}_0$ contains the critical point of $f$.
\end{enumerate}
Periodic little Julia sets will be usually labeled dynamically:
\[f\left(\filled^{[n]}_i\right)=\filled^{[n]}_{i+1},\sp\sp\sp i,\ i+1\in \Z/{p_n}\simeq\{0,1,\dots, p_n-1\}.\]
However, in the satellite combinatorics, we will often label the $\filled^{[n]}_i$ with respect to their cyclic order in the satellite flower; see, for instance, Lemma~\ref{lem:per arcs}. Preperiodic little Julia sets $\filled^{[n]}_j$ will be labeled with some subindex.  The labeling of little bushes is naturally inherited from the labeling of little Julia sets: $ \bush^{[n],(m)}_i\subset \filled^{[n]}_i$. We will always indicate when the cyclic order is used. And we will also mention when a preperiodic object appears.

 The Teichm\"uller contraction of ql renormalization will be established in~\S\ref{ss:Teichm contra}.

In~Lemma~\ref{lem:Alignment}, we will show that invariant horizontal arc diagrams are aligned with the Hubbard dendrite. Figures~\ref{Fig:exm Inv Arc Diagr} and~\ref{Fig:PeriodicArcs} demonstrate new subtleties arising in the satellite case. In particular, there are genuine periodic arcs, described in Lemma~\ref{lem:per arcs}; they encode almost periodic rectangles (see~\S\ref{sss:PeriodRect} and Figure~\ref{Fg:Pull off:Feigen Comb}) that will have a separate treatment in Section~\ref{s:Pulloff for per rectangles}.

\subsection{Quadratic-like maps}\label{ss:QL maps}
A {\it quadratic-like map}  $f: X\ra Y$,
which will also be abbreviated as a {\it ql map}, 
  is a holomorphic double branched covering between two Jordan disks
$X \Subset Y\subset \C$.  
It has a single critical point that we usually put at the origin $0$. The 
annulus $A= Y\sm \ol X$ is called the {\it fundamental annulus} of $f\colon X\to Y$. 
We let $\mod f : = \mod A$. 
The {\it filled Julia set} $\filled_f$
is the set of non-escaping points:
$$
  \filled \equiv  \filled_f \equiv \filled (f)\coloneqq \{ z:\,  f^n z\in X, \ n=0,1,2,\dots \}.
$$
Its boundary is called the {\it Julia set} $\Julia \equiv \Julia(f)$. The (filled) Julia set is either connected or Cantor,
depending on whether the critical point is non-escaping (i.e., $0\in \filled (f)$) or  otherwise.
In this paper, filled Julia sets will play a bigger role than the Julia sets, and we will often skip 
the adjective ``filled'' when it is clear (e.g., from the notation) what we mean.

 We say that two quadratic-like maps with connected Julia set represent
the same {\it germ } if they have the same filled Julia set and
coincide in some neighborhood of it.
For a ql germ $f$, we define 
\begin{equation}
\label{eq:mod(f)} \mod f \coloneqq \sup \mod (Y'\setminus X'), 
\end{equation}
where the supremum is taken over all $f\colon X'\to Y'$ representing the ql germ of $f$. 

Two quadratic-like maps $f: X\ra Y$ and $\tl f: \tl X\ra \tl Y$ are called {\it hybrid conjugate}
% {\it hybrid equivalent}
if they are conjugate by a quasiconformal map $h: (Y,X) \ra (\tl Y, \tl X)$ such that
$\dibar h= 0$ a.e. on $\filled$.  

A simplest example of a quadratic-like map is provided by a quadratic polynomial $f_c: z\mapsto z^2+c$
restricted to a disk $Y=\Disk_R$ of sufficiently big radius. The  Douady and Hubbard {\it Straightening Theorem} asserts that 
any quadratic-like map $f$ is hybrid conjugate % equivalent
to some restricted quadratic polynomial $f_c$. 
Moreover, if $J$ is connected then the parameter $c\in \MM$ is unique.

As for quadratic polynomials, two fixed points of a quadratic-like map with connected Julia set
have a different dynamical meaning. One of them, called $\beta$, is 
 the landing point of a proper arc $\gamma\subset X \sm K(f) $  such that $f(\gamma)\supset \gamma$.
It is either repelling or parabolic with multiplier one. The other fixed point, called  $\alpha$, is either 
non-repelling or a {\it  cut-point}  of the Julia set.

% The {\em postcritical set} $\Post\equiv \Post_f$ of $f$ is the closure of  $\orb 0$.   

\subsubsection{Combinatorial models}
A quadratic polynomial $f$ is called {\em periodically  hyperbolic}  (resp., {\em repelling}) if it does not have neutral
(resp., non-repelling) cycles. Note that in the periodically repelling case, $\inter \filled=\emptyset$, so $\Julia=\filled$. 
The filled Julia set of a periodically hyperbolic map admits a locally connected {\em combinatorial model}
$\filled^\com= \filled^\com (f)$ obtained by taking the quotient of the unit disk by the associated
rational lamination (see \cite[\S 32.1.3]{L:book}). The combinatorial model is endowed with the induced dynamics $f_\com: \filled^\com\ra \filled^\com$ which is semi-conjugate to $f$ by a natural  projection $\pi: \filled\ra \filled^\com$ whose {\em fibers} are ``combinatorial classes'' -- sets of points inseparable by preperiodic points.

 \subsubsection{Hubbard continua}\label{sss:CombHT}
In the superattracting case (when the critical point is periodic) the {\em Hubbard tree} $\Hub_f$
is defined as the smallest tree in $\filled$ containing the postcritical set such that $\Hub_f$ intersects the components
of $\inter \filled$ along internal rays. It is invariant under $f$ and is marked with the orbit of $0$. If $\filled$ is periodically repelling and locally connected, we can define the
{\em Hubbard dendrite}  $\Hub_f$ as the smallest connected closed forward-invariant subset of $\filled$ containing the postcritical set.

For a general periodically repelling map,
the (combinatorial) Hubbard dendrite $\Hub_f^\com\subset \filled^\com$ is defined for its combinatorial model $f\colon \filled^\com \selfmap$.
Lifting $\Hub_f^\com$ via $\pi\colon \filled\to \filled^\com$, we obtain the {\em Hubbard continuum} $\Hub_f$.

For a periodically repelling $f$ and $x,y \in \Julia_f=\filled_f$, we define a \emph{geodesic} continuum 
\begin{equation}
\label{eq:geod contin}
[x,y]\coloneqq \pi^{-1}[\pi(x),\pi(y)],
\end{equation} 
where $[\pi(x),\pi(y)]$ is a unique arc in $\Julia^\com$ connecting $\pi(x),\pi(y)$.

\subsubsection{QL-renormalization}\label{renorm sec} 

A quadratic-like map  
$f: X \ra Y$ is called {\it DH renormalizable} (after Douady and Hubbard) if there is a quadratic-like
restriction 
\begin{equation}
\label{eq:DH pre renorm}
        f_1=f_{1,0}= \RR f= f^{p_1}: X_{1,0} \ra Y_{1,0}
\end{equation}
 with connected Julia set $\filled_{0}^{[1]}$ such that
the little Julia sets 
\begin{equation}
\label{eq:small J cycle}
\filled^{[1]}_i\coloneqq f^i\left(\filled_{0}^{[1]}\right),  \hspace{1cm} k=0,\dots, p_1-1, 
\end{equation} are either pairwise disjoint or else touch at their $\beta$-fixed points. In the former case the renormalization is called {\it primitive}, while in the latter it is called {\it satellite}.

Note that  there are many  ql maps $f_1\colon X_{1,0}' \ra Y'_{1,0} $ that satisfy the above requirements
(if $f$ is renormalizable). However, all of them represent the same {\em ql germ}.

The map $f_1 $ in~\eqref{eq:DH pre renorm} is called a {\it pre-renormalization} of $f$.
If it is considered up to a linear rescaling, it is called  the {\it  renormalization} of $f$. (In what follows, we will often skip the prefix ``-pre'' as long as it does not lead to a confusion.)   

For every $\filled^{[1]}_i$, there are $Y_{1,i} \Supset X_{1,i}\Supset \filled^{[1]}_i$ such that 
\begin{equation}
\label{eq:f_1 i} f_1=f_{1,i}= \RR f= f^{p_1}: X_{1,i} \ra Y_{1,i}
\end{equation}
is a ql map with non-escaping set $\filled^{[1]}_i=\filled_{f_{1,i}}$. All maps~\eqref{eq:f_1 i} represent conformally conjugate germs. We assume that $\filled^{[1]}_0$ contains the critical point of $f$; see~\ref{CrtPoint}.
%We will usually assume that $f_{1,0}$

\subsubsection{Little copies of $\MM$}\label{sss:LittleCopies} The sets $\filled^{[1]}_i$ in~\eqref{eq:small J cycle} are referred to as the {\it little (filled) Julia sets}.
Their positions in the big Julia set $\filled_f$ determine the renormalization {\it combinatorics}.
By the Douady-Hubbard Straightening Theorem \cite{DH:pol-like},
the set of parameters $c$ for which the quadratic polynomial $f_c\colon z\mapsto z^2+c$ is renormalizable
with a given combinatorics forms a {\it little Mandelbrot copy}  $\MM_1\subset \Mandel$
(see \cite[Theorem 43.1]{L:book}).
% (see Figure \ref{little copy fig}).
%We can encode the renormalization combinatorics by the corresponding
%copy $\MM'$ itself.
%Equivalently, it  can be  encoded  by the center $c_\base$ of $\MM'$
%or the corresponding  Hubbard tree $\Hub_\base$ obtained by connecting
%the points of the superattracting cylce by arcs in the Julia set. 
%
The renormalization combinatorics can be formally encoded by the  Hubbard tree $\Hub_\base$
of the superattracting center  $c_\base$ of $\MM_1$.

Each little copy $\MM_1$ can be canonically mapped onto the whole
Mandelbrot set $\MM$ by the {\em straightening homeomorphism}
$\chi_1: \MM_1\ra \MM$.

\comm{*****
In fact, the family of renormalizations
$R(P_c)$, $c\in \MM'$, with a given combinatorics can be included in
a quadratic-like family $\Bbb F= (f^p : X_c\ra Y_c)$ over
some domain $\La\supset \MM'$ so that $\MM'= \MM_{\Bbb F}$.
A natural base point $c_\circ \in \MM'$ in this family is the superattracting
parameter with period $p$. 
It is called the {\it center} of $\MM'$.
Any superattracting parameter
in $\MM$  with period $p>1$ is the center of some Mandelbrot  copy
$\MM'$ like this. 
Moreover, in case of primitive combinatorics the quadratic-like family
$\Bbb F$ is proper and unfolded. 
(See \cite{DH,D-ICM,L-book} for a discussion of all these facts.)
*************}

%\begin{figure}    
%\centerline{\includegraphics[width=.8\hsize]{pictures/copy.pdf}}
%\caption{A primitive little $M$-copy. 
%\label{little copy fig}}
%\end{figure}

A little Mandelbrot copy $\MM_1$ is called {\it primitive} or {\it satellite} depending on  the type of the corresponding renormalization. They can be easily distinguished as any satellite copy is attached to
some hyperbolic component of $\inter \MM$ and does not have the cusp at its root point. 

\subsubsection{Infinitely renormalizable maps}\label{sss:infin renorm} The notions  of an infinitely DH renormalizable map $f$ with renormalization periods $p_n$, and its renormalizations $f_n= \RR^n f$,
are defined naturally. 

By default, we assume that $p_n$ is the \emph{smallest renormalization period after $p_{n-1}$}. We will denote by $\filled^{[n]}_i, i\in \{0,1,\dots, p_n-1\}$ the level $n$ little Julia sets of $f$ enumerated dynamically so that $\filled^{[n]}_0$ contains the critical point of $f$. We will write
\begin{equation}
\label{eq:bfilled}
\bfilled^{[n]}\coloneqq \bigcup_i \filled^{[n]}_i
\end{equation}

 The ratios $q_n:= p_n/p_{n-1}$ are called {\em relative periods}.
One says that such a map has  {\it bounded combinatorics of type
  $\bar p$} 
if the relative periods are bounded by $\bar p$.
In this case, the  map $f$  is called {\em Feigenbaum of type $\bar p$}.
We say that a Feigenbaum map is {\it
  primitive/satellite}  if all its renormalizations are such.  
A Feigenbaum map has   {\it a priori} bounds if
% the renormalizations
\begin{equation}
\label{eq: ql a priori bounds}
\mod \RR^n f\geq \eps>0
\end{equation}

We say that the family $\FF_{\bar p}$ of Feigenbaum maps of type $\bar p$
have {\it beau bounds}
if there exists $\mu>0$ depending only on $\bar p$  such that for any $\nu>0$ there exists 
$n_0= n_0(\bar p, \nu)$ such that  for any  $f\in \FF_{\bar p}$  with $\mod f\geq \nu$ we have
$$
    \mod \RR^n f \geq \mu\quad  \text{for all }\ n\geq n_0. 
$$    
%It was proved by Kahn \cite{K} that {\it primitive Feigenbaum  maps
%  have beau bounds}, with $\mu$ depending only on the combinatorial
%bound.
%%% In fact,  $\mu$  can be made uniform over  some class of combinatorics \cite{KL}.%

\subsection{Bushes}\label{ss:bushes}
 Consider a DH renormalizable quadratic-like map $f\colon X\to Y$. The \emph{bush} of $f$ is 
\begin{equation}
\label{eq:dfn:bush}
\bush_f\equiv \bush(f)\coloneqq \Hub_f \cup  \bfilled^{[1]},
\end{equation}
where $\Hub_f$ is the Hubbard continuum and $\bfilled^{[1]}$ is the periodic cycle of level one little Julia sets~\eqref{eq:bfilled}. For $m \ge 0$, we define the \emph{bush of height $m$} to be 
\begin{equation}
\label{eq:dfn:bush:preim}
\bush^{(m)}=\bush^{(m)}_f\coloneqq f^{-m}(\bush_f).
\end{equation}

\subsubsection{Little bushes}\label{sss:LittleBushes}
Suppose that $f$ is at least $n+1$ times DH renormalizable and let $f_n=\RR^n f$ be its $n$th renormalization of $f$. Then $f_n$ has a well defined bush $\bush(f_n)$. Consider the periodic cycle of little level $n$ filled Julia sets $\filled_i^{[n]}$ associated with $f_n$ in the dynamical plane of $f$. Let $f_{n,i}$ be the $n$th prerenormalization around $\filled_i^{[n]}$, compare~\eqref{eq:f_1 i}. Then $\filled_i^{[n]}$ contains the little bush $\bush_i^{[n]}\equiv\bush(f_{n,i})\simeq \bush(f_n)$ as well as $\bush_i^{[n],(m)}\equiv\bush^{(m)}(f_{n,i})\simeq \bush^{(m)}(f_n)$. Note that $ \bush^{[n],(m)}_i$ is the unique periodic lift of $\bush_{i+m}^{[n]}$ under $f^{m}$.

We write
\[\bbush^{[n]}\coloneqq \bigcup_i \bush_i^{[n]}
,\sp  \bbush^{[n],(m)}\coloneqq f^{-m} \left(\bbush^{[n]}\right),\sp\bbush^{[n],(m)}_\per\coloneqq\bbush^{[n],(m)} \cap \bfilled^{[n]}.
\]
Observe that $\bbush^{[n],(m)}_\per$ is the union of periodic little bushes of height $m$. A \emph{preperiodic bush} $\bush^{[n],(m)}_a$ is a non-periodic connected component of $\bbush^{[n],(m)} $. The \emph{preperiod} of $\bush^{[n],(m)}_a$ is the smallest $s\le m$ such that $f^s \left( \bush^{[n],(m)}_a \right)$ is periodic.

 Every periodic or preperiodic little bush $\bush_{a}^{[n],(m)}$ is within a unique little periodic or preperiodic filled Julia set $\filled^{[n]}_a$ associated with $f_n$. The $\bush_{a}^{[n],(m)}$ exhaust $\filled^{[n]}_a$:
 \[\filled^{[n]}_a = \overline{\bigcup_{m} \bush^{[n],(m)}_a}\ .\] 
By construction, the $\bush^{[n],(m)}_a$ are pairwise disjoint but the $\filled^{[n]}_a$ may touch each other in the satellite case. 
\subsubsection{Superattracting model}
\label{sss:psib:super attr model}
Consider a ql map $f\colon X\to Y$ and assume it is $n+1$ DH renormalizable. Then $f$ is hybrid equivalent to $z^2+c$, where $c$ is in a level $n+1$ little copy $\MM^{[n+1]}_i\subset \MM$. A \emph{superattracting} model for $f$ of level $n+1$ is any ql map $f_\circ \colon X_\circ \to Y_\circ$ hybrid equivalent to the center of $\MM^{[n+1]}_i$.

It is well-known (follows, for instance, from the lamination theory) that the Hubbard continua $\Hub_f$ and $\Hub_\circ\equiv \Hub_{f_\circ}$ are \emph{combinatorially equivalent} up to $\bbush^{[n]}$: there is a bijection between 
\begin{itemize}
\item components of $\Hub_f\setminus f^{-m}\left(\bbush^{[n]}_f\right)$ and  components of $\Hub_\circ\setminus f^{-m}_\circ\left(\bbush^{[n]}_\circ\right)$ for all $m\ge0$; and
\item components of $\Hub_f\cap f^{-m}\left(\bbush^{[n]}_f\right)$ and of $\Hub_\circ\cap  f^{-m}_\circ\left(\bbush^{[n]}_\circ\right)$ for all $m\ge0$
\end{itemize}
that respects the adjacency, natural embedding, and dynamical relations between respective components.  In other words, the above components define equivalent Markov partitions for $f, f_\circ$.

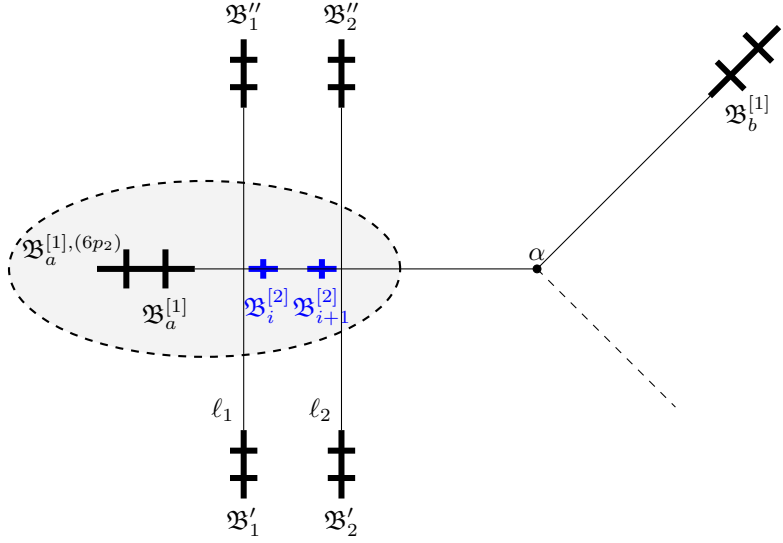
\begin{figure}[t!]
\[\begin{tikzpicture}[scale=1.3]

\begin{scope}[rotate =0]
\draw[line width=0.8mm] (-0.5,0)--(0.5,0); 
\draw[line width=0.8mm] (-0.2,-0.2)--(-0.2,0.2);
\draw[line width=0.8mm] (0.2,-0.2)--(0.2,0.2);
\coordinate (ba) at (0.5,0);
\node[below] at  (0.2,-0.2) {$\bush^{[1]}_{a}$};

 \draw[dashed,line width=0.8,fill, fill opacity=0.05] (0.6,0) ellipse (2cm and 0.9cm);
\node[] at (-0.75,0.2){$\bush^{[1], (6p_2)}_{a}$};

\begin{scope}[shift={(1,2)},rotate=90,scale=0.7]
\draw[line width=0.8mm] (-0.5,0)--(0.5,0); 
\draw[line width=0.8mm] (-0.2,-0.2)--(-0.2,0.2);
\draw[line width=0.8mm] (0.2,-0.2)--(0.2,0.2);
\coordinate (xup) at (-0.5,0);

\node[above] at (0.5,0) {$\bush''$};
\end{scope}
\begin{scope}[shift={(1,-2)},rotate=-90,scale=0.7]
\draw[line width=0.8mm] (-0.5,0)--(0.5,0); 
\draw[line width=0.8mm] (-0.2,-0.2)--(-0.2,0.2);
\draw[line width=0.8mm] (0.2,-0.2)--(0.2,0.2);
\coordinate (xdown) at (-0.5,0);

\node[left] at (-0.8,0) {$\ell_1$};
\node[below] at (0.5,0) {$\bush'$};
\end{scope}

\draw[] (xup)  -- (xdown); 

\begin{scope}[shift={(-0.05,0)}]
\draw [blue, line width=0.8mm] (1.1,0)--(1.4,0)
(1.25,-0.1)  -- (1.25,0.1) ;
\node[blue, above] at(1.38,0.) {$\bush_{i_1}^{[2]}$};
\end{scope}

%xxxxx

\begin{scope}[shift={(0.55,0)}]
\draw [blue, line width=0.8mm] (1.1,0)--(1.4,0)
(1.25,-0.1)  -- (1.25,0.1) ;
\node[blue, below] at(1.2,-0.1) {$\bush_{i_2}^{[2]}$};
\end{scope}

\begin{scope}[shift={(1,0)}]
\begin{scope}[shift={(1,2)},rotate=90,scale=0.7]
\draw[line width=0.8mm] (-0.5,0)--(0.5,0); 
\draw[line width=0.8mm] (-0.2,-0.2)--(-0.2,0.2);
\draw[line width=0.8mm] (0.2,-0.2)--(0.2,0.2);
\coordinate (xup) at (-0.5,0);
\node[above] at (0.5,0) {$\bush'''$};

\end{scope}
\begin{scope}[shift={(1,-2)},rotate=-90,scale=0.7]
\draw[line width=0.8mm] (-0.5,0)--(0.5,0); 
\draw[line width=0.8mm] (-0.2,-0.2)--(-0.2,0.2);
\draw[line width=0.8mm] (0.2,-0.2)--(0.2,0.2);
\coordinate (xdown) at (-0.5,0);
\node[below] at (0.5,0) {$\bush^{\rn{4}}$};

\node[left] at (-0.8,0) {$\ell_2$};

\end{scope}

\draw[] (xup)  -- (xdown); 

\end{scope}

\end{scope}

\begin{scope}[rotate around={45:(4,0)} ,shift={(7,0)},xscale=-1]
\draw[line width=0.8mm] (-0.5,0)--(0.5,0); 
\draw[line width=0.8mm] (-0.2,-0.2)--(-0.2,0.2);
\draw[line width=0.8mm] (0.2,-0.2)--(0.2,0.2);
\coordinate (bb) at (0.5,0);
\node[below] at  (0.15,-0.2) {$\bush^{[1]}_{b}$};

\end{scope}

\draw[] (ba)--(4,0)--(bb);

\begin{scope}[rotate around={-45:(4,0)}]
\coordinate (cc) at (6,0);
\end{scope}
\draw[dashed] (4,0)--(cc);

\filldraw  (4,0) circle (0.04 cm);
\node[above] at (4,0) {$\alpha$};

\end{tikzpicture}\]

\caption{Illustration to Lemma~\ref{lem:bush:prim comb}.}
\label{Fg:lem:bush:prim comb}
\end{figure}
%Since quadratic maps have a single critical point, the combinatorial equi
\subsubsection{Satellite combinatorics} Suppose that  $f$ is at least $3$ times renormalizable and the first renormalization of $f$ is satellite. Then all $\filled_i^{[1]}$ are organized in the \emph{satellite flower} around the $\alpha$ fixed point of $f$: \[\bush_f=\bfilled^{[1]}_f\hspace{1cm} \text{ because }\hspace{1cm} \bfilled^{[1]}_f\supset \Hub_f,\]
where $\Hub_f$ is the Hubbard continuum of $f$. For $s\ge 0$ we write $\Hub^{(s)}_f = f^{-s}(\Hub_f)$. As before, the $p_k$ are the periods of level $k$ Julia sets.

\begin{lem}[Satellite flower] 
\label{lem:bush:prim comb} For a satellite $f$ as above, let $\bush^{[1]}_{a}, \bush^{[1]}_{b}$ be two of its level $1$ periodic bushes. Then the continuum $\Hub^{(6p_2)}_f$ contains geodesic continua $\ell_1$ and $\ell_2$ (as in~\S\ref{sss:CombHT}) connecting preperiodic lifts $\bush',\bush''$  and $\bush''',\bush^\rn{4}$ of $\bush^{[1]}_b$ of preperiods $\le 6 p_2$ such that, see Figure~\ref{Fg:lem:bush:prim comb}:
\begin{itemize}
\item $\ell_1$ and $\ell_2$ separate $\bush^{[1]}_a$ from $\alpha$ (within $\filled_f$),
\item the geodesic continuum  $\ell\subset \Hub_f$ connecting $\ell_1, \ell_2$ intersects little preperiodic bushes $\bush_{i_1}^{[2]},\bush^{[2]}_{i_2}$ of level $2$ with preperiods less than $6p_1$;
\item  $\bush_{i_1}^{[2]},\bush^{[2]}_{i_2}$ are disjoint from $\ell_1\cup \ell_2$.
\end{itemize}
 \end{lem}
 \noindent We remark that $\ell,\bush_{i_1}^{[2]},\bush^{[2]}_{i_2}$ are within $ \bush^{[1],(6 p_2)}_a$.
 \begin{proof}
Let $\bush_\gamma^{[2]}\subset \bush^{[1]}_a$ be the level $2$ periodic bush closest to $\alpha$; i.e. $(\Hub_f\cap \bush^{[1]}_a)\setminus \bush^{[2]}_\gamma$ and $\alpha$ are in different components of $\Hub_f\setminus \bush_\gamma^{[2]}$. Let the $\bush_{i(t)}^{[2]}\subset  \bush^{[1],(tp_1)}_a$ be the lifts of $\bush_\gamma^{[2]}$ under $f^{t p_1}$ towards $\alpha$; i.e. each  $\bush_{i(t)}^{[2]}$ separates $\alpha$ from  $\bush^{[1],(tp_1-p_1)}_a$. 

Since $f^{6p_2}$ has critical points in $\bush^{[2]}_{i(1)}$ and in $\bush_{i(4)}^{[2]}$, we can select $\ell_1$ and $\ell_2$ passing through these critical points such that $\ell_1,\ell_2$ connect preperiodic lifts $\bush',\bush''$  and $\bush''',\bush^\rn{4}$ of $\bush^{[1]}_b$. Then $\ell_1, \ell_2$ separate $\bush^{[1]}_a$ from $\alpha$, and $\ell_1\cup \ell_2$ separate $\bush^{[2]}_{i(2)},\bush^{[2]}_{i(3)}$ from $\alpha$ and $\bush^{[1]}_a$; i.e.~we can take $\bush^{[2]}_{i_1},\bush^{[2]}_{i_2}$ to be $\bush^{[2]}_{i(2)},\bush^{[2]}_{i(3)}$.
 \end{proof}

\subsection{Invariant arc diagrams}
\label{sss:InvArcDiagr} In this subsection, we will discuss invariant up to homotopy arc diagrams of ql maps. Arc diagrams endowed with weights will naturally appear from the thin-thick decompositions of the dynamical planes of $\psi^\bullet$-ql maps, see~\S\ref{ss:WAD}, \S\ref{ss:Width+WAD}. 

\subsubsection{Arc diagrams}\label{sss:AD}Consider a hyperbolic open Riemann surface $S$ of finite type without cusps. We endow $S$ with its ideal boundary $\partial^i S$. This naturally makes $V\coloneqq \overline S\equiv S\cup \partial^i S$ a compact surface.

 A \emph{path} (closed or open) $\ell$ in $S$ is an embedded (closed or open) interval $\ell\colon I \to S$. 

An open path $\gamma\colon (0,1)\to S$ is \emph{proper} if it extends to $\gamma\colon [0,1]\to V\equiv S\cup \partial^i S $ with $\gamma\{0,1\}\subset \partial^i S$. 
Two proper paths $\gamma_0,\gamma_1$ in $S$ are homotopic if there is a homotopy $\gamma_t$ among proper paths. Similarly, \emph{proper curves} and their homotopy are defined.

An \emph{arc} on $S$ is a class of properly homotopic paths, $\alpha=[\gamma]$. A curve $\gamma$ is \emph{trivial} if it can be represented in an arbitrary small neighborhood of $\partial^i S$. Two different arcs are \emph{non-crossing} if they can be represented by non-crossing paths. 

An \emph{arc diagram} (AD) is a family of non-trivial pairwise non-crossing arcs $A=\{\alpha_i\}$. A \emph{weighted arc diagram} (WAD) $\AA=\sum_{ \alpha_i \in A} w_i \alpha_i,\ w_i\in \R_{>0}$ is an arc diagram endowed with positive weights.

\subsubsection{Arc diagrams of ql maps}

Consider a ql map $f\colon X\to Y$ and assume that it is $n+1$ times DH renormalizable.

An arc diagram $A=\{\alpha_i\}$ on $X\setminus \bbush^{[n]}$ is \emph{horizontal} if every $\alpha_i$ connects components of $\bbush^{[n]}$. A horizontal arc diagram $A=\{\alpha_i=[\gamma_i]\}$ is called \emph{invariant} if every $\alpha_i$ can be represented up to a proper homotopy in $X\setminus \bbush^{[n]}$ in an arbitrary small neighborhood of $f^{-1}\left(\bbush^{[n]}\cup \Gamma \right)$, where $\Gamma=\bigcup_i \gamma_i$. In other words, every $[\gamma]\in A$ can be presented as a concatenation 
\begin{equation}
\label{eq:form for gamma:base}
\gamma = \ell_0\#\gamma_1\# \ell_1 \#\gamma_2\# \dots \#\gamma_s\#\ell_s,
\end{equation}
where
\begin{itemize}
\item $[\gamma_j]\in f^*(A)$ are proper arcs in $X\setminus \bbush^{[n],(1)}=f^{-1}\left(Y\setminus \bbush^{[n]}\right)$; and
\item every component of $\ell_i\setminus \bbush^{[n],(1)}$ is trivial in $X\setminus \bbush^{[n],(1)}$ (with respect to a proper homotopy, see~\S\ref{sss:AD}).
\end{itemize} (In other words, the $\ell_i$ are contractible into $\bbush^{[n],(1)}$.) Here $f^*(A)$ is the pullback of $A$. Figure~\ref{Fig:exm Inv Arc Diagr} gives an example of an invariant arc diagram.

\begin{figure}[t!]
\[\begin{tikzpicture}[xscale=2,yscale=1.5]

\begin{scope}[rotate =0,shift={(1.3,0)},scale=0.6]
\draw[line width=0.8mm] (-0.5,0)--(0.5,0); 
\draw[line width=0.8mm] (-0.2,-0.2)--(-0.2,0.2);
\draw[line width=0.8mm] (0.2,-0.2)--(0.2,0.2);

\coordinate (la1) at (-0.2,0.2);
\coordinate (la0) at (-0.2,-0.2);

\draw[line width=0.4mm] (4,0.7) edge[->,bend right]node[above]{$f$} (0.7,0.7);

\node[above,scale=1.2] at(0,0.2){$\bush^{[1]}_0$};

 \end{scope}
  
  \begin{scope}[rotate =120,shift={(1.3,0)}, scale=0.6]
\draw[line width=0.8mm] (-0.5,0)--(0.5,0); 
\draw[line width=0.8mm] (-0.2,-0.2)--(-0.2,0.2);
\draw[line width=0.8mm] (0.2,-0.2)--(0.2,0.2);

\coordinate (lb1) at (-0.2,0.2);
\coordinate (lb0) at (-0.2,-0.2);
  \end{scope}

  \begin{scope}[rotate =240,shift={(1.3,0)}, scale=0.6]
\draw[line width=0.8mm] (-0.5,0)--(0.5,0); 
\draw[line width=0.8mm] (-0.2,-0.2)--(-0.2,0.2);
\draw[line width=0.8mm] (0.2,-0.2)--(0.2,0.2);
%\node[above,scale=1.2] at(0,0.2){$\bush_0$};

\coordinate (lc1) at (-0.2,0.2);
\coordinate (lc0) at (-0.2,-0.2);
  \end{scope}

 \draw[line width=0.3mm, red] (la1) edge[bend left=10] node[above] {$\alpha$} (lb0);

 \draw[line width=0.3mm, blue] (lc0) edge[bend right=10] node[left] {$\beta$}  (lb1);

\begin{scope}[shift={(4.5,0)}]

\begin{scope}[rotate =0,shift={(1.3,0)},scale=-0.6]
\draw[line width=0.8mm] (-0.5,0)--(0.5,0); 
\draw[line width=0.8mm] (-0.2,-0.2)--(-0.2,0.2);
\draw[line width=0.8mm] (0.2,-0.2)--(0.2,0.2);

\draw[line width=0.9mm,gray] (0.5,0) -- (0.9,0); 
\draw[line width=0.9mm,gray] (0.65,-0.2) -- (0.65,0.2); 

\coordinate (ra1) at (0.65,-0.2) ;
\coordinate (ra0) at (0.65,0.2) ;

\node[above,scale=1.2] at(0,-0.2){$\bush^{[1],(1)}_0$};

  \end{scope}
  
  \begin{scope}[rotate =120,shift={(1.3,0)}, scale=0.6]
\draw[line width=0.8mm] (-0.5,0)--(0.5,0); 
\draw[line width=0.8mm] (-0.2,-0.2)--(-0.2,0.2);
\draw[line width=0.8mm] (0.2,-0.2)--(0.2,0.2);

\coordinate (rb1) at (-0.2,0.2);
\coordinate (rb0) at (-0.2,-0.2);
  \end{scope}

  \begin{scope}[rotate =240,shift={(1.3,0)}, scale=0.6]
\draw[line width=0.8mm] (-0.5,0)--(0.5,0); 
\draw[line width=0.8mm] (-0.2,-0.2)--(-0.2,0.2);
\draw[line width=0.8mm] (0.2,-0.2)--(0.2,0.2);

\coordinate (rc1) at (-0.2,0.2);
\coordinate (rc0) at (-0.2,-0.2);
  \end{scope}

\draw[line width=0.3mm, red] (rc1) edge[bend left=10] node[above] {$\tilde \alpha$} (ra0);

 \draw[line width=0.3mm, blue] (rb0) edge[bend right=10] node[above] {$\tilde \beta$}  (ra1);

\end{scope}
\end{tikzpicture}\]

\caption{The arc diagram $\{[\alpha], [\beta]\}$ is invariant. Here $\widetilde \alpha, \widetilde \beta$ are lifts of $\alpha,\beta$, the map $f$ is the Rabbit map tuned with the Feigenbaum map, peripheral lifts of $\alpha,\beta$ and preperiodic bushes are omitted on the right side. The arc diagram is invariant because $\alpha$ can be properly homotoped into $\widetilde \beta$ and $P\coloneqq \bush^{[1],(1)}_0\setminus \bush^{[1]}_0$ while $\beta$ can be properly homotoped into $\widetilde \beta, P, \widetilde \alpha$.}

\label{Fig:exm Inv Arc Diagr}
\end{figure}
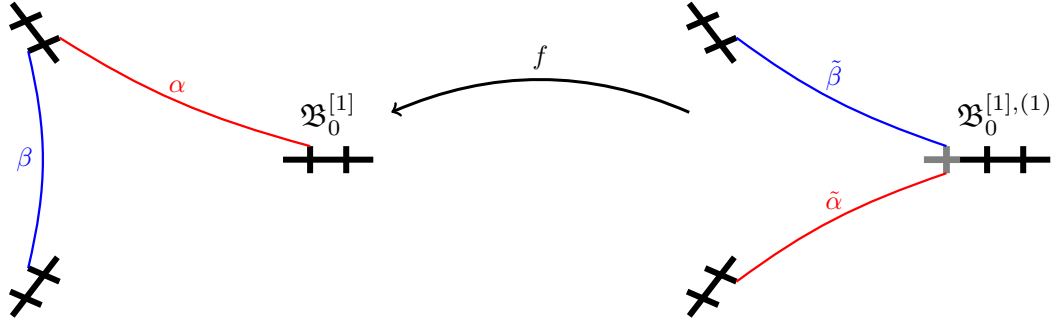

\begin{lem}[Alignment with $\Hub_f$, following~{\cite[\S4]{K}}]\label{lem:Alignment}
If $A=\{\alpha_i\}$ is an invariant horizontal AD on $X\setminus \bbush^{[n]}$, where $f\colon X\to Y$ is a ql map, then $A$ is \emph{aligned with the Hubbard continuum} $\Hub_f$: every $\alpha_i$ can be represented in an arbitrary small neighborhood of a geodesic continuum $T_i\subset \Hub_f\setminus \bbush^{[n]}$ connecting components of $\bbush^{[n]}$.
\end{lem}
\begin{proof}
Write $\alpha_i=[\gamma_i]$ and consider $Y'\coloneqq Y\setminus \left(\bbush^{[n]}\cup  \Gamma \right)$, where $\Gamma=\bigcup_i \gamma_i$. Then one of the little bushes $\bush^{[n]}_i$ is accessible from the outermost component of $\partial Y'$. Therefore, we can select a proper path $\ell \subset Y'$ connecting $\partial Y$ and $\bbush^{[n]}$. For every such a path $\ell$, its \emph{legal pullback} $\widetilde \ell$ is any of its lift connecting $\bbush^{[n]}$ and $\partial X$ and concatenated by a path in $Y\setminus X$ so that $\widetilde \ell\subset Y\setminus \bbush^{[n]}$ is a proper path connecting $\partial Y$ and $\bbush^{n}$. Since $A$ is invariant, $\widetilde \ell$ can be again represented in $Y'$ up to a proper homotopy in $Y\setminus \bbush^{[n]}$.

It is well-known that iterated pullbacks of $\ell$ can realize all periodic rays landing at $\bbush^{[n]}$ up to a proper homotopy in $Y\setminus \bbush^{[n]}$. (It is sufficient to verify the property for the superattracting model~\S\ref{sss:psib:super attr model}.)  Therefore, $A$ is aligned with $\Hub_f$.
\end{proof}

\subsubsection{Genuine periodic arcs} \label{sss:periodic arcs}Let $A=\{\alpha_i\}$ be a horizontal invariant AD on $Y\setminus \bbush^{[n]}$ of $f\colon X\to Y$. For every $m\ge 1$, we can present $[\gamma]\in A$ as a concatenation
\begin{equation}
\label{eq:form for gamma}
\gamma = \ell_0\#\gamma_1\# \ell_1 \#\gamma_2\# \dots \#\gamma_s\#\ell_s
\end{equation}
(compare with~\eqref{eq:form for gamma:base}), where 
\begin{itemize}
\item $[\gamma_j]\in\big( f^{m}\big)^*(A)$ are proper arcs in $f^{-m}(Y)\setminus \bbush^{[n],(m)}=f^{-m}\left(Y\setminus \bbush^{[n]}\right)$; and
\item every component of $\ell_i\setminus \bbush^{[n],(m)}$ is trivial in $X^m\setminus \bbush^{[n],(m)}$  (with respect to a proper homotopy, see~\S\ref{sss:AD}) .
\end{itemize}
We call~\eqref{eq:form for gamma} the \emph{decomposition of $\gamma$ rel $\bbush^{[n],(m)}$.}  We define the \emph{expansivity number} \[\ExpN_A(f^m, [\gamma])\coloneqq \min \{s \mid \ s \text{ is in \eqref{eq:form for gamma}}\},\]
where the minimum is taking over all paths possible $\gamma$ as above.

We call an arc $\alpha\in A$ a \emph{genuine periodic} if $\ExpN_A(f^m,\alpha)=1$ for all $m\ge 1$.

\begin{figure}[t!]
\[\begin{tikzpicture}[xscale=2,yscale=1.5]

\begin{scope}[rotate =0,shift={(1.3,0)},scale=0.6]
\draw[line width=0.8mm] (-0.5,0)--(0.5,0); 
\draw[line width=0.8mm] (-0.2,-0.2)--(-0.2,0.2);
\draw[line width=0.8mm] (0.2,-0.2)--(0.2,0.2);

\coordinate (la1) at (-0.2,0.2);
\coordinate (la0) at (-0.2,-0.2);

\draw[line width=0.4mm] (4,0.7) edge[->,bend right]node[above]{$f$} (-0.,1.8);

\node[above,scale=1.2] at(0,0.2){$\bush^{[1]}_0$};

 \end{scope}
  
  \begin{scope}[rotate =90,shift={(1.3,0)}, scale=0.6]
\draw[line width=0.8mm] (-0.5,0)--(0.5,0); 
\draw[line width=0.8mm] (-0.2,-0.2)--(-0.2,0.2);
\draw[line width=0.8mm] (0.2,-0.2)--(0.2,0.2);

\coordinate (lb1) at (-0.2,0.2);
\coordinate (lb0) at (-0.2,-0.2);

\coordinate (lbb) at (-0.5,0);
  \end{scope}

  \begin{scope}[rotate =180,shift={(1.3,0)}, scale=0.6]
\draw[line width=0.8mm] (-0.5,0)--(0.5,0); 
\draw[line width=0.8mm] (-0.2,-0.2)--(-0.2,0.2);
\draw[line width=0.8mm] (0.2,-0.2)--(0.2,0.2);
%\node[above,scale=1.2] at(0,0.2){$\bush_0$};

\coordinate (lc1) at (-0.2,0.2);
\coordinate (lc0) at (-0.2,-0.2);
  \end{scope}
  
    \begin{scope}[rotate =270,shift={(1.3,0)}, scale=0.6]
\draw[line width=0.8mm] (-0.5,0)--(0.5,0); 
\draw[line width=0.8mm] (-0.2,-0.2)--(-0.2,0.2);
\draw[line width=0.8mm] (0.2,-0.2)--(0.2,0.2);
%\node[above,scale=1.2] at(0,0.2){$\bush_0$};

\coordinate (ld1) at (-0.2,0.2);
\coordinate (ld0) at (-0.2,-0.2);
\coordinate (ldd) at (-0.5,0);

  \end{scope}

 \draw[line width=0.3mm, red] (la1) edge[bend left=10] node[above right] {$\alpha_1$} (lb0);
  \draw[line width=0.3mm, red] (lb1) edge[bend left=10] node[above left] {$\alpha_2$} (lc0);

  \draw[line width=0.3mm, red] (lc1) edge[bend left=10] node[below left] {$\alpha_3$} (ld0);
    \draw[line width=0.3mm, red] (ld1) edge[bend left=10] node[below right] {$\alpha_4$} (la0);

% \draw[line width=0.3mm, blue] (lc0) edge[bend right=10] node[left] {$\beta$}  (lb1);

 \draw[line width=0.3mm, blue] (lbb) edge[] node[left] {$\beta$}  (ldd);

\begin{scope}[shift={(4.5,-1)}]

\begin{scope}[rotate =0,shift={(1.3,0)},scale=-0.6]
\draw[line width=0.8mm] (-0.5,0)--(0.5,0); 
\draw[line width=0.8mm] (-0.2,-0.2)--(-0.2,0.2);
\draw[line width=0.8mm] (0.2,-0.2)--(0.2,0.2);

\draw[line width=0.9mm,gray] (0.5,0) -- (0.9,0); 
\draw[line width=0.9mm,gray] (0.65,-0.2) -- (0.65,0.2); 

\coordinate (ra1) at (0.65,-0.2) ;
\coordinate (ra0) at (0.65,0.2) ;

\coordinate (raa) at (0.9,0) ;

\node[above,scale=1.2] at(0,-0.2){$\bush^{[1],(1)}_0$};

  \end{scope}
  
   \begin{scope}[rotate =90,shift={(1.3,0)}, scale=0.6]
\draw[line width=0.8mm] (-0.5,0)--(0.5,0); 
\draw[line width=0.8mm] (-0.2,-0.2)--(-0.2,0.2);
\draw[line width=0.8mm] (0.2,-0.2)--(0.2,0.2);

\coordinate (rb1) at (-0.2,0.2);
\coordinate (rb0) at (-0.2,-0.2);

\coordinate (rbb) at (-0.5,0);
  \end{scope}

   \begin{scope}[rotate =180,shift={(1.3,0)}, scale=0.6]
\draw[line width=0.8mm] (-0.5,0)--(0.5,0); 
\draw[line width=0.8mm] (-0.2,-0.2)--(-0.2,0.2);
\draw[line width=0.8mm] (0.2,-0.2)--(0.2,0.2);

\coordinate (rc1) at (-0.2,0.2);
\coordinate (rc0) at (-0.2,-0.2);
\coordinate (rcc) at (-0.5,0);

  \end{scope}
  
   \begin{scope}[rotate =270,shift={(1.3,0)}, scale=0.6]
\draw[line width=0.8mm] (-0.5,0)--(0.5,0); 
\draw[line width=0.8mm] (-0.2,-0.2)--(-0.2,0.2);
\draw[line width=0.8mm] (0.2,-0.2)--(0.2,0.2);

\coordinate (rd1) at (-0.2,0.2);
\coordinate (rd0) at (-0.2,-0.2);

  \end{scope}

\draw[line width=0.3mm, blue] (raa) edge node[blue,above ] {$\tilde \beta$} (rcc);

 \draw[line width=0.3mm, red] (rb0) edge[bend right=10] node[above right] {$\tilde \alpha_2$}  (ra1);

 \draw[line width=0.3mm, red] (rc0) edge[bend right=10] node[above left] {$\tilde \alpha_3$}  (rb1);

 \draw[line width=0.3mm, red] (rd0) edge[bend right=10] node[below left] {$\tilde \alpha_4$}  (rc1);
 
  \draw[line width=0.3mm, red] (ra0) edge[bend right=10] node[below right] {$\tilde \alpha_1$}  (rd1);

\end{scope}
\end{tikzpicture}\]

\caption{An invariant arc diagram $A=\{\alpha_1,\alpha_2,\alpha_3, \alpha_4, \beta\}$ is depicted for the $1/4$ Rabbit map tuned with the Feigenbaum map. Here, $\ExpN_A(f,\alpha_i)=1$ while $\ExpN_A(f,\beta)=2$ because $\beta$ must travel through $\tilde \alpha_2$ and $\tilde \alpha_1$ ($\beta$ can not cross-intersect its lift $\widetilde \beta$). Note that if we replace in $A$ the arc $\beta$ with its orthogonal arc $\beta^\perp$ connecting $\bush^{[1]}_0$ and $\bush^{[1]}_2$, then the new arc diagram $A^\new$ will not be invariant: $\widetilde \beta^\perp=\beta$ will completely block $\beta^\perp$.}

\label{Fig:PeriodicArcs}
\end{figure}
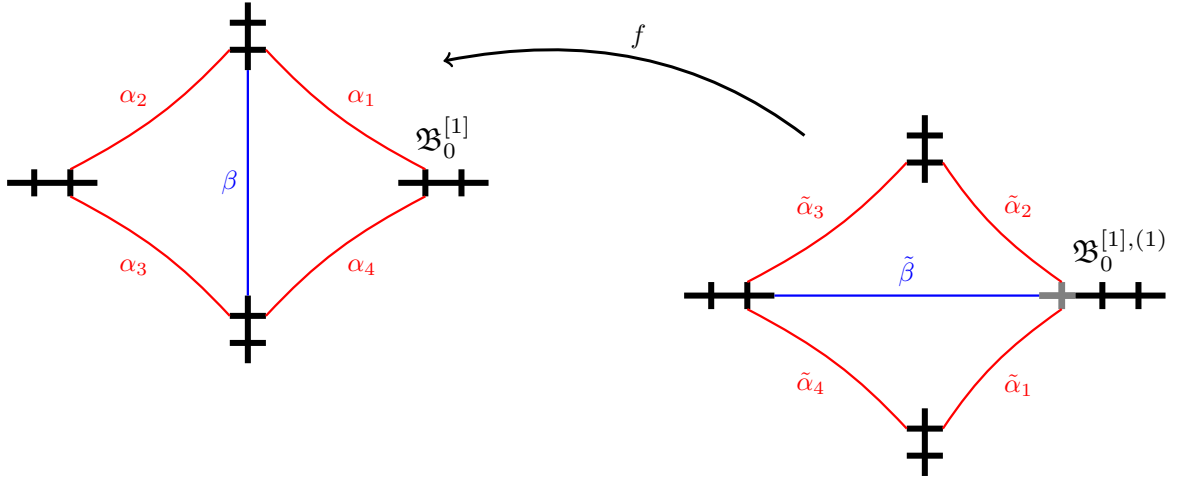

\begin{lem}[Expansivity of $  f\mid {\big[\Hub_f\setminus \bbush^{[n]}\big]}$] 
\label{lem:per arcs}Consider an invariant AD $A$ and an arc $\gamma\in A$ aligned with a proper geodesic continuum $T_\gamma$ of $\Hub_f\setminus \bbush^{[n]}$. Then
\begin{itemize}
\item $\gamma$ is genuine periodic if and only if $T_\gamma $ connects two neighboring bushes $\bush^{[n]}_a, \bush^{[n]}_{a+1}$ (with respect to the cyclic order) of a periodic satellite flower $ \bush^{[n-1]}_j\supset \bush^{[n]}_a, \bush^{[n]}_{a+1}$ of level $n-1$. In particular, the $n$-th renormalization of $f$ is satellite.
\item if $\gamma$ is not genuine periodic, then 
\begin{equation}
\label{eq:lem:per arcs} \ExpN_A(f^{p_n}, \gamma) \ge 2.
\end{equation}
\end{itemize}
\end{lem}
See Figure~\ref{Fig:PeriodicArcs} for illustration.
\begin{proof}Clearly, if $T_\gamma$ is in a satellite flower $\bush^{[n-1]}_j$ and connects two neighboring bushes $\bush^{[n]}_a,\ \bush^{[n]}_{a+1}$, then $\gamma$ is genuine periodic. It is also clear that~\eqref{eq:lem:per arcs} holds unless $T_\gamma$ is in a satellite flower $\bush^{[n-1]}_j$. (It can be easily proven for a superattracting model~\S\ref{sss:psib:super attr model}.)

Assume $T_\gamma$ is in a satellite flower $\bush^{[n-1]}_j$ but connects two non-neighboring bushes $\bush^{[n]}_a,\ \bush^{[n]}_b,\ |a-b|>1$. Then a certain lift $T'$ of $T_\gamma$ under $f^{m}$ with $m=kp_{n-1}<p_n$ cross-intersects $T_\gamma$.  It follows that $\big(f^{m}\big)^* A$ has an arc aligned with $T'$ but has no arc aligned with $T_\gamma\setminus \left( \bush^{[n],(m)}_a\cup \bush^{[n],(m)}_b\right)$. This implies that  $\ExpN_A(f^{m}, \gamma) \ge 2$.

\end{proof}

\subsubsection{Arcs inside and outside of satellite flowers}\label{sss:in out: sat flowers} Consider an invariant AD $A$ on $X\setminus \bbush^{[n]}$. Consider an arc $\alpha\in A$ aligned with a geodesic continuum $T_\alpha\subset \Hub_f\setminus \bbush^{[n]}$. We say that $\alpha$ is in $\bush^{[n-1]}_c$ if $T_\alpha\subset \bush^{[n-1]}_c$.

\begin{lem}
\label{lem:sss:in out:sat flower}
Consider an invariant arc diagram $A$ on $X\setminus \bbush^{[n]}$ and an arc $\alpha\in A$. Assume that $\alpha$ is not in any satellite flower $\bush^{[n-1]}_c$ of level $n-1$. Then there are two arcs \[\alpha_1,\alpha_2\in (f^{2p_n})^* (A)\hspace{0.6cm} \text{ with }\sp f^{2p_n}(\alpha_1)=f^{2p_n}(\alpha_2)=\alpha_\new \in A\]
such that $\alpha$ overflows $\alpha_1, \alpha_2$ in the following sense: any decomposition~\eqref{eq:form for gamma} of $\alpha=[\gamma]$ rel $\bbush^{[n],(2p_n)}$ contains two paths $\gamma_a, \gamma_b$ representing $\alpha_1=[\gamma_a]$ and $\alpha_2=[\gamma_b]$.
\end{lem}
\begin{proof}
Since $T_\alpha$ is not in any satellite flower $\bush^{[n-1]}_c$, there is a strictly preperiodic component $\bush^{[n],(p_n)}_d$ of $\bbush^{[n],(p_n)}$ intersecting $T_\alpha$. There are two components $T_1,T_2\subset T_\alpha\setminus \bbush^{[n],(2p_n)}$ adjacent to $\bush^{[n],(2p_n)}_d$. Since there is an $m\le 2p_n$ such that
$f^m\big(\bush^{[n],(2p_n)}_d\big)$ contains the critical value, $T_1,T_2$ have a common injective image of generation $m\le 2p_n$. There must be two arcs $\alpha_1,\alpha_2\in  (f^{2p_n})^* (A)$ aligned with $T_1,T_2$. 
\end{proof}

Assume that the $n$th renormalization is satellite and consider a satellite flower $\bush^{[n-1]}_c$. Let $A$ be an invariant AD as above. Let $A_c=A(\bush^{[n-1]}_c)\subset A$ be the AD consisting of all arcs from $A$ that are in $\bush^{[n-1]}_c$. Similarly, $A^{(p_n)}_c\subset (f^{p_n})^*(A)$ be the AD consisting of all arcs from $ (f^{p_n})^*(A)$ that are in $\bush^{[n-1]}_c$. Then 
\begin{equation}
\label{eq:invarirance:A_c}
(f^{p_n})^*\colon A_c^{p_n}\overset{1:1}{\longrightarrow}  A_c
\end{equation}
is a bijection and every arc $\alpha\in A_c$ is homotopic to its unique lift $\tilde \alpha\in A_c^{p_n}, (f^{p_n})^*(\tilde \alpha)=\alpha$ in the following sense: $T_{\alpha}\setminus \bbush^{[n],(p_n)}=T_{\tilde \alpha}$.
\subsubsection{Inhomogeneous configuration} \label{sss:mixed conf} The discussion here will be only used in~Section~\ref{s:conclusions}. Consider a ql map $f\colon X\to Y$ with a cycle of bushes $\bbush^{[1]}$. Write $\bUpsilon\coloneqq \bbush^{[1]}\cup \filled^{[1]}_0$. Note that $f$ does not permute components of $\bUpsilon$.

An AD $A=\{\alpha_i\}$ on $X\setminus \bUpsilon$ is \emph{horizontal} if every $\alpha_i$ connects components of $\bUpsilon$. A horizontal AD $A=\{\alpha_i=[\gamma_i]\}$ is \emph{$f^{p_1}$-invariant} if every $\alpha_i$ can be represented up to a proper homotopy in $X\setminus \bUpsilon$ in an arbitrary small neighborhood of $f^{-p_1}\left(\bUpsilon\cup \Gamma \right)$, where $\Gamma=\bigcup_i \gamma_i$.  By replacing $f$ with its superattracting model (\S\ref{sss:psib:super attr model}, compare also with Lemma~\ref{lem:Alignment}), we obtain:

\begin{lem}[Alignment with $\Hub_f$]\label{lem:Alignment:Ups} If $A=\{\alpha_i\}$ is an $f^{p_1}$-invariant AD on $Y\setminus \bUpsilon$ as above, then $A$ is \emph{aligned with the Hubbard continuum} $\Hub_f$: every $\alpha_i$ can be represented in an arbitrary small neighborhood of a geodesic continuum $T_i$ of $\Hub_f\setminus \bUpsilon$ connecting components of $\bUpsilon.$\qed
\end{lem}

For $m\ge 1$ and $\alpha\in A$, the \emph{expansivity number} $\ExpN_A(f^{mp_1}, \alpha)$ is defined in the same way as in \S~\ref{sss:periodic arcs}; i.e., it is the smallest number of arcs in $f^{-p_1m}\left( A\right)$ overflown by $\alpha$.  An arc $\alpha\in A$ is \emph{genuine periodic} if $\ExpN_A(f^{p_1m},\alpha)=1$ for all $m\ge 1$. By the same argument as in Lemma~\ref{lem:per arcs}, we have:

\begin{lem}[Expansivity of {$  f\mid [\Hub_f\setminus \bUpsilon]$}] 
\label{lem:per arcs:ups}
Consider an $f^{p_1}$-invariant AD $A$ on $Y\setminus \bUpsilon$ and an arc $\gamma\in A$ aligned with a proper geodesic continuum $T_\gamma$ of $\Hub_f\setminus \bUpsilon$. If the first renormalization of $f$ is primitive, then
\begin{equation}
\label{eq:lem:per arcs:ups} \ExpN_A(f^{p_1}, \gamma) \ge 2.
\end{equation}\qed
\end{lem}

\subsection{Teichm\"uller contraction}
\label{ss:Teichm contra} The Teichm\"uller contraction comes from the observation that the restriction of a qc conjugacy to deeper renormalization levels can only decrease the dilatation. Therefore, the renormalization orbits (for bounded-type Feigenbaum families) can escape to infinity with at most linear rate. This fact can be traced back to Sullivan~\cite{S:Berkeley}, where structures of the Teichm\"uller spaces (reminiscent to the Thurston's machinery) were brought into the Renormalization Theory.  The proposition below is stated for quadratic maps; the Straightening Theorem easily extends it quadratic-like maps.

%We will deduce the Teichm\"uller contraction from the Straightening Theorem combined with the compactness of the Feigenbaum family of type $\overline p$. The idea can be traced back to Sullivan~\cite{S:Berkeley}, where structures of the Teichm\"uller spaces (reminiscent to the Thurston's machinery) were brought into the Renormalization Theory.  The proposition below is stated for quadratic maps; the Straightening Theorem easily extends it quadratic-like maps.

\begin{prop}
\label{prop:Teichm contr}
For every combinatorial bound $\bar p$, there is a constant $\Delta=\Delta_{\overline p}>1$ such that the following holds. Let $f_c(z)= z^2+c$ be an infinitely renormalizable quadratic polynomial of bounded type $\overline p$, see~\S\ref{sss:infin renorm}. Then $g_n=\RR^n f_c$ has ql prerenormalization $g_n=f_c^{p_n}\colon X_n\to Y_n$, $X_n\Subset Y_n$ such that
\begin{enumerate}[label=\text{(\Roman*)},font=\normalfont,leftmargin=*]
\item $\Width(Y_n\setminus X_n)=O( \Delta^n)$; and\label{prop:Teichm contr:Cond 1}
\item  $Y_n\setminus \filled_{g_n}$ is disjoint from $\bbush^{[n]}_{f_c}$.\label{prop:Teichm contr:Cond 2}
\end{enumerate}
\end{prop}
\begin{proof} 
Let $\MM_i, i \in I$ be the (finite) set of all maximal $\MM$-copies in $\MM$ of period $\le \bar p$, and let 
\[R\colon \bigcup _{i\in I}\MM_i\to \MM\]
be the ql straightening map. Let us remove from every satellite $\MM_i$ a  small open neighborhood of its cusp such that the remaining $\MM^\circ_i$ contains all the preimages of the $\MM_j$ under $R\colon \MM_i\to \MM$. For a primitive copy $\MM_i$, set $\MM_i^\circ\coloneqq \MM_i$. By compactness, there is a $\Delta>1$ such that every $f_c$ with $c\in \bigcup _i \MM_i^\circ$ has a ql restriction $g_1=f_c^{p_1}\colon X_{1,c}\to Y_{1,c}$ around its maximal little Julia set containing $0$ such that $g_1$ satisfies \ref{prop:Teichm contr:Cond 2} and such that $g_1$ is hybrid conjugate via $h_{c}\colon Y_{1,1}\to Y_{R(c)}$ with dilatation $\le \Delta$ to a ql restriction $f_{R(c)}\colon X_{R(c)}\to Y_{R(c)}$. Moreover, we can select the $Y_{1,c}$ and $Y_c$ so that $Y_c\Supset Y_{1,c}$.

Wright $c_n \coloneqq R^n(c)$. Then \[\bar h\coloneqq \big(h_{c_{n-1}}\circ \dots \circ h_{c_1}\circ h_{c}\big)^{-1}\] is a hybrid conjugacy from $f_{c_n}\colon X_{c_n}\to Y_{c_n}$ to a ql restriction $g_n=f^{p_1\dots p_n}_c\colon X_{n,c}\to Y_{n,c}$ satisfying \ref{prop:Teichm contr:Cond 2}. Since the dilatation of $\overline h$ is $\le \Delta^n$, Property \ref{prop:Teichm contr:Cond 1} follows.
\end{proof}

\begin{rem}
\label{rem:prop:Teichm contr}
Property~\ref{prop:Teichm contr:Cond 2} implies that the prerenormalization $g_n\colon X_n\to Y_n$ is unbranched: $\Post(f_c)\cap Y_n\subset \filled(g_n)$, compare with~\cite{McM2}. Moreover, by induction, $Y_n$ is disjoint from 
\[\Upsilon\coloneqq \bigcup_{m\le n} \left(\bbush^{[m]}_f\setminus \bush^{[m]}_{f,0} \right).\]
This implies that the quadratic-like prerenormalizations $g_n\colon X_n\to Y_n$ can be univalently lifted to the dynamical plane of any $\psi^\bullet$ renormalization \[\RR^{n_1\bullet}\circ\RR^{n_2\bullet}\dots\dots \circ \RR^{n_s\bullet} (f_c)\]of 
$f_c$, see~\S\ref{sss:psi b renorm}.
\end{rem}

\section{$\psi^\bullet$-ql renormalization and near-degenerate regime}
 In this section, we will introduce $\psi^\bullet$-ql renormalization and discuss tools to detect its degeneration; see~\cite{A,K, covering lemma, L:book, DL2} for details.

 Given a compact subset $K\Subset S$, we denote by $\Fam(S,K)$ the family of non-trivial proper curves in $S\setminus K$ connecting $\partial K$ to $\partial S$. We write
 \begin{equation}
\label{eq:width:dfn}
 \Width(S,K) =\Width \big(\Fam(S,K)\big). 
 \end{equation}

\subsection{Outline} Consider a ql map $f\colon X\to Y$ and let $\bbush^{[n]}$ be its level $n$ cycle of little bushes.

Around every periodic $\filled^{[n]}_i$ there is an associated ql prerenormalization 
\begin{equation}
\label{eq:f_n,i}
f_{n,i}=f^{p_n}\colon X_{n,i}\to Y_{n,i}.
\end{equation}
 In~\S\ref{psi-renorm}, we will define the $\psi^\bullet$-ql renormalization 
\begin{equation}
\label{eq:F_n,i}
F_{n,i} = (f_{n,i},\ \iota_{n,i})\colon U_{n,i} \rightrightarrows V_{n,i}
\end{equation}
of $f\colon X\to Y$ by \emph{extending \eqref{eq:f_n,i} along all curves} in $Y\setminus \bbush^{[n]}$, see~\S\ref{sss:psi b renorm}. This is similar to the $\psi$-ql renormalization~\cite{K}, where the extension is performed along all curves in $Y\setminus  \bfilled^{[n]}$ under the assumption that the $\filled^{[n]}_i$ are pairwise disjoint  (the primitive case); see~\S\ref{sss:psi renorm: motiv} and Remark~\ref{rem:general constr}.

The correspondence $F_{n,i}$ in~\eqref{eq:F_n,i} is called a $\psi^\bullet$-ql map. It consists of a degree $2$ branched covering $f_{n,i}$ and an immersion $\iota_{n,i}$, see definitions in~\S\ref{sss:qlb:Defn}. We define the Julia set and bush of $F_{n,i}$ to be that of the ql map $f_{n,i}$~\eqref{eq:f_n,i}:
\[ \filled_{F_{n,i}}\coloneqq \filled_{f_{n,i}}\sp\sp\text{ and }\sp\sp \bush_{F_{n,i}}\coloneqq \bush_{f_{n,i}} .\]
Similarly, $\filled_F$ and  $\bush_F$ are defined for any $\psi^\bullet$-ql map, see~\eqref{eq:dfn:non-esc b}.

A $\psi^\bullet$-ql renormalization can be naturally defined for a $\psi^\bullet$-ql map; i.e., $\psi^\bullet$-ql renormalization can be iterated. The Sup-Chain Rule for renormalization domains is stated in~\S\ref{sss:sup chain rule}.

Just like a ql map $f\colon X\to Y$ can be restricted to $f\colon X^{k+1}\to X^k$, where $X^{k+1}=f^{-k}(X)$, a $\psi^\bullet$-ql map $F=(f,\iota )\colon U\rightrightarrows V$
can be restricted to \[F=(f,\iota )\colon U^{k+1}\rightrightarrows U^{k}\] by considering the fiber product~\S\ref{sss:FibProd}. We have a natural $2^{k}:1$ branched covering $f^k\colon U^k\to V$ and an immersion $\iota^k\colon U^k\to V$.

 The \emph{width} of a $\psi^\bullet$-ql map $F\colon U\rightrightarrows V$ is 
\[\Width_\bullet(F)\coloneqq \Width(V\setminus \bush_F).\] 
We denote by $\AAA^{[n]}\equiv \AAA^{[n]}_F$ the weighted arc diagram (WAD) of $V\setminus \bbush^{[n]}$: it is a formal sum of weighted arcs representing rectangles in the thick-thin decomposition of $V\setminus \bbush^{[n]}$; see \S\ref{ss:WAD} and Figure~\ref{ss:WAD}. The \emph{horizontal part}  $\AAA^{[n]}_\hor$ of $\AAA^{[n]}$ consists of weighted arcs connecting components of $\bbush^{[n]}$.  The \emph{vertical part} $\AAA^{[n]}_\ver$ of $\AAA^{[n]}$ consists of weighted arcs connecting $\partial V$ and components of $\bbush^{[n]}$. The \emph{local WAD} $\AAA^{[n]}_i$ consists of arcs in $\AA^{[n]}$ adjacent to $\bush^{[n]}_i$, where the weight of self-arcs to $\bush^{[n]}_i$ is doubled. More generally, $\AA^{[n],(m),k}$ is the WAD of $U^k\setminus \bbush^{[n],(m)}$, and 
the WADs $\AA^{[n],(m),k}_\hor, \AA^{[n],(m),k}_i,\dots$ are defined accordingly, see~\S\ref{sss:WADs}.  We have 
\[\Width(F_{n,i})=\Width\left(\AA^{[n]}_i\right)+O_{p_n}(1).\]

By \emph{identifying up to homotopy} $U^k\setminus \bbush^{[n],(m)}$ and $U^{k+1}\setminus \bbush^{[n],(m)}$, we obtain $\AA^{[n],(m),k+1}_\hor\le\AA^{[n],(m),k}_\hor$, see~\S\ref{sss:monot:AA^k}. Since the complexity of $A^{[n],k}_\hor\equiv\AD\left[\AA^{[n],k}_{\hor}\right]$ decreases, the $A^{[n],k}_{\hor}$ are essentially invariant for $k\ge 3 p_n$. Then either most of $\AA^{[n]}_\hor$ is in $\AA^{[n],k}_\hor$ or a substantial part of $\AA^{[n]}_\hor$ is in $\AA^{[n],k}_\ver$ -- this is a key dichotomy in Section~\ref{s:Pull off for non-periodic rect}; see also the dichotomy~``\ref{C1:intro} vs~\ref{C2:intro}'' in~\S\ref{ss:intro:SatCase}.

 \subsection{Rectangles} \label{ss:rectangles} A \emph{Euclidean rectangle} is a rectangle $E_x\coloneqq[0,x]\times [0,1] \subset \C$, where:
\begin{itemize}
\item $(0,0), (x,0), (x,1), (0,1)$ are four vertices of $E_x$,
\item $\partial^h E_x=[0,x]\times\{0,1\}$ is the horizontal boundary of $E_x$,
\item $\partial^{h,0} E_x\coloneqq [0,x]\times\{0\}$ is the \emph{base} of $E_x$,
\item $\partial^{h,1} E_x\coloneqq [0,x]\times\{1\}$ is the \emph{roof} of $E_x$,
\item $\partial^v E_x=\{0,x\}\times [0,1]$ is the \emph{vertical} boundary of $E_x$,
%\item $\partial^{v,\ell} E_x\coloneqq \{0\}\times [0,1], \sp  \partial^{v,\rho} E_x\coloneqq \{x\}\times [0,1]$ is the \emph{left} and \emph{right vertical} boundaries of $E_x$; 
\item $\Fam(E_x)\coloneqq \{\{t\}\times[0,1]\mid t\in [0,x]\}$ is the \emph{vertical foliation} of $E_x$,
\item $\Fam^\full(E_x)\coloneqq \{\gamma\colon [0,1]\to E_x \mid \gamma(0)\in \partial ^{h,0}E_x,\ \gamma(1)\in \partial ^{h,1}E_x \}$ is the \emph{full family of curves} in $E_x$; 
\item $\Width(E_x)=\Width(\Fam(E_x))=\Width\big(\Fam^\full(E_x)\big)=x$ is the \emph{width} of $E_x$,
\item $\mod (E_x)=1/\Width(E_x)=1/x$ the extremal length of $E_x$.
\end{itemize}
%The width of the vertical foliation $\Fam(E_x)$ is equal to the width of all the curves in $E_x$ connecting $\partial^{h,0} E_x$ and $\partial^{h,1} E_x$.

By a \emph{(topological) rectangle} in a Riemann surface we mean a closed Jordan disk $\RR$ together with a conformal map $h\colon \RR\to E_x$ into the standard rectangle $E_x$.  The vertical foliation $\Fam(\RR)$, the full family $\Fam^\full(\RR)$, the base $\partial^{h,0}\RR$, the roof $\partial^{h,1}\RR$, the vertices of $\RR$, and other objects are defined by pulling back the corresponding objects of $E_x$. Equivalently, a rectangle $\RR$ is a closed Jordan disk together with four marked vertices on $\partial \RR$ and a chosen base between two vertices.

A \emph{genuine subrectangle} of $E_x$ is any rectangle of the form $E'=[x_1,x_2]\times [0,1]$, where $0\le x_1<x_2\le x$; it is identified with the standard Euclidean rectangle $[0,x_2-x_1]\times [0,1]$ via $z\mapsto z- x_1$. A genuine subrectangle of a topological rectangle is defined accordingly. 

A \emph{subrectangle} of a rectangle $\RR$ is a topological rectangle $\RR_2\subset \RR$ such that $\partial^{h,0}\RR_2\subset \RR$ and $\partial^{h,1}\RR_2\subset \RR$. By monotonicity: $\Width(\RR_2)\le \Width(\RR)$.

Assume that $\Width(E_x)>2$. The \emph{left and right $1$-buffers} of $E_x$ are defined \[ B^\ell_1\coloneqq [0,1]\times [0,1]\sp\text{ and }\sp B^\rho_1\coloneqq [x-1,x]\times [0,1]\] respectively. We say that the rectangle 
\[E^\new_x\coloneqq [1,x-1]\times [0,1]=E_x\setminus \big( B^\ell_1 \cup B^\rho_1\big)\] is obtained from $E_x$ by \emph{removing $1$-buffers}. If $\Width(E_x)\le 2$, then we set $E_x^\new\coloneqq\emptyset$. Similarly, buffers of any width are defined.

\subsubsection{Monotonicity and Gr\"otzsch inequality} \label{sss:Monot+Gr}We say a family of curves $\SS$ \emph{overflows} a family $\FamG$ if every curve $\gamma\in\SS$ contains a subcurve $\gamma'\in \FamG$. We also say that 
\begin{itemize}
\item a family of curves $\Fam$ \emph{overflows} a rectangle $\RR$ if $\Fam$ overflows $\Fam^\full(\RR)$;
\item a rectangle $\RR_1$ overflows another rectangle $\RR_2$ if $\Fam(\RR_1)$ overflows $\Fam^\full(\RR_2)$.
\end{itemize}

If $\Fam$ overflows a family or a rectangle $\FamG$, then $\FamG$ is wider than $\Fam$:
\begin{equation}
\label{eq:Width Monot}
\Width(\Fam) \le \Width(\FamG).
\end{equation}
If $\Fam$ overflows both $\FamG_1,\FamG_2$, and $\FamG_1,\FamG_2$ are disjointly supported, then the \emph{Gr\"otzsch} inequality states:
\begin{equation}
\label{eq:Grot}
\Width(\Fam) \le \Width(\FamG_1)\oplus \Width(\FamG_2),
\end{equation}
where $x\oplus y =(x^{-1}+y^{-1})^{-1}$ is the harmonic sum.

\subsection{Weighted arc diagrams} \label{ss:WAD} Let us recall the notion of the Weighted Arc Diagram (WAD) describing wide rectangles in the thick-thin decomposition of a Riemann surface with boundary, see Figure~\ref{Fg:TTD} for illustration and brief summary. We also recall that WAD were abstractly defined in~\S\ref{sss:AD}.

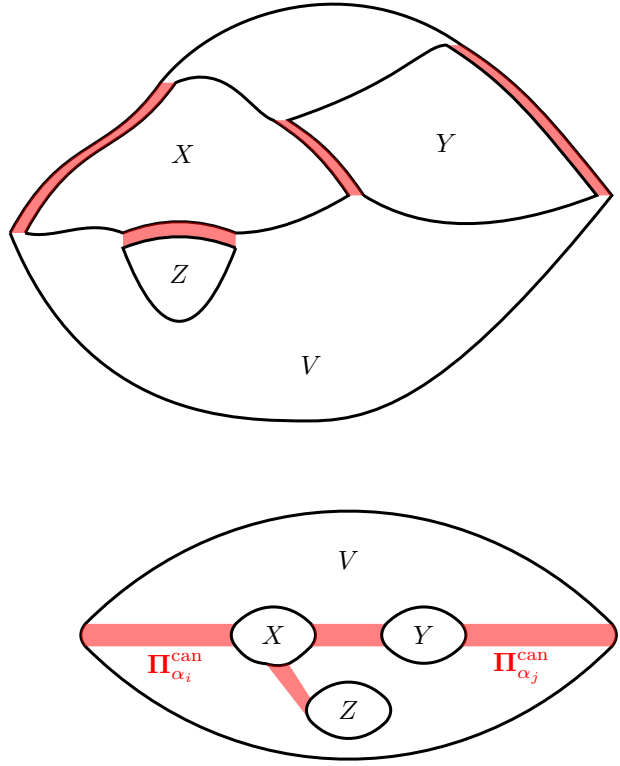
\begin{figure}[t!]
\[\begin{tikzpicture}

\draw[,line width=0.4mm] (0,0)
 .. controls  (0.7, 1.3) and (1.3,1) .. 
 (2,2)
  .. controls  (2.8, 3) and (4.6,3.5) .. 
(6,2.5)
.. controls  (6.8, 2) and (7.2,1.5) .. 
(8,0.5)
.. controls  (6, -2) and (5,-2.5) .. 
(4,-2.5)
.. controls  (3, -2.5) and (1,-2.5) .. 
(0,0);

\node[above] at (4,-2) {$V$};

 \draw[opacity=0, fill=red, fill opacity=0.5]
(6,2.5)
.. controls  (6.8, 2) and (7.2,1.5) .. 
(8,0.5)
.. controls  (8,0.5) and (7.8,0.5).. 
(7.8,0.5)
.. controls   (7,1.5)  and (6.6, 2).. 
(5.8,2.5);
 
 \draw[,line width=0.4mm] (0.2,0)
 .. controls  (0.9, 1.3) and (1.5,1) .. 
 (2.2,2)
 .. controls  (2.9, 2.3) and (3.2,1.6) .. 
 (3.5,1.5)
 .. controls  (3.8, 1.3) and (4.1,1.1) .. 
 (4.5,0.5)
.. controls  (4, 0.2) and (3.5,0) .. 
 (3,0)
 .. controls  (2.5, 0.2) and (2,0.2) .. 
(1.5,0)
 .. controls  (1, 0.2) and (0.6,-0.1) .. 
(0.2,0);

\node[below] at (2.3, 1.3) {$X$};

 \draw[opacity=0, fill=red, fill opacity=0.5]
(3,0)
 .. controls  (2.5, 0.2) and (2,0.2) .. 
(1.5,0)
.. controls  (1.5, 0) and (1.5,-0.2) .. 
(1.5,-0.2)
 .. controls  (2, 0.) and (2.5,0) .. 
(3,-0.2);

\draw[,line width=0.4mm] (3,-0.2)
 .. controls  (2.5, 0.) and (2,0) .. 
(1.5,-0.2)
 .. controls  (2, -1.5) and (2.5,-1.5) .. 
(3,-0.2);

\node[below] at (2.25, -0.3) {$Z$};

\draw[opacity=0, fill=red, fill opacity=0.5]
(0,0)
 .. controls  (0.7, 1.3) and (1.3,1) .. 
 (2,2)
  .. controls  (2, 2) and (2.2,2) .. 
(2.2,2)
 .. controls (1.5,1) and (0.9, 1.3)  .. 
(0.2,0);

 \draw[,line width=0.4mm] (4.7,0.5)
 .. controls  (4.3,1.1) and (4, 1.3)  .. 
 (3.7,1.5)
 .. controls  (5.2, 2) and (5.5,2.5) .. 
 (5.8,2.5)
.. controls  (6.6, 2) and (7,1.5) .. 
(7.8,0.5)
.. controls  (6.5, 0) and (5.5,0) .. 
(4.7,0.5);
\node at (5.8, 1.2){$Y$};

\draw[opacity=0, fill=red, fill opacity=0.5]
 (3.5,1.5)
 .. controls  (3.8, 1.3) and (4.1,1.1) .. 
 (4.5,0.5)
  .. controls  (4.5, 0.5) and (4.7,0.5) .. 
 (4.7,0.5)
  .. controls   (4.3,1.1)  and (4, 1.3).. 
(3.7,1.5) ;

\begin{scope}[shift={(1,-5.5)}]

 \draw[,line width=0.4mm] (0,0)
  .. controls  (-0.08, 0.1) and (-0.08,0.2) .. 
(0,0.3)
  .. controls  (2, 2.3) and (5,2.3) .. 
(7,0.3)
  .. controls  (7.08, 0.2) and (7.08,0.1) .. 
(7,0.)
  .. controls  (5, -2) and (2,-2) .. 
(0,0)
;

\draw[opacity=0, fill=red, fill opacity=0.5]
(0,0)
  .. controls  (-0.08, 0.1) and (-0.08,0.2) .. 
(0,0.3)
  .. controls (0,0.3) and (0,0.3) .. 
(2,0.3)
  .. controls  (2-0.08, 0.2) and (2-0.08,0.1) .. 
(2,0)
;

\draw[,line width=0.4mm] (2,0)
  .. controls  (2-0.08, 0.1) and (2-0.08,0.2) .. 
(2,0.3)
  .. controls  (2.3, 0.6) and (2.7,0.6) .. 
(3,0.3)
  .. controls  (3.08, 0.2) and (3.08,0.1) .. 
(3,0.)
  .. controls  (2.8, -0.2) and (2.7, -0.23)..%(2.3,-0.3) .. 
(2.7,-0.23)
.. controls  (2.6, -0.27) and (2.5, -0.26)..%here
(2.4,-0.23)
.. controls (2.3,-0.22) and  (2.15,-0.2) ..
(2,0);

\node at (2.5,0.15) {$X$};

\draw[opacity=0, fill=red, fill opacity=0.5]
(2.7,-0.23)
.. controls  (2.6, -0.27) and (2.5, -0.26)..%here
(2.4,-0.23)
.. controls  (2.4,-0.23) and (2.4,-0.23)..
(3,-1)
  .. controls  (3-0.08, 0.1-1) and (3-0.08,0.2-1) .. 
(3,0.3-1);

\draw[,line width=0.4mm,shift={(1,-1)}] (2,0)
  .. controls  (2-0.08, 0.1) and (2-0.08,0.2) .. 
(2,0.3)
  .. controls  (2.3, 0.6) and (2.7,0.6) .. 
(3,0.3)
  .. controls  (3.08, 0.2) and (3.08,0.1) .. 
(3,0.)
  .. controls  (2.7, -0.3) and (2.3,-0.3) .. 
(2,0);

\node at (3.5,0.15-1) {$Z$};

\node at (3.5,0.15+1) {$V$};

\draw[opacity=0, fill=red, fill opacity=0.5]
(3,0.3)
  .. controls  (3.08, 0.2) and (3.08,0.1) .. 
(3,0.)
  .. controls (3,0.) and (3,0.) .. 
(4,0)
  .. controls  (4-0.08, 0.1) and (4-0.08,0.2) .. 
(4,0.3);

\draw[,line width=0.4mm,shift={(2,0)}] (2,0)
  .. controls  (2-0.08, 0.1) and (2-0.08,0.2) .. 
(2,0.3)
  .. controls  (2.3, 0.6) and (2.7,0.6) .. 
(3,0.3)
  .. controls  (3.08, 0.2) and (3.08,0.1) .. 
(3,0.)
  .. controls  (2.7, -0.3) and (2.3,-0.3) .. 
(2,0);

\node at (4.5,0.15) {$Y$};

\node[red] at (1.2,-0.28) {$\bPi^\can_{\alpha_i}$};
\node[red] at (5.8,-0.28) {$\bPi^\can_{\alpha_j}$};

\draw[opacity=0, fill=red, fill opacity=0.5]
(5,0.)
  .. controls  (5.08, 0.1) and (5.08,0.2) .. 
(5,0.3)
  .. controls (4,0.3) and (4,0.3) .. 
(7,0.3)
  .. controls  (7.08, 0.2) and (7.08,0.1) .. 
(7,0.);

\end{scope}
\end{tikzpicture}\]

\caption{Top: the thin-thick decomposition of $V\setminus (X\cup Y\cup  Z)$. As we are interested in homotopy properties of wide rectangles, they will usually be depicted
as narrow rectangles (bottom) and denoted by $\bPi^\can_{\alpha_i}$, where $\alpha_i$ is an arc representing the homotopy class of the rectangle $\bPi^\can_{\alpha_i}$. The Weighted Arc Diagram (WAD) is the formal sum $\AA=\sum_i \Weight(\bPi^\can_{\alpha_i})\alpha_i$, while the Arc Diagram (AD) is the associated set of curves $A=\{\alpha_i\}\equiv\AD[\AA]$. Thin annuli of the thick-thin decomposition are disregarded in the WAD.}
\label{Fg:TTD}
\end{figure}

\subsubsection{Proper rectangles} A \emph{proper rectangle} $R$ in an open hyperbolic surface $S$ is a rectangle in its ideal completion $\overline S =S \cup \partial^i S$ such that $\partial^h R\subset \partial^i S$. We naturally view vertical curves of $R$ as proper (and open) paths in $S$. We say that $R$ \emph{connects} boundary components  $J_0, J_1\subset \partial^i S $ if $\partial ^{h,0}R\subset J_0,\  \partial^{h,1}R\subset J_1$.

 If $S=U\setminus K$, where $U$ is a topological disk and $K\Subset U$ is a compact subset with finitely many connected components, then $\partial^i S$ has one \emph{outer} component $\partial^i U $; all remaining \emph{inner} components are parameterized by the components of $K$. If $J_0, J_1$ represent the boundaries of components $K_0, K_1$ of $K$, then we say that the above $R$ connects $K_0,K_1$.

\subsubsection{Universal cover} For compact Riemann surface  $V$ as in~\S\ref{sss:AD}, write its non-empty boundary as $\di V = \bigsqcup_k J_k$. We assume that $V$ has negative Euler characteristic $\chi(V)<0$. We set $\pi: \Disk \ra \inter V$ to be the 
universal covering, 
and we let $\La\subset \T=\partial \Disk$ be the limit set for the  group $\De$
of deck transformations of $\pi$.  Then $\pi$ extends continuously to $\T\sm
\La$, and
$ \pi:  \ol \D \sm \La \ra  V $ 
is the  universal covering of $V$. Moreover, $\pi$ restricted to  any component of
$\T\sm \La$  gives us a universal covering of some component $J_k$ of
$\di V$.  Such a component of $\T\sm \La$ will be denoted $\BJ_k$
(usually we will need only one component for each $k$, 
so we will not use an extra label for it). 
The stabilizer of $\BJ_k$ in $\De$ is a cyclic group generated by
a hyperbolic M\"obius transformation $\tau_k$ with fixed points in  $\di \BJ_k$.

\sss {Covering annuli and local weights}
\label{sss:CovAnn LocWeights}

We let 
\begin{equation}
\label{eq:CovAnn} \A(V, J_k)\coloneqq \Disk/<\tau_k>
\end{equation}
 be the \emph{covering annulus} of $V$ corresponding to $J_k$. We call its width \[\WW(J_k)\coloneqq \WW\left[ \A(V, J_k)\right]\]
the {\it local  weight} of $J_k$. We let
$$
  \Width(V) = \sum_k \WW(J_k)
$$
be the {\it total weight} of $V$. 

% \begin{lem}
%    We have: $\WW(J_k) = \WW)\BJ_k + O(1) $.
% \end{lem}

\sss{Canonical arc diagram}
\label{sss:WAD}
  
Recall that an arc  $\alpha$ is a proper path in $V$ up to proper homotopy. 
Any arc $\alpha$ connects some components $J_k$ and $J_i$ of $\di V$
and lifts to an arc $\ba$ in $\D$ connecting some intervals $\BJ_k$
and $\BJ_i$. We let  $\WW(\balpha)= \WW_{\overline \Disk}(\BJ_k, \BJ_i)$ to be the width of all curves in $\Disk$ connecting $\BJ_k$ and $\BJ_i$. Then $\WW(\balpha)$ is also the width of the rectangle $ \BPi_\ba=\ol \D$ with horizontal sides $\BJ_k$ and $\BJ_i$. It is independent of the lift used. 

If $\WW(\balpha)> 2$, then by removing from $\BPi_\ba$ the square buffers we obtain a
rectangle $\BPi^{\can}_\ba$.  In this case we let
$$
     \WW^\can(\alpha) = \WW^\can(\balpha) = \WW(\BPi^\can_\ba)  =  \WW(\balpha)-2.
$$ By construction, the $\Delta$-orbit of $\BPi^\can_\ba$ consists of pairwise disjoint rectangles. Therefore, $\BPi^\can_\ba$ projects injectively onto a rectangle $\BPi^\can_\alpha$ in $V$. Moreover, different $\BPi^\can_{\ba_i}$ project to pairwise disjoint $\BPi^\can_{\alpha_i}$.

If $\WW(\balpha)\leq 2$ we let  $\WW^\can(\alpha) =0$ and $\BPi^\can(\alpha)\coloneqq \emptyset$. Arcs with positive weight form the {\it canonical weighted arc diagram
(WAD) on $V$:}
\begin{equation}
\label{eq:WAD:dfn}
\AAA_V\equiv \WAD (V)\coloneqq\sum_{\Width^\can(\alpha)>0} \Width^\can(\alpha)\alpha .
\end{equation}

Every $\bPi^\can(\alpha)$ supports the canonical
vertical foliation $\Fam(\bPi^\can(\alpha)$. Altogether these foliations form the
{\it canonical lamination} of $V$:
\begin{equation}
\label{eq:can laminat}
\Fam(\AA_V)\coloneqq \bigsqcup_{\alpha \in \AA_V} \Fam(\bPi^\can(\alpha)),
\end{equation}
 We set
\begin{equation}
\label{eq:can laminat:weight}\Width(\AA_V)\coloneqq \sum_{\alpha\in \AA_V}  \Width(\alpha) =\Width\big( \Fam(\AA_V)\big).
\end{equation}

The canonical \emph{arc diagram} of $V$ is the set of arcs in $\AAA_V$ 
\[A_V \coloneqq \{\alpha\mid \Width^\can (\alpha)>0\}\equiv \AD[\AA_V].\]

\subsubsection{Local WADs and the thick-thin decomposition}\label{sss:LocalWAD}

For a boundary component $J_k$ of $\partial V$, its \emph{local WAD} $\AA_{J_k}$ consists of weighted arcs of $\AA_V$ adjacent to $J_k$ such that the weights of self-arcs adjacent to $J_k$ are doubled. We have:
\[   \sum_{J_k} \AA_{J_k}= 2\AA_V .\]

See~\cite[\S 7.6]{L:book} for a reference of the following fact. 

\begin{thm}[Thin-Thick Decomposition]\label{thin-thick for S}
 For any compact Riemann surface $V$ with boundary $\di V = \bigsqcup_k J_k$,  we have:
\[\Width(J_k)  -  C  \leq   \Width(\AA_{J_k}) \leq \Width(J_k)\sp\sp \text{ for every }J_k\]
and     
\[
    \sum_{J_k} \Width(J_k)   -  C  \leq  2 \Width(\AA_V) \leq     \sum_{J_k} \Width(J_k)  
=\Width(V),    \]
 where $C$ depends only on the topological type of $V$.
\end{thm}

\subsubsection{Sub-diagrams and removing buffers}\label{sss:sub-diagram} Let $\AA\equiv\AA_V$ be the canonical WAD of $V$ as in~\S\ref{sss:WAD}, and let $B\subset A$ be a subset of arcs from $A\equiv\AD (\AA)$. Then the \emph{WAD induced by $B$} is the formal sum $\BB\coloneqq\sum_{\alpha\in B} \Width^\can(\alpha)\alpha$, compare to~\eqref{eq:WAD:dfn}. The weight $\Weight(\BB)$ of $\BB$ and its canonical lamination $\Fam(\BB)$ is defined as for $\AA$, see \eqref{eq:can laminat} and \eqref{eq:can laminat:weight}. In~\S\ref{sss:WADs}, we will specify the vertical and horizontal parts of the WAD associated with $\psi^\bullet$-ql maps.

For $B,\BB$ as above and $c>0$, we define
\[\BB-c \coloneqq \sum_{
\substack{\alpha\in B \\ 
\Width(\alpha)>c}
} (\Width^\can(\alpha)-c)\alpha,\]
\[B-c \coloneqq \AD (\BB-c);\]
i.e., we remove the $c$-weight from each arc. For a rectangle $\bPi^\can(\alpha)$ as  in~\S\ref{sss:WAD}, we denote by $\bPi^\can(\alpha)-c$ the rectangle obtained from $\bPi^\can(\alpha)$ by removing the $c/2$-buffers on each side. We call 
\[\Fam(\BB-c)\coloneqq  \bigsqcup_{\alpha \in \BB-c } \Fam(\bPi^\can(\alpha)-c)\]
the \emph{canonical lamination of $\BB-c$}. We also denote by  
\[\Fam^\full(\BB-c)\coloneqq  \bigsqcup_{\alpha \in \BB-c } \Fam^\full(\bPi^\can(\alpha)-c)\]
the full family of vertical curves within respective rectangles, see~\S\ref{ss:rectangles}. 

%\substack{c\text{ cycle}\\\text{of }\sigma_i}}

\subsubsection{Transformation rules}

\begin{lem}\label{generaltransform rules}
Let $i: S' \ra S$ be a  holomorphic map between
two compact  Riemann surfaces with boundary.
Then 

\ssk\nin $ \mathrm{(i)} $  If a boundary component $J' $ of $S'$ is mapped
with degree $d$ to a boundary component $J $ of $S$ then 
$ \WW(J')\leq d\cdot \WW(J)$;

\ssk\nin $\mathrm { (ii) } $
 If an arc $a ' $ of $S'$ is mapped
to an arc  $\alpha $ of $S$ then $\WW (\alpha') \leq \WW (\alpha)$.

\end{lem}

\begin{lem}\label{tranform rule for rectangles}
   Let $\Bi: \BPi'\ra \BPi$ be a holomorphic map between two rectangles
which extends to homeomorphisms between the respective  horizontal  sides
$\BI'\ra \BI$ and $\BJ'\ra \BJ$. Then
$$
      \WW(  \BPi' ) \leq \WW ( \BPi ). 
$$ 
\end{lem}

\subsubsection{WAD under covering}
\label{sss:WAD:cov}

Assume that $f \colon U\to V$ as a finite-degree covering between surfaces with boundaries, where $\chi(V)<0$. Then, see~\cite[Lemma 3.3]{K}
\begin{equation}
\AA_U=f^*\AA_V, \sp \text{ where }\sp  f^*\AA_V=\sum_{\alpha \in \AA_V} \Weight^\can(\alpha) f^{-1}(\alpha) 
\end{equation}
and $ f^{-1}(\alpha)$ is the set of preimages of $\alpha$.

\subsection{$\psi^\bullet$-ql renormalization}\label{psi-renorm} As in~\S\ref{renorm sec}, consider a ql map $f\colon X\to Y$ and its ql prerenormalization 
\begin{equation}
\label{eq:f n,0}
f_{n,0}=f^{p_n}\colon X_n\to Y_n.
\end{equation}

\subsubsection{Motivation}
\label{sss:psi renorm: motiv}
A major technical issue is that there is no natural choice of $Y_n$ in~\eqref{eq:f n,0} so that ${\Width (Y_n\setminus \filled_n)}$ is optimal and satisfies, in particular,~\eqref{eq:outline:2}. To handle this problem, a $\psi$-ql renormalization was introduced in~\cite{K}: assuming that~\eqref{eq:f n,0} is primitive and $Y_n\cap \bfilled ^{[n]}=\filled^{[n]}_0$, \emph{extend $f_n\colon X_n \to Y_n$ along all curves in $Y_n\setminus \bfilled^{[n]}$}; the result is a correspondence $F_n=(f_n,\iota_n)\colon U_n\rightrightarrows V_n$, where
\[V_n\coloneqq \A\left(Y\setminus \bfilled^{[n]}, \ \partial \filled^{[n]}_0\right) \bigcup _{\partial \filled^{[n]}_0} \filled^{[n]}_0\] is the covering annulus~\eqref{eq:CovAnn} of $Y\setminus \bfilled^{[n]}$ rel $\partial \filled^{[n]}_0$ glued with $\filled^{[n]}_0$ and 
\[U_n\coloneqq \A\left( X\setminus f^{-1}(\bfilled^{[n]}), \ \partial \filled^{[n]}_0\right) \bigcup _{\partial \filled^{[n]}_0} \filled^{[n]}_0\] 
is the covering annulus of $X\setminus f^{-1}\left(\bfilled^{[n]}\right)$ rel $\partial \filled^{[n]}_0$ glued with $\filled^{[n]}_0$. The correspondence $F_n$ is called a \emph{$\psi$-ql map} and it is independent of the choice of $Y_n$ in~\eqref{eq:f n,0}.  A $\psi$-ql renormalization can be naturally iterated. Moreover, $\psi$-ql bounds can be converted into ql bounds, see Lemma~\ref{lem:from psi-ql to ql}.

In this section, we will introduce $\psi^\bullet$-renormalization by replacing $\bfilled^{[n]}$ with the periodic cycle $\bbush^{[n]}$ of little bushes so that it is also applicable for satellite combinatorics. 

\begin{rem}
The viewpoint that various classes of self-correspondences ${(g,h)\colon A \rightrightarrows B}$ form interesting dynamical systems was popularized by Sullivan and later by A. Epstein in the context of deformation spaces. In the 2000s, it became apparent that self-correspondences give a natural framework to study some of the classical dynamical systems~\cite{IS,K, N}. See also~\cite{BD, T} for more recent developments.
\end{rem}

\subsubsection{Definitions}
\label{sss:qlb:Defn}
A {\em  pseudo${}^\bullet$-quadratic-like map} (``$\psi^\bullet$-ql map'') 
is a pair of holomorphic  maps 
\begin{equation}
\label{eq:dfn:psib-ql}
F=(f,\iota)\colon \sp ( U, \bush')  \rightrightarrows (V, \bush )
\end{equation} between two conformal  disks $U$, $V$
with the following properties:

\begin{enumerate}[label=\text{(\Roman*)},font=\normalfont,leftmargin=*]
\item\label{dfn:psib-ql:1}  $f\colon U\to V$ is a double branched covering
    (we usually normalize it so that its critical point is located at $0$);

\item \label{dfn:psib-ql:2} $\iota\colon U\to V$ is an immersion;

\item \label{dfn:psib-ql:3}
 $\bush$ and $\bush'$  are hulls (compact connected full sets) such that \[\iota^{-1} (\bush)=\bush'  \subset f^{-1} (\bush) ;\]

\item \label{dfn:psib-ql:4}
 there exist neighborhoods $X'\supset \bush'$ and $X \supset \bush$ with the following property: $\iota: X' \ra X$ is a conformal isomorphism such that
\begin{equation}\label{eq:dfn:ql b domain}
f_X\coloneqq f\circ \big(\iota\mid X'\big)^{-1} : X \to f(X')\eqqcolon Y 
\end{equation} 
 is a quadratic-like map with connected filled Julia set
\begin{equation}
\label{eq:dfn:non-esc b} \filled_F\coloneqq \filled\big(f_X : X\to Y\big),
\end{equation}
 and such that $\bush\equiv\bush_F\equiv \bush(f_X)$ is the bush of $f_X$.
 \end{enumerate}

Since $\iota$ is a conformal isomorphism in a neighborhood of $\filled_F$, we will below identify 
\[\filled_F\simeq \big(\iota\mid X\big)^{-1}(\filled_F)\sp\sp \text{ and hence }\sp\sp\bush\simeq \bush'\equiv \bush_F,\]
 and write  \[F\colon ( U, \bush)  \rightrightarrows (V, \bush) \sp\sp\text{ or }\sp\sp F\colon U \rightrightarrows V.\]

Let us say that a subset $\Omega\subset \filled_F$ is called \emph{$\iota$-proper} if
\[\iota^{-1}(\Omega)=\Omega.\] 
Item~\ref{dfn:psib-ql:3} implies that $\bush_F$ is $\iota$-proper. If $\filled_F$ is $\iota$-proper, then $F$ is $\psi$-ql map~\cite{K}. 

\subsubsection{$\psi^\bullet$-ql renormalization}\label{sss:psi b renorm} Consider a $\psi^\bullet$-ql map $F$ from~\eqref{eq:dfn:psib-ql}, and assume that $f_X$ (see~\eqref{eq:dfn:ql b domain}) is $n+1$ DH renormalizable. Let $f_{n,i}=f^{p_n}\colon X_{n,i}\to Y_{n,i}$ be a ql prerenormalization of $f_X\colon X\to Y$. The associated $\psi^\bullet$-ql renormalization is the extension of  $f_{n,i}\colon X_{n,i} \to Y_{n,i}$ along all curves in $V\setminus \bbush_n$, compare with~\S\ref{sss:psi renorm: motiv}; the result is a $\psi^\bullet$-ql map $F_{n,i}=(f_{n,i},\iota_{n,i})\colon U_{n,i}\rightrightarrows V_{n,i}$, where

\[V_{n,i}\coloneqq \A\left(V\setminus \bbush^{[n]}, \ \partial \bush^{[n]}_i\right) \bigcup _{\partial \bush^{[n]}_i} \bush^{[n]}_i\] is the covering annulus~\eqref{eq:CovAnn} of $V\setminus \bbush^{[n]}$ rel $\partial \bush^{[n]}_i$ glued with $\bush^{[n]}_i$ and 
\[U_{n,i}\coloneqq \A\left( U\setminus f_{n,i}^{-1}(\bbush^{[n]}), \ \partial \bush^{[n](1)}_i\right) \bigcup _{\partial \bush^{[n](1)}_i} \bush^{[n](1)}_i\] 
is the covering annulus of $U\setminus f^{-1}(\bbush^{[n]})$ rel $\partial \bush^{[n](1)}_i$ glued with $\filled^{[n]}_i$.

We will write suppress index ``$0$'' for $F_{n,0}$:
\[F_{n,0}\equiv F_{n}=(f_{n},\iota_{n})\colon U_{n}\rightrightarrows V_{n}, \hspace{0.4cm}\text{ where } \sp U_n\equiv U_{n,0},\ V_n\equiv V_{n,0}.\]

We write
\begin{equation}
\label{RR bullet f}
\RR^{n\bullet }( F) \coloneqq F_n \equiv F_{n,0}.
\end{equation}

\begin{rem}\label{rem:general constr} The theory works similarly if the bush $\bush$ is replaced with any connected forward invariant set $\Upsilon$ satisfying $\Hub\subset \Upsilon\subset \filled$, where $\Hub$  is the Hubbard continuum. Namely, let us say that a map
\begin{equation}
\label{eq:psi^* ql map}
F=(f,\iota)\colon \sp ( U,  \Upsilon')  \rightrightarrows (V, \Upsilon )
\end{equation} is \emph{$\psi^\Upsilon$-ql}
if it satisfies~\ref{dfn:psib-ql:1} -- \ref{dfn:psib-ql:4} from~\S\ref{sss:qlb:Defn} so that $\bush'\simeq \bush=\bush(f_X)$ are replaced with \[\Upsilon'=\iota^{-1} (\Upsilon)\simeq \Upsilon=\Upsilon(f_X),\hspace{1.5cm}\text{ i.e., }\Upsilon \text{ is $\iota$-proper}.\] For example, one can take $\Upsilon =\Hub$ or $\Upsilon=\Hub\cup \bbush^{[m]}$. To define ``$\psi^\Upsilon$-renormalization'', one should extend a ql prerenormalization  $f_{n,i}\colon {X_{n,i} \to Y_{n,i}}$ along all curves in $V\setminus  \bigcup_{j}  \Upsilon_{f_{n,j}}$. Moreover, $\Upsilon_{f_{n,j}}$ can often be enlarged, see~\S\ref{sss: psi^b to psi}.
\end{rem}

\begin{figure}[t!]
\[\begin{tikzpicture}
\begin{scope}
\draw[line width=0.4mm] (-0.5,0)-- (2.5,0)
(0,0.5)--(0,-0.5)
(1,0.5)--(1,-0.5)
(2,0.5)--(2,-0.5);
   
\draw[red,line width=0.8mm] (-0.3,0) -- (0.3,0)
(0,0.2)--(0,-0.2);  
\draw[red,line width=0.8mm,shift={(1,0)}] (-0.3,0) -- (0.3,0)
(0,0.2)--(0,-0.2);  
\draw[red,line width=0.8mm,shift={(2,0)}] (-0.3,0) -- (0.3,0)
(0,0.2)--(0,-0.2);
\end{scope}

\begin{scope}[shift={(4.5,0)}]
\draw[ dashed ,line width=0.3mm] (-0.5,0)-- (2.5,0)
(0,0.8)--(0,-0.8)
(1.5,0.8)--(1.5,-0.8)
(3,0.8)--(3,-0.8);
   
\draw[red,line width=0.8mm] (-0.3,0) -- (0.3,0)
(0,0.2)--(0,-0.2);  

\draw[red,line width=0.8mm,shift={(3,0)}] (-0.3,0) -- (0.3,0)
(0,0.2)--(0,-0.2);

\draw[purple,line width=1mm,shift={(1.5,0)}] 
 (-0.3,0) -- (0.3,0)
(0,0.2)--(0,-0.2);

\draw[blue,line width=0.4mm] (1.5,0.2) .. controls  (1.8, 0.2) and (2.2,0.2) .. (2.3,0)
.. controls  (3, -0.7) and (3.3, -0.8)  .. (3.6, -0.8)
.. controls (4,-0.8) and (4.5,-0.7) .. (5.,0);

\draw[blue,line width=0.4mm]
(3.4, -0.68)-- (3.6, -0.8)--(3.4, -0.88);
\node[blue, below]at (3.6, -0.9)  {$\gamma_1$};
 
\draw[blue,line width=0.4mm] (1.5,-0.2) .. controls  (1.2, -0.2) and (1.1,-0.2) .. (0.5,0)
.. controls  (0, 1.5) and (-5,1.5)  ..  (-5.5,0)
.. controls  (-5,-1.2) and  (-3,-1.6) .. (-2.5,-1.6)
.. controls  (-1,-1.6) and  (0, -1.1) .. (0,-0.8);

\node[above,blue] at (-2.5,-1.4) {$\gamma_2$};
%\draw[blue,line width=0.4mm]
\draw[blue,line width=0.4mm]
(-2.7,-1.48) --(-2.5,-1.6)--(-2.7,-1.69);

\end{scope}

\begin{scope}[shift={(10,0)}]
\draw[line width=0.4mm] (-0.5,0)-- (2.5,0)
(0,0.5)--(0,-0.5)
(1,0.5)--(1,-0.5)
(2,0.5)--(2,-0.5);
   
\draw[red,line width=0.8mm] (-0.3,0) -- (0.3,0)
(0,0.2)--(0,-0.2);  
\draw[red,line width=0.8mm,shift={(1,0)}] (-0.3,0) -- (0.3,0)
(0,0.2)--(0,-0.2);  
\draw[red,line width=0.8mm,shift={(2,0)}] (-0.3,0) -- (0.3,0)
(0,0.2)--(0,-0.2);
\end{scope}
\end{tikzpicture}\]
\caption{Illustration to Sup-Chain Rule~\eqref{eq:sup chain rule}, where $n_1=n_2=1$. Level one bushes $\bbush^{[1]}$ are depicted black and dashed black (the central bush $\bush^{[1]}_0$), while level two  bushes $\bbush^{[2]}$ are depicted red and purple (the central bush $\bush^{[2]}_0$). The map $\RR^{\bullet}\circ \RR^{\bullet}(F)$ is obtained from $F$ by extending the ql prerenormalization $f_{2,0}$ (defined near $\bush^{[2]}_0$, purple) along some paths $\gamma_i$ in $V\setminus \bUpsilon$ (where $\bUpsilon$ consists of non-dashed sets). The curve $\gamma_1$ terminates at $\bUpsilon$ (i.e., no extension beyond). The curve $\gamma_2$ terminates at $\bbush^{[1]}_0$ because it makes a non-trivial loop around $\bbush^{[1]}\setminus \bush^{[1]}_0$.  Since $\RR^{2 \bullet } F$ is obtained by extending $f_{2,0}$ along \emph{all} paths in $V\setminus \bbush^{[2]}$ (where $\bbush^{[2]}$ is red and purple), we obtain the Sup-Chain Rule $\RR^{\bullet}\circ \RR^{\bullet}(F)\subset \RR^{2 \bullet } F$.}
\label{Fg:sup chain rule}
\end{figure}
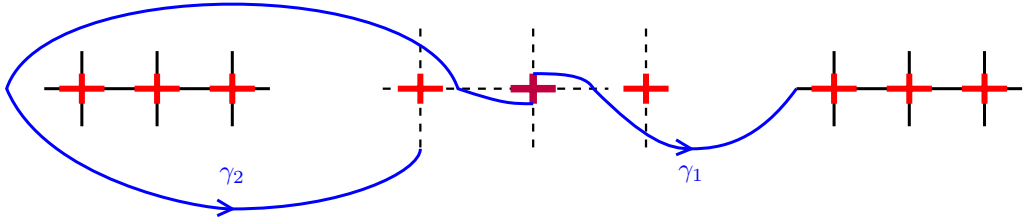

\subsubsection{Sup-Chain Rule for renormalization domains} \label{sss:sup chain rule}Consider a $\psi^\bullet$-ql map $F$ as in~\eqref{eq:dfn:psib-ql}. Since 
\begin{enumerate}
\item \label{eq:SupChainRule:1}$\RR^{n_1\bullet} \circ \RR^{n_2 \bullet} (F)$, see \eqref{RR bullet f}, is the extension of $f_{n_1+n_2,0}$ along curves in $V\setminus \bUpsilon$ (subject to certain restrictions illustrated on Figure~\ref{Fg:sup chain rule}), where
\[\bUpsilon\coloneqq \bbush^{[n_1+n_2]}\cup \left(\bbush^{[n_2]}\setminus \bush^{[n_2]}_0\right),\]
\item $\RR^{n_1+n_2 \bullet} (F)$ is the extension of $f_{n_1+n_2,0}$ along all curves in $V\setminus \bbush^{[n_1+n_2]},$  
\item and every connected component of $\bbush^{[n_1+n_2]}$ is within a connected component of $\bUpsilon$ (compare with Figure~\ref{Fg:sup chain rule}),
\end{enumerate}
we obtain the natural embedding of dynamical systems (respecting all the maps)
\begin{equation}
\label{eq:sup chain rule}
\RR^{n_1\bullet} \circ \RR^{n_2 \bullet} (F) \ \subset \ \RR^{n_1+n_2 \bullet} (F).
\end{equation}

\subsubsection{Unbalanced renormalization from $\psi^\bullet$ to $\psi$-ql maps}\label{sss: psi^b to psi} Consider a $\psi^\bullet$-ql map $F$ from~\eqref{eq:dfn:psib-ql}, and assume that $f_X$ (see~\eqref{eq:dfn:ql b domain}) is twice DH renormalizable. Let $f_{1,0}\colon X_{1,0}\to Y_{1,0}$ be a ql prerenormalization of $f_X$. 

By Item~\ref{dfn:psib-ql:3} of the definition of $F$, the little Julia set $\filled^{[1]}_{0}\equiv \filled(f_{1,0})$ is $\iota$-proper. Therefore, we can extend $f_{1,0}\colon X_{1,0}\to Y_{1,0}$ along all curves in $\bUpsilon\coloneqq \filled^{[1]}_{0}\cup \bbush^{[1]}$ and obtain $\psi$-ql map 
\begin{equation}
\label{eq:F^Ups} F_{\bUpsilon,1} = (f_{1,0},\iota_{1,0})\colon U _{\bUpsilon,1}\to V_{\bUpsilon,1}.
\end{equation}
 
We recall from~\S\ref{sss:mixed conf} that the set $\bUpsilon$ is inhomogeneous: $f_X$ does not permute components of $\bUpsilon$ .

\subsubsection{Restrictions}\label{sss:FibProd} A $\psi^\bullet$-ql map $F$ as in~\eqref{eq:dfn:psib-ql} has the natural \emph{restriction} (also known as the \emph{pullback}, \emph{fiber product}, \emph{graph}) denoted by 
\[F=(f,\iota)\colon U^2\rightrightarrows  U^1=U,\sp\sp \text{ where }\sp U^2=\{(x,y)\in U\times U\mid\sp f(x)=\iota(y)\},\]
where $f,\iota\colon U^2\to U$ are component-wise projections; see~\cite[\S2.2.2]{K}. Note that $F\colon U^2\rightrightarrows  U^1$ is also a $\psi^\bullet$-ql map. Repeating the construction, we obtain the sequence
\[F\colon U^k\rightrightarrows U^{k-1},\hspace{1cm} n\ge 1, \sp U^0=V ,\sp U^1=U,\]
together with induced iterations denoted by
\begin{equation}
\label{eq:F^k}
 F^k = \big(f^k, \iota^k\big) \colon U^k \to V.
\end{equation}  

Since $\iota\mid U^{k+1}$ is the lift of $\iota\mid U$ under the covering $f^k$, we have:

\begin{lem}
\label{lem:iota proper:preim}
The set $f^{-k}(\bush_F)$ is $\iota$-proper for $F\colon U^{k+1}\rightrightarrows U^k$.\qed
\end{lem}

\subsubsection{Sup-Chain rule for iterations}\label{sss:SupChainRule:iterations} Consider a  ql map $g\colon A\to B$. Assume that a $\psi$-ql map $F\colon U\rightrightarrows V$ is obtained from $g\colon A\to B$ using finitely many $\psi^\bullet$-ql renormalizations followed by the unbalanced renormalization from~\S\ref{sss: psi^b to psi}. Then $F$ has a ql restriction $f_X\colon X\to Y$ as in~\eqref{eq:dfn:ql b domain} that is realized in the dynamical plane of $g$:
\[ f_{X} \simeq f_{X,g}, \sp \sp \text{ where }\sp f_{X,g}=g^p\colon X_g\to Y_g.\] 
Consider the filled Julia set $\filled^{[n]}_i\equiv \filled(f_{X,g})\subset B$ of $f_{X,g}=g^p\colon X_g\to Y_g$. Let $G\colon N\rightrightarrows M$ be the $\psi$-ql map obtained from $g\colon A\to B$ by extending $f_{X,g}=g^p\colon X_g\to Y_g$ along all paths in $\bUpsilon\coloneqq \filled^{[n]}_i\cup \bbush^{[n]}$; compare with~\S\ref{sss: psi^b to psi} where the case $(n,i)=(1,0)$ is defined. By construction, we have a covering map
\[\rho \colon M\setminus \filled_G \to B\setminus \bUpsilon\] so that $\rho$ induces an isomorphism $\rho \colon \partial \filled_G \overset{\simeq }\longrightarrow \partial\filled^{[n]}_i.$ The Sup-Chain Rule~\S\ref{sss:sup chain rule} implies by induction that $F\subset G$. Therefore, we have the induced partial covering:
\begin{equation}
\label{eq:part_covering}
\rho'\colon V\setminus \filled_F \hookrightarrow M\setminus \filled_G\overset{\rho}\longrightarrow B\setminus \bUpsilon
\end{equation}
so that $\rho' \colon \partial \filled_F \overset{\simeq }\longrightarrow \partial\filled^{[n]}_i.$ Let $\gamma$ be the core hyperbolic geodesic of the annulus $V\setminus \filled_F$. Then $\gamma_F$ can be described as the set points $x$ such that the Brownian motion starting at $x$ has equal probabilities of hitting each of the boundary components of the annulus $V\setminus  \filled_F$. Therefore, the image of the subdisk $V_\gamma\Subset V$ bounded by $\gamma$ under $V\setminus \filled_F \hookrightarrow M\setminus \filled_G$ is within the subdisk $M_\beta\Subset M$ bounded by the core hyperbolic geodesic $\beta$ of the annulus $M\setminus \filled_G$. Since $\rho$ is a covering with $\rho \colon \partial \filled_G \overset{\simeq }\longrightarrow \partial\filled^{[n]}_i,$ we have $\rho\colon \beta\overset{1:1}\longrightarrow \rho(\beta)$ and $\rho(\beta)$ is a hyperbolic geodesic of $B\setminus \bUpsilon$. Therefore, $V\gamma\setminus \filled_F$ descents univalently to the dynamical plane of $g\colon A\to B$ via the partial covering~\eqref{eq:part_covering}.
 
\subsection{Width and WADs of $\psi^\bullet$ maps}\label{ss:Width+WAD}

 For a $\psi^\bullet$-ql map $F\colon U\rightrightarrows V$ we write \[\Width_\bullet(F)\coloneqq \Width(V\setminus \bush_F).\] 
 If $F$ is $\psi$-ql map  (i.e., $\filled_F$ is $\iota$-proper), then we can also measure
  \[\Width(F)\coloneqq \Width(V\setminus \filled_F).\]

The Gr\"otzsch inequality easily implies 
\[
 \Width\left(U^k\setminus \bush\right)\le\Width\left(U^k\setminus f^{-k}(\bush)\right)= 2^k\Width_\bullet(F).
\]

The Sup-Chain Rule~\eqref{eq:sup chain rule} implies that
\begin{equation}
\label{eq:width:sup chain rule}
\Width\left[\RR^{n_1\bullet} \circ \RR^{n_2 \bullet} (F) \right] \ge \Width\left[ \RR^{n_1+n_2 \bullet} (F)\right].
\end{equation}

 \subsubsection{Compactness of $\psi$ and $\psi^\bullet$-ql maps}\label{sss:compactness}  If $\Width(F)$ is bounded for $\psi$-ql map $F$, then $F$ also has regular ql bounds: 

\begin{lem}[{\cite[Lemma 2.4]{K}}]
\label{lem:from psi-ql to ql}
There is a positive function $\mu(K)$ with the following property. If $F$ is a $\psi$-ql map with $\Width(F)\le K$, then the quadratic-like map $f_X\colon X\to Y$ in~\eqref{eq:dfn:ql b domain} can be selected so that    
\begin{equation}
\label{eq:lem:from psi-ql to ql}\mod \left(Y\setminus X \right) \ge \mu(K).
\end{equation}
\end{lem}

\begin{rem}
\label{rem:lem:from psi-ql to ql}  Assume that $\psi$-ql map $F\colon U\rightrightarrows V$ is obtained from a ql map $g\colon A\to B$ using finitely many $\psi^\bullet$-renormalizations followed by the unbalanced renormalization from~\S\ref{sss: psi^b to psi}. As in~\S\ref{sss:SupChainRule:iterations},  let $\gamma$ be the core hyperbolic geodesic of the annulus $V\setminus \filled_F$. In Lemma~\ref{lem:from psi-ql to ql} we can assume (by replacing $X,Y$ with $ f_X^{-1}(X),X$) that the ql restriction $f_X\colon X\to Y$ is within the subdisk $V_\gamma\Subset V$ bounded by $\gamma$. It is shown in~\S\ref{sss:SupChainRule:iterations} that $V_\gamma\setminus \filled_F$ (and hence $V_\gamma$) descents univalently to the dynamical plane of $g$ via the partial covering~\eqref{eq:part_covering}. Therefore, the ql map $f_X\colon X\to Y$ descents univalently to the dynamical plane of $g$.
\end{rem}

\noindent It is shown in the proof of {\cite[Lemma 2.4]{K}} that $\iota$ is an embedding in a neighborhood of $\filled_F$; this argument is applicable to $\psi^\bullet$ maps:

\begin{lem}
\label{lem:compac:psi bullet} There is a positive function $\mu_\bullet(K)$ with the following property.  If $F$ is a $\psi_\bullet$-ql map with $\Width_\bullet(F)\le K$, then $\bush_F$ has a neighborhood $\Omega$ such that $\iota$ is injective on $\Omega$ and 
\[ \mod \left(\Omega\setminus \bush_F \right) \ge \mu_\bullet(K).\]
\end{lem}\qed

\subsubsection{WADs of $F$} \label{sss:WADs} Assume that $\bbush^{[n]}$ is well-defined, i.e., $F$ is $n+1$ renormalizable. Let us consider the canonical WAD of $U^k\setminus \bbush^{[n],(m)}$, $k\ge m$:
\[ \AA^{(m),k}\equiv \AA^{[n],(m),k}\coloneqq \WAD\left(U^k\setminus \bbush^{[n],(m)} \right).\]

We say (compare with~\S\ref{sss:intro:PrimPullOff}) that an arc $\alpha\in \AA^{[n],(m),k}$ is 
\begin{itemize}
\item \emph{vertical} if it connects the outer boundary $\partial^\out \left(U^k\setminus \bbush^{[n],(m)}\right)\coloneqq \partial U^k$ of $U^k\setminus \bbush^{[n],(m)}$ to one of its inner boundary components;
\item \emph{horizontal} if it connects two inner (i.e., non-outer) boundary components of $U^k\setminus \bbush^{[n],(m)}$. 
\end{itemize}

We set (see~\S\ref{sss:sub-diagram}) 
\begin{itemize}
\item $\AA^{(m),k}_\ver$ to be \emph{vertical part} of $\AA^{(m),k}$: it consists of arcs connecting $\partial U^k$ and one of the components of $\partial \bbush^{[n],(m)}$;
\item $\AA^{(m),k}_\hor$ to be \emph{horizontal part} of $\AA^{(m),k}$: it consists of arcs connecting components of $\partial \bbush^{[n],(m)}$;
%\item $\AA^{k}=\AA^{(0),k}, \sp \AA_\hor^{k}=\AA_\hor^{(0),k}, \sp \AA^{k}_\ver=\AA^{(0),k}_\ver$;
\item $\AA_i^{(m),k}$ to be the \emph{local WAD} for $\bush^{[n],(m),k}_i$, see~\S\ref{sss:LocalWAD}: $\AA_i^{(m),k}$ consisting of arcs adjacent to $\bush^{[n],(m),k}_i$ such that the weights of self-arcs adjacent to $\bush^{[n],(m),k}_i$ are doubled;
\item $\AA_{i,\ver}^{(m),k},\AA_{i,\hor}^{(m),k} $ to be the local parts of $\AA^{(m),k}_\ver,\ \AA^{(m),k}_\hor$.
%\item $\AA^{k}_i=\AA^{(0),k}_i, \sp \AA^{k}_{i,\ver}=\AA^{(0),k}_{i,\ver}\sp \AA_{i,\hor}^{k}=\AA_{i,\hor}^{(0),k},$
%\item $\AA_i=\AA^{(0),0}_i, \sp \AA_{i,\ver}=\AA^{(0),0}_{i,\ver}\sp \AA_{i,\hor}=\AA_{i,\hor}^{(0),0}.$
\end{itemize}

We will usually suppress upper zero induces:
\[\AA_i=\AA^{(0),0}_i, \hspace{0.4cm}\AA_i^{(m)}=\AA^{(m),0}_i, \hspace{0.4cm}\AA_i^{k}=\AA^{(0),k}_i, \dots \]

The following lemma provides a reverse to~\eqref{eq:width:sup chain rule} estimate. 

\begin{lem}
\label{lem:reverse to sup chain rule}
Assume that in the dynamical plane of $F=\RR^{n_0 \bullet }G\colon U\rightrightarrows V$, there is a non-trivial horizontal lamination of curves $\LL\subset V\setminus \bbush^{[n]}$ emerging from $\bush^{[n]}_0$ and landing at $\bbush^{[n]}$. Then 
\[\Width\left( \RR^{n_0+n \bullet }G \right)\ge \Width(\LL).\]
\end{lem}
\begin{proof}
By definition~\S\ref{sss:qlb:Defn}, the $\LL$ lifts univalently to a horizontal lamination $\widetilde \LL$ in the dynamical plane of $G$ emerging from $\bush^{n_0+n}_{G, 0}$ and landing at $\bbush^{[n_0+n]}_{G}$. After that $\widetilde \LL$ lifts univalently to a vertical lamination in the dynamical plane of $\RR^{n_0+n \bullet }G$. We obtain that $\Width\left( \RR^{n_0+n \bullet }G \right)\ge \Width(\widetilde \LL)\ge \Width(\LL)$.
\end{proof}

\subsubsection{WADs and $\Width_\bullet(F_{n,i})$} \label{sss:Width:F_ni} It follows from Theorem~\ref{thin-thick for S} that
\begin{equation}
\label{eq:Width is Loc WAD}
\Width_\bullet(F_{n,i}) = \Width(\AA_i)+O_{p_n}(1).
\end{equation}

As with $\psi$-ql renormalization (compare with~\cite[Theorem 9.3]{McM1}), we have
\begin{equation}
\label{eq:LocWidth is compar}
\frac 12\Width_\bullet(F_{n, 0}) \le \Width_\bullet(F_{n, i}) \le \Width_\bullet(F_{n, 0})\hspace{0.7cm} \text{ for all }\ i.
\end{equation}
Indeed, let $\gamma_i \subset V\setminus \bbush^{[n]}$ be the peripheral hyperbolic geodesic around $\bush^{[n]}_{i}$. Then the hyperoblic length of $|\gamma_i|_{V\setminus \bbush^{[n]}}$ is proportional to $\Width_\bullet(F_{n,i})$. Let $\gamma^{1}_i$ be the lift of $\gamma_i$ under $f\colon U^{1}\setminus \bbush^{[n],(1)}\to V\setminus \bbush^{[n]}$ so that $\gamma^{1}_i$ is the peripheral hyperbolic geodesic around $\bush^{[n],(1)}_{i-1}$, where the subscript is $\mod p_n$. Counting the degree of $f\colon \gamma^1_i\to \gamma_i$, we obtain 
\begin{itemize}
\item $|\gamma^1_i|_{U^{1}\setminus \bbush^{[n],(1)}}=|\gamma_i|_{V\setminus \bbush^{[n]}}$ if $i\not=1$; and
\item $\frac 12|\gamma^1_1|_{U^{1}\setminus \bbush^{[n],(1)}}=|\gamma_1|_{V\setminus \bbush^{[n]}}$.
\end{itemize}
Finally, $|\gamma^1_i|_{U^{1}\setminus \bbush^{[n],(1)}}\ge |\gamma_{i-1}|_{V\setminus \bbush^{[n]}}$; this implies~\eqref{eq:LocWidth is compar}.

\subsubsection{Covering: from $U^{k}\setminus \bbush^{(m)}$ to $U^{k+s}\setminus \bbush^{(m+s)}$}\label{sss:WAD:covering:psi maps}
It follows from~\S\ref{sss:WAD:cov} that the WADs change naturally under the covering:
\begin{equation}
\label{eq:WAD:under coverin} \AA^{(m+s),k+s}_\hor =(f^s)^* \AA^{(m),k}_\hor, \hspace{0.7cm} \AA^{(m+s),k+s}_{i-s} =(f^s)^* \AA^{(m),k}_i,
\end{equation}
and similar for other WAD such as $\AA^{(m),k}, \AA^{(m),k}_\ver,\dots$.

\subsubsection{Restriction by $\iota^s$: from $U^{k}\setminus \bbush^{(m)}$ to $U^{k+s}\setminus \bbush^{(m)}$}\label{sss:WAD:imm} Assume that $m \le k$. By Lemma~\ref{lem:iota proper:preim}, the set $\bbush^{[n],(m)}\subset U^k$ is $\iota$-proper. Let $\gamma$ be a proper curve in $U^{k}\setminus \bbush^{[n],(m)}$ emerging from $\bush^{[n],(m)}_a$. Applying $\iota^{-s}$ along $\gamma$, we construct, see~\eqref{eq:iota to s} below, its \emph{restriction} $(\iota^{*})^s[\gamma]$ which is a proper curve in $U^{k+s}\setminus \bbush^{[n],(m)}$. Moreover, 
\begin{itemize}
\item if $\gamma$ is vertical, then so is  $(\iota^{*})^s[\gamma]$,
\item if $\gamma$ is horizontal, then $(\iota^{*})^s[\gamma]$ is either horizontal or vertical. 
\end{itemize}
If $(\iota^{*})^s[\gamma]$ is horizontal, then we will also call $(\iota^{*})^s[\gamma]$ the \emph{lift} of $\gamma$ under $\iota^s$.

 More precisely, write $\gamma\colon (0,1)\to U^k\setminus \bbush^{[n],(m)}$. For every $s>0$, there is a $t_s\in (0,1]$ such that $\iota^{-s}$ extends along $\gamma\mid (0,t_s)$ and the resulting curve 
 \begin{equation}\label{eq:iota to s}
 (\iota^{*})^s[\gamma]\coloneqq \iota^{-s}\circ \gamma\colon  (0,t_s) \to U^{k+s}\end{equation}
is a proper curve in $U^{k+s}\setminus \bbush^{[n],(m)}$ (because $\bbush^{[n],(m)}\subset  U^k$ is $\iota$-proper). Moreover,
\begin{itemize}
\item if $t_m=1$, then $(\iota^{*})^n\gamma$ is vertical if and only if $\gamma$ is vertical,
\item  if $t_m<1$, then $(\iota^{*})^n\gamma$ is vertical. 
\end{itemize}

The following lemma implies that curves in a horizontal rectangle that restrict under $(\iota^s)^*$ to vertical curves form buffers of the rectangle.

\begin{lem}\label{lem:restr:buffers}
Assume that two disjoint horizontal paths $\gamma_1,\gamma_2\subset U^{k}\setminus \bbush^{(m)}$ representing the same arc  $\alpha=[\gamma_1]=[\gamma_2]$ lift under $(\iota^s)^*$ to horizontal paths in $U^{k+s}\setminus \bbush^{(m)}$. Let $R\subset U^{k}\setminus \bbush^{(m)}$ be the proper rectangle between $\gamma_1$ and $\gamma_2$ such that all vertical curves in $R$ also represent $\alpha$. Then all curves in $\Fam^\full(R)$ lift under $(\iota^s)^*$ to homotopic horizontal curves in $U^{k+s}\setminus \bbush^{(m)}$. 
\end{lem}
\begin{proof}
The curve $\gamma\in \Fam^\full(R)$ becomes vertical under the restriction by $(\iota^s)^*$ if and only if $\gamma(t)$ hits $\iota^s\big(\partial U^{k+s}\big)$. This is impossible because $\gamma_1(t),\gamma_2(t)$ do not hit $\iota^s\big(\partial U^{k+s}\big)$. 
\end{proof}

\subsubsection{Monotonicity of the $\AA^{k}_\hor$} \label{sss:monot:AA^k} Since $\iota\colon U^{k+1}\to U^k$ is an embedding in a neighborhood of $\filled_F$, we \emph{identify up to homotopy} $U^k\setminus \bbush^{[n],(m)}$ and $U^s\setminus \bbush^{[n],(m)}$ for all $k,s$. In particular, horizontal arcs $\alpha$ in $U^k\setminus \bbush^{[n],(m)}$ are naturally viewed as horizontal arcs in $U^s\setminus \bbush^{[n],(m)}$ by realizing the $\alpha$ in a small neighborhood of $\filled_F$ where $\iota^{s-k}$ is an embedding.

It follows essentially from Lemma~\ref{tranform rule for rectangles} (by lifting $\iota$ to the universal coverings, see~\cite[\S 3.5]{K} for reference) that for $k\ge m$
\begin{equation}
\label{eq:sss:monot:AA^k}
\AA^{[n],(m),k}_\hor \ge  \AA^{[n],(m),k+1}_\hor\hspace{0.5cm} \text{ and } \hspace{0.5cm}  \AA^{[n],(m),k}_{i,\hor} \ge  \AA_{i,\hor}^{[n],(m),k+1}.
\end{equation}

\subsubsection{Domination: from $U^{k}\setminus \bbush^{(m)}$ to $U^{k}\setminus \bbush^{(m+1)}$ (following \cite[{\S3.6}]{K})} \label{sss:domination} 
Assume that $\gamma\subset U^{k}\setminus \bbush^{[n],(m)}$ is a proper horizontal curve. Then it has the decomposition 
\begin{equation}
\label{eq:sss:domination}
\gamma = \ell_0\#\gamma_1\# \ell_1 \#\gamma_2\# \dots \#\gamma_s\#\ell_s
\end{equation}
  (compare with~\S\ref{sss:periodic arcs}) such that
 \begin{itemize}
\item $\gamma_j\subset U^{k}\setminus \bbush^{[n],(m+1)}$ are proper horizontal curves; and
\item every component of $\ell_i\setminus \bbush^{[n],(m+1)}$ is trivial in $U^k\setminus \bbush^{[n],(m+1)}$  (with respect to a proper homotopy, see~\S\ref{sss:AD}).
\end{itemize} 
We call $s$ the \emph{expansivity number} of $\gamma$ in rel $\bbush^{(m+1)}$.

It follows from Decomposition~\eqref{eq:sss:domination} that there is a combinatorial constant $C\equiv C_{p_n,m}$ (essentially, the maximal number of possible arcs rel $\bbush^{[n],(m+1)}$) such that for all $M>1$ 
\begin{equation}
\label{eq:prop:dominat}
\AA^{[n],(m),k}_\hor- M C \sp\sp \text{ is \emph{dominated} by }\sp\sp \AA^{[n],(m+1),k}_\hor-M
\end{equation}in the following sense, see~\S\ref{sss:sub-diagram} for notation. For every arc $\alpha$ in $\AA^{[n],(m),k}_\hor- M C$, there is a vertical sublamination 
\[\LL_\alpha\subset \Fam(e)\equiv \bPi^\can (\alpha)\sp\sp\text{ with }\sp\sp \Width(\LL_\alpha) \ge \Width(\alpha) -  M C\]
such that Decomposition~\ref{eq:sss:domination} of every $\gamma\in \LL_\alpha$ has the additional property that 
\begin{equation}
\label{eq:dfn:domin:gamma_i}
\gamma_i\in \Fam^\full \left(\AA^{[n],(m+1),k}_\hor-M\right).
\end{equation}
The lamination $\LL_a$ is constructed by removing all $\gamma$ from $\Fam(\alpha)$ that do not satisfy~\eqref{eq:dfn:domin:gamma_i} -- the width of removed curves is bounded by  $ M C$.

\subsubsection{Almost periodic rectangles} \label{sss:PeriodRect} Consider a $\psi^\bullet$-ql map $F\colon U \rightrightarrows V$ as in~\eqref{eq:dfn:psib-ql}. We say that a proper rectangle $R$ in $V\setminus \bbush^{[1]}$ is almost periodic if most of the width of $R$ overflows its iterative lift; more precisely:

\begin{defn}
\label{dfn:s delta inv rect} 
For a $\psi^\bullet$-ql map $F\colon U \rightrightarrows V$, a proper rectangle $R\subset V\setminus \bbush^{[1]}$ connecting $\bush^{[1]}_a,\bush^{[1]}_{a+1}$ is called \emph{$\delta$-almost periodic} if $R$ represents a genuine periodic arc (of some AD, see \S\ref{sss:periodic arcs}) and $R$ has a vertical sublamination $\FamG\subset \Fam(R)$ with $\Width(\FamG)\ge (1-\delta)\Width(R)$ such that Conditions~\ref{c1:dfn:s delta inv rect}, \ref{c3:dfn:s delta inv rect} stated below hold for all $s\le 10 \overline p$, where $\overline p$ is the combinatorial bound~\S\ref{sss:infin renorm}.
\begin{enumerate}[label=\text{(\Roman*)},font=\normalfont,leftmargin=*]
\item \label{c1:dfn:s delta inv rect} Under the immersion, $\FamG$ lifts to the lamination \[\FamG^s\coloneqq (\iota^{sp_1})^*(\FamG)\subset U^{sp_1}\setminus \bbush^{[1]}\] still connecting $\bush^{[1]}_a,\bush^{[1]}_{a+1}$.
\end{enumerate}
Let $R^{(s)}\subset U^{(sp_1)}\setminus \bbush^{[1],(sp_1)}$ be the periodic lift of $R$ under $f^{s p_1}$; i.e., $R^{(s)}$ connects $\bush^{[1],(sp_1)}_a,\bush^{[1],(sp_1)}_{a+1}$.
\begin{enumerate}[label=\text{(\Roman*)},start=2,font=\normalfont,leftmargin=*]
\item \label{c3:dfn:s delta inv rect} The lamination $\FamG^s$ overflows $R^{(s)}$ as follows: every curve $\gamma$ in $\FamG^s$ is the concatenation $\gamma=\ell_a\#\gamma'\#\ell_{a+1}$ such that 
\begin{itemize}
\item $\gamma'\in \Fam^\full\big(R^{(s)}\big)$, see~\S\ref{ss:rectangles}; and
\item every component of $\ell_a\setminus \bush^{[1],(sp_1)}_a$ and every component of $\ell_{a+1}\setminus \bush^{[1],(sp_1)}_{a+1}$ is trivial in $U^{sp_1}\setminus \bbush^{[1],(sp_1)}$  with respect to a proper homotopy, see~\S\ref{sss:AD}.
\end{itemize}
\end{enumerate}
\end{defn}

By Lemma~\ref{lem:per arcs}, the first renormalization of $F$ is satellite and $\bush^{[1]}_a, \bush^{[1]}_{a+1}$ are neighboring bushes with respect to the cyclic order.

\begin{rem}
\label{rem:dfn:s delta inv rect} An almost periodic rectangle $R$ between little bushes $\bush^{[n]}_a,\bush^{[n]}_{a+1}$ is defined in the same way as in the case $n=1$. Such a rectangle $R$ can be lifted to the dynamical plane of $F_{n-1,c}$ (via the covering map representing the $\psi^\bullet$-renormalization, see~\S\ref{sss:psi b renorm}) and its lift will be an almost periodic rectangle between level $1$ little bushes as in Definition~\ref{dfn:s delta inv rect} . 
\end{rem}

\section{Pull-off for non-periodic rectangles}\label{s:Pull off for non-periodic rect} In this section we will establish the following theorem that refines a result from~\cite{K}: 
\begin{thm}
\label{thm:prim pull-off}
For every bound $\bar p$ on renormalization periods as in \S\ref{sss:infin renorm}, every small $\delta>0$, and every $n\gg_{\bar p, \delta} 1$, the following holds. 

Let $F=(f,\iota)\colon U\rightrightarrows V$ be a $\psi^\bullet$-ql map $n+1$ times renormalizable of type $\bar p$, and let $F_{n,i}=(f_{n,i},\iota_{n,i})\colon U_{n,i}\rightrightarrows V_{n,i}$ be its $n$th $\psi^\bullet$-renormalization, see~\S\ref{sss:psi b renorm}.  If 
\[ \Width_\bullet (F_{n,i}) =K\gg_{\bar  p ,\delta, n} 1,\]
then  
\begin{enumerate}[label=\text{(P)},font=\normalfont,leftmargin=*]
\item \label{case:P:thm:prim pull-off}either $\Width_\bullet(F)\ge 2K$;
\end{enumerate}
\begin{enumerate}[label=\text{(S)},font=\normalfont,leftmargin=*]
\item \label{case:S:thm:prim pull-off}or the $n$th renormalization of $F$ is satellite and the $(n-1)$-th $\psi^\bullet$-ql renormalization $F_{n-1}$ has a $\delta$-almost periodic rectangle $R$  with $\Width(R) \ge K/20$ (see Definition~\ref{dfn:s delta inv rect}) between two neighboring bushes in its satellite flower.
\end{enumerate}
\end{thm}

If the $n$th renormalization of $F$ is primitive, then $F_{n-1}$ has no periodic rectangles; i.e., Case \ref{case:P:thm:prim pull-off} holds.

\begin{proof}  (See also~\S\ref{sss:intro:PrimPullOff}.) Consider the periodic cycle $\bbush^{[n]}$ of level $n$ little bushes in the dynamical plane of $F\colon U\rightrightarrows V$. Recall from~\eqref{eq:F^k} that $F^k\colon U^k\rightrightarrows V$ represents the $k$th iteration of $F$.

The degeneration of $U^k\setminus \bbush^{[n]}$ is described by the WAD $ \AA^k\equiv \AA^{[n],k}$; let us consider its horizontal and vertical parts $\AA^{k}_\hor, \AA^k_\ver$, see details in~\S\ref{sss:alignment with HT}. Since the $U^k\setminus \bbush^{[n]}$ decrease, so are the arc diagrams representing $\AA^k_\hor$; i.e., the $A^k_\hor\equiv \AD \left(\AA^k_\hor\right)$ stabilize with $k\le 3p_n$. And since $f^{-1}(U^k\setminus \bbush^{[n]})\subset U^{k+1}\setminus \bbush^{[n]}$, the $A^k_\hor$ stabilize at an invariant AD; i.e., $A^k_\hor$ is aligned with the Hubbard tree of the superattracting model.

Consider the local part $\AA^k_i$ of $\AA^k$ around $\bush^n_i$. Write $q\coloneqq  (10 \overline p +3 )p_n$. We remark that since $q$ is linear in $p_n$, the $q$th iterate has a bounded degree on small bushes and the below application of the Covering Lemma will lead to estimates independent of $n$.

If for all $i$ we have
\begin{equation}
\label{eq:LC:prim pull-off} \Width(\AA^{q}_i)-\Width(\AA^0_i)\le \delta_1 K,\sp\sp \sp \delta_1\ll \delta,
\end{equation}
 then most of the degeneration in $\AA^0_i$ is represented by genuine periodic arcs \S\ref{sss:periodic arcs}; this leads to Case~\ref{case:S:thm:prim pull-off},  see~\S\ref{prf:P:thm:prim pull-off}.
 
If~\eqref{eq:LC:prim pull-off} does not hold some $i$, then $\Width(\AA^q_{i, \ver})\ge \delta_1 K$ and applying the Covering Lemma we obtain
\[\Width\left(V\setminus \bush_j^{[n]}\right)\asymp_{\delta_1, \overline p}\  K\hspace{0.5cm }\text{ for all }j.\]
Applying the Quasi-Additivity Law and using $p_n\gg_{\delta_1,\overline p} 1 $, we obtain
\[\Width(F)\succeq_{\delta_1, \overline p} \ \sum_j \Width\left(V\setminus \bush_j^{[n]}\right) \succeq_{\delta_1, \overline p}  \ p_n K\gg 2K,\]
see~\S\ref{ss:CoverLmm}; this is Case~\ref{case:P:thm:prim pull-off}.

Below we will make a more technical exposition.
\end{proof}

\subsection{Alignment of WAD with the Hubbard continuum}\label{sss:alignment with HT} Following~\S\ref{sss:LittleBushes}, let \[\bbush\equiv\bbush^{[n],(0)} \sp\sp \text{ and } \sp\sp \bbush^{(m)}\equiv\bbush^{[n],(m)}=f_X^{-m} (\bbush)\]
 be the periodic cycle of level-$n$ bushes and its preimage, where $f_X=[f\colon X\to Y]$ is a ql restriction of $f$, see~\eqref{eq:dfn:ql b domain} in~\S\ref{sss:qlb:Defn}. The similar suppression of indices is applied to WAD from~\S\ref{sss:WADs}.

Recall from \S\ref{sss:WAD:covering:psi maps} and~\S\ref{sss:monot:AA^k} that we have 
\begin{equation}
\label{eq:covering+monoton:AA}
\AA^{(m+1),k+1}_{\hor}= f^*\left(\AA^{(m),k}_{\hor}\right)\sp\sp \text{ and }\sp\sp \AA^{(m),k+1}_\hor \le  \AA^{(m),k}_\hor.
\end{equation}

Let us choose a sufficiently big threshold $C\gg_{p_n} 1$ so that~\eqref{eq:prop:dominat} is applicable for $m=0$. By~\eqref{eq:covering+monoton:AA}, the ADs forming $\AA^k_\hor-C^k$ decrease, hence for a certain
\begin{equation}
\label{eq:t:main est}
\bbt\le \  \big[\text{maximal number of horizontal arcs on  $V\setminus \bbush_n$}\big]\ \le 3p_n
\end{equation}
the arc diagram
\begin{equation}
\label{eq:dfn:HH}
H=\AD(\HH), \hspace{0.5cm} \text{ where }\hspace{0.5cm} \HH\coloneqq \AA^\bbt_\hor -  C^\bbt
\end{equation} coincides with the AD of $\AA^{\bbt+1}_\hor - C^{\bbt+1}$.

Since the WAD $\AA^{\bbt+1}_\hor - C^{\bbt+1}$ on $U^{\bbt+1}\setminus \bbush$ is dominated by the WAD $ f^* \left(\HH \right)=f^*\left(\AA^\bbt_\hor -  C^\bbt \right)$ on $U^{\bbt+1}\setminus \bbush^{(1)}$, see~\eqref{eq:prop:dominat} (where $M=C^\bbt$) in \S\ref{sss:domination}, we obtain that $H$ is an invariant AD, and therefore it is aligned with the Hubbard continuum by Lemma~\ref{lem:Alignment}.

\subsection{Case~\ref{case:P:thm:prim pull-off}} \label{ss:CoverLmm} Consider local WAD $\AA_{i}^{k}$ as in~\S\ref{sss:WADs}, we have by~\S\ref{sss:Width:F_ni}
\begin{equation}
\Width(\AA_{i}^{k}) \ge \Width(\AA_{i}^{0})=\Width_\bullet(F_{n,i})-O_{p_n}(1)\ge K/2-O_{p_n}(1).
\end{equation}

Write $q\coloneqq  (10 \overline p +3 )p_n$ and fix a sufficiently small $\delta_1>0$. Assume that:
\begin{enumerate}[label=\text{(P${}_\loc$)},font=\normalfont,leftmargin=*]
\item \label{case:P:loc:thm:prim pull-off} there is a $\kappa$ such that
\[\Width(\AA_{\kappa,\ver}^{ q})\ge \delta_1 K.\]
\end{enumerate}

Since all local weights are comparable~\S\ref{sss:Width:F_ni}, we have the Collar Assumption:
\begin{equation}
\label{eq:collar assumpt}
\Width\left(V\setminus \bigcup_{j\not=i}\bush_j, \sp \bush_i \right)\asymp K\hspace{0.8cm} \text{ for all $i$.}
\end{equation}

Since $\Width(\AA_{\kappa,\ver}^{ q+j})\ge  \Width(\AA_{\kappa,\ver}^{ q}) \ge \delta_1 K$ and~\S\ref{eq:collar assumpt} holds, we obtain from the~\cite[Covering Lemma]{covering lemma} that
\[\Width(V, \bush_{\kappa+i})\succeq_{\delta_1, \overline p} \ \Width(U^{q+i}, \bush_{\kappa}) \ge \Width(\AA_{\kappa,\ver}^{ q+i}) -O_{p_n}(1) \succeq_{\delta_1, \overline p}\ K\]
for all $i\in \{0,1, \dots, p_n-1\}$. Therefore,
\[\sum_{j=0}^{p_{n}}\Width(V, \bush_{j})\succeq_{\delta_1, \overline p}\  p_n K.\]
Applying \cite[Quasi Additivity Law with separation]{covering lemma} together with~\eqref{eq:collar assumpt} and using $p_n\gg _{\delta_1,\overline p}\ 1$, we obtain
\[\Width(V\setminus \bbush)\succeq_{\delta_1,\overline p }\ \sum_{j=0}^{p_{n}}\Width(V, \bush_{j})\succeq_{\delta_1, \overline p}\   p_n K\ge 2K.\]
We conclude that
\[ \Width_\bullet(F) = \Width( V,\  \bush_F)\ge \Width( V,\ \bbush) \ge 2K.\]

This establishes Case~\ref{case:P:thm:prim pull-off} from~\ref{case:P:loc:thm:prim pull-off}.

\subsection{Case~\ref{case:S:thm:prim pull-off}} 
\label{prf:P:thm:prim pull-off}
Let us now assume that~\ref{case:P:loc:thm:prim pull-off} does not hold: for all $i$ we have $\Width(\AA_{i,\ver}^{q})\le \delta_1 K$. This implies that most paths in the canonical lamination $\Fam (\AA_{i})$ of $\AA_i$ are horizontal and they restrict to horizontal paths under the immersions $\iota^s\colon (U^s,\bbush)\to (V,\bbush)$ for $s\le q$. We will use notations of~\S\ref{sss:in out: sat flowers}.

Let $\HH$ be the WAD from~\eqref{eq:dfn:HH} on $(U^\bbt, \bbush)$.  Consider an arc $e$ in $H$ with maximal $\Width(e)$. Since the number of arcs in $H$ is bounded by $3p_n$ and $\Width(\HH)\ge p_nK/4 - \delta_1 p_n K$ by~\eqref{eq:LocWidth is compar}, we obtain that $\Width(e)\ge K/14$. We claim that $e$ is within a periodic satellite flower $\bush^{[n-1]}_c$. In particular, the $n$-th renormalization of $f$ is satellite. 

Assume converse: $e$ is not in any satellite flower $\bush^{[n-1]}_c$. By Lemma~\ref{lem:sss:in out:sat flower}, $e$ overflows arcs $e_1,e_2\in (f^{2p_n})^*(H)$ with $f^{2p_n}(e_1)=f^{2p_n}(e_2)=e_\new\in H$. By Lemma~\ref{lem:restr:buffers}, $\big(\iota^{2p_n}\big)^*\left[\Fam(e-\delta_1K)\right]$ is a horizontal family of curves in $U^{\bbt+2p_n}\setminus \bbush$; after removing $O_{p_n}(1)$ curves, this family is dominated by $(f^{2p_n})^*(H)$. Therefore, by the Series Law:
\[\max\left(\Width(e_1), \ \Width(e_2)\right) \ge 2 \Width(e) - 2.1\ \delta_1K. \]
Applying $f^{2p_1}$, we obtain the contradiction: $\Width(e_\new)>\Width(e)$.

Consider now a  satellite flower $\bush^{[n-1]}_c$.  As in~\S\ref{sss:in out: sat flowers}, we denote by $\HH_c$ the WAD consisting of arcs of $\HH$ that are in $\bush^{[n-1]}_c$.

Write $q'=10 \overline p p_n$. Let  $\HH^{(q')}_c$ be the WAD consisting of arcs of $\big(f^{q'}\big)^*\HH$ that are in $ \bush^{[n-1]}_c$. By Lemma~\ref{lem:restr:buffers}, the restriction $(\iota^{q'})^*\left[\HH_c-\delta_1 K \right]$ consists of horizontal curves. Since $H=\AD(\HH)$ is aligned with the Hubbard dendrite, $(\iota^{q'})^*\left[\HH_c-\delta_1 K \right]$ is, after removing $O_{p_n}(1)$ curves, dominated by $\HH^{(q')}_c$, see~\S\ref{sss:domination}. And since $(f^{q'})_*\colon \HH^{q'}_c\to \HH_c$ is a bijection, as most $\overline p \delta_1K + 2C^2< \delta_2 K$ curves in $\HH_c$ can have expansivity number greater than $2$, where $\delta_1\ll \delta_2\ll \delta$. This implies that most curves in $\Fam(\HH_c)$ are within rectangles representing genuine periodic arcs~\S\ref{sss:periodic arcs} that connect neighboring level $n$ bushes of $\bush^{[n-1]}_c$, see Lemma~\ref{lem:per arcs}. This also demonstrates that all the $\HH_c$ have comparable width -- the difference of their weights are bounded by $\delta_2 K$.

Since one of the $\HH$ has an edge $e$ with $\Width(e)\ge K/14$, 
the map $F\colon U^{\bbt+1}\rightrightarrows U^\bbt$ has a $\delta/2$-almost periodic rectangle $R$ between neighboring level $n$ bushes of $\bush^{[n-1]}_0$ with $\Width(R)\ge K/15$ and satisfying Remark~\ref{rem:dfn:s delta inv rect}.

\subsubsection{Pushing forward $R$ into $V\setminus\bbush^{n}$ and then into the dynamical plane of $F_{n-1}$} Since $\Width(\AA_{i,\ver}^{q})\le \delta_1 K$, curves in $\Fam\big(\AA_{i,\hor}-2\delta_1 K\big)$ restrict under $(\iota^\bbt)^*$ to horizontal curves in $(U^\bbt,\bbush)$. Since $\AA^{\bbt}_i \le \AA_i$, see~\eqref{eq:sss:monot:AA^k}, we obtain from Lemma~\ref{lem:restr:buffers} that 
\[\Fam\big(\AA^\bbt_{i,\hor}-2 \delta_1 K -2\big)  \subset\ (\iota^\bbt)^*\left[ \Fam^\full\big(\AA_{i,\hor}-2\delta_1 K\big) \right]. \]
Therefore, $\Fam(\AA^\bbt_{i,\hor}-2 \delta_1 K -2)$ can be univalently pushed forward under $(\iota^\bbt)_*$. By removing $\delta_1K$ buffers from $R$, we push forward $R$ into the dynamical plane of $F\colon  U\rightrightarrows V$ and then push forward $R$ into the dynamical plane of $F_{n-1}$, see Remark~\ref{rem:dfn:s delta inv rect}. This establishes Case~\ref{case:S:thm:prim pull-off} of the theorem.

%It follows that all curves in $\Fam\big (\AA_{i,\ver}^{\bbt}\big-\delta_1 K)$ can be lifting to horizontal curves in 

\section{Waves}\label{waves sec} 
Given a compact connected filled set $ X\subset \C$, we denote by $\partial ^c X$ its Carath\'eodory boundary;  i.e., the set of prime ends of $\wC\setminus X$. A Riemann map identifies $\partial^c \widetilde X$ with the unit circle $\T$. For a compact connected subset $Y\subset X$, we denote $\partial^c_X Y\subset \partial^c X$ the set of prime ends of $X$ accumulating at $Y$. A \emph{side of $Y$ rel $X$} is a connected component of $\partial^c_X Y $ viewed as a subset of $ \T\simeq \partial ^cX$.

Let $F=(f,\iota)\colon U\rrto V$ be a $\psi^\bullet$-ql map. Assume that
\begin{itemize}
\item $(\filled_k)_{k}$ is a forward invariant collection of periodic and preperiodic little filled Julia sets of $f$ of the same level;
\item $T\supset \bigcup_k \filled_k $ is forward invariant compact connected filled subset of the Julia set of $f$;
\item $T$ is $\iota$-proper;
\item every $\filled_k$ has finitely many sides in $T$. 
\end{itemize}
A relevant example for us is $T=\bush_F^{(m)}$. Let us denote by $M$ the total number of sides of all $\filled_k$.

Consider a side $\filled_k^\iota$ of $\filled_k$ in $T$. A \emph{wave $\FamS$ above} $\filled_k^\iota$ is a lamination of proper paths in $U\setminus T$ such that every curve $\gamma\in \FamS$ starts and ends at $\partial^cX\setminus \filled_k^\iota$ and goes above $\filled_k^\iota$: the bounded component $O$ of $U\setminus (\gamma\cup T)$ contains $\filled_k^\iota$ on its Carath\'eodory boundary (i.e., prime ends of $\filled_k^\iota$ are accessible from $O$). 

\begin{lem}[Wave Lemma]
\label{lem:Wave}
Let $\FamS$ be a wave as above. Then $\Width(U\setminus T) \succeq_M \Width(\FamS).$
\end{lem}
\begin{proof}
We will use the following fact:
\begin{lem}
\label{lem:WL: removing bugfers}
Suppose $f\colon A\to B$ is a degree $m$ covering between closed annuli. Suppose $J\subset \partial A$ is an interval. Let $\FamS$ be a wave in $A$ above $J$. Then $\FamS$ contains a genuine subwave $\FamS^\new$ such that $ f\mid \FamS^\new$ is injective and \[\Width(\FamS^\new)\ge \Width(\FamS)-O(\ln m).\]
\end{lem}
\noindent In particular, if $f\mid J$ is not injective, then $\Width(\FamS)=O(\ln m)$.
\begin{proof}
Consider universal covering maps $X\to A$ and $X\to B$. Their group of deck transformations are isomorphic to $m\Z$ and $\Z$ respectively. Let $\FamS_{k }$, $k\in m\Z$ be the lifts of $\FamS$ under $X\to A$; all these lifts are disjoint and permuted by $m\Z$. Let $\FamS_{k}$, $k\in \Z$ be the orbit of $\FamS$ under $\Z$. We {\bf claim} that, by removing $O(\ln m)$ outermost curves from $\FamS$, we obtain $\FamS^\new$ and the new $\FamS^\new_k, k\in \Z$ so that \[\FamS^\new_0,\sp \FamS^\new_1,\sp\dots ,\FamS^\new_{m-1}\] are pairwise disjoint. Then, by the claim, all $\FamS^\new_k,k\in \Z$ are pairwise disjoint, i.e.~\[f\colon \FamS^\new\longrightarrow f(\FamS^\new)=\FamS^\new_0/\Z\]
is injective. 

Let us verify the claim. Let us denote by $A_1$ the component of $\partial A$ containing $J$, by $H\simeq \Z/m$ the  group of deck transformations of $f\colon A\to B$. Then $H$ acts on $A_1$. We can decompose $\FamS=\FamS^0\sqcup \FamS^1\sqcup \dots \FamS^T$ into possibly empty pairwise-disjoint laminations such that
\begin{itemize}
\item curves in $\FamS^0$ start and end at an interval $I^0\subset  A_1$ that is a fundamental interval for the action $H\curvearrowright
 A_1$; 
\item curves in $\FamS^t$, $t>0$ start at an interval $I^t_- \subset  A_1$ and end at an interval $I^t_+\subset  A_1 $ such that one of the intervals $I^t_-, I^t_+$ is within a union of $2^t-1$ fundamental intervals for the action $H\curvearrowright
 A_1$ and the interval $J^t\subset A_1,$ $ J^t\cap J\not=\emptyset $ between $I^t_-,I^t_+$ is the union of exactly $2^t-1$ fundamental intervals of the action $H\curvearrowright
 A_1$;
\item $T= O(\ln m)$.
\end{itemize}
Then $\Width(\FamS^t)=O(1)$ for $t>0$ because, after removing $1$ buffers, $\FamS^t$ crosses its shift under $H$, see~\cite[Shift Argument \S A.3]{DL2}. 

Replacing $\FamS$ with $\FamS^0$, we obtain that curves in the $\FamS^{0}_k$ start and end at pairwise disjoint intervals. Removing an extra $O(1)$ outermost curves, we obtain that the new $\FamS^\new_l$ are pairwise disjoint.
\end{proof}

Let us assume first that $\filled_k$ is periodic. Since the number sides is $M$ and since the period of small Julia sets is bounded by $M$, we can fix an iteration $f^n\mid T$, where $n$ is bounded in terms of $M$ (for instance, $\  n\le M^3$) such that $f^n\mid \filled_k^\iota$ covers $\partial^c_T \filled_k$ at least twice. Consider the associated iteration \[(f^n,\iota^n)\colon U^n\rrto V,\hspace{1cm}\text{see~\eqref{eq:F^k}}.\] Under the immersion $\iota^n$, either $\frac 1 3 \Width(\FamS)$ part of $\FamS$ lifts to a vertical family of $U^n\setminus T$ or $\frac 2 3 \Width(\FamS)$ part of $\FamS$ lifts univalently into the lamination $\FamS'$. In the former case, the lemma is proven: 
\begin{equation}
\label{eq:lem:Wave:prf}
\Width(V\setminus T)\ge \frac{1}{2^n}\Width(U^n\setminus T)\ge  \frac{1}{2^n 3}\Width(\FamS).
\end{equation}

Assume the latter case. Let $\widetilde T$ be the full preimage of $T$ under $f^n$. Then $\filled_k^\iota$ splits into finitely many sides $X_1,X_2,\dots, X_t$ of $\filled_k$ in $\widetilde T$. Every $X_i$ maps univalently to the side $f^n(X)$ of $\filled_k$ in $T$ under $f^n$. Since $f^n\mid \filled_k^\iota$ covers $\partial^c_T \filled_k$ at least twice, we can choose two sides $X_a,X_b$ that map onto $\filled^\iota_k$. Every curve $\gamma\in \FamS'$ has first shortest subcurves $\gamma_a, \gamma_b$ (which may coincide) in $U^n\setminus \widetilde T$ above $X_a,X_b$ respectively. By Lemma~\ref{lem:WL: removing bugfers}, the width of $\gamma$ with $\gamma_a=\gamma_b$ is $O_M(1)$.

Write $K\coloneqq \Width(\FamS)$. Let $\Fam'$ be the family of all $\gamma$ in $\FamS$ with $\gamma_a\not =\gamma_b$. Let $\FamS_a,\FamS_b$ be family of curves consisting of $\gamma_a$ and $\gamma_b$ with $\gamma$ in $\FamS'$. Since $\FamS'$ overflows $\FamS_a$ and then $ \FamS_b$, either $\Width(\FamS_a)$ or $\Width(\FamS_b)$ is at least $4/3K-O_M(1)$. After removing $O_M(1)$ buffers and applying Lemma~\ref{lem:WL: removing bugfers}, the waves $\FamS_a,\FamS_b$ map univalently by $f^n$ onto waves above $\filled_k^\iota$. We obtain a wave with width $\ge \frac 4 3 K-O_M(1)$ and the whole argument can be repeated with a bigger wave leading eventually to the contradiction.

Assume now that $\filled_k$ is strictly preperiodic. Let $f^n$ be the smallest iteration so that $f^n(\filled_k)$ is periodic. As above, either the $1/3$ part of $\FamS$ lifts to a vertical family under $\iota^n$ or the $2/3$ part of $\FamS$  lifts univalently. In the former case,~\eqref{eq:lem:Wave:prf} concludes the proof. In the latter case, we construct the family $\FamS_a$ as above and then we pushforward $\FamS_a$ under $f^n$. The result will be a wave above a side of a periodic Julia set. This reduces the preperiodic to the periodic case. 
\end{proof}

\section{Pull-off for periodic rectangles}
\label{s:Pulloff for per rectangles}

 Consider a $\psi$-ql like map $F=(f,\iota)\colon U\rightrightarrows V$. We assume that the first renormalization of $F$ is satellite.

\begin{thm}
\label{thm:sat pull off}
Fix a combinatorial bounds $\overline p$ on the renormalization period~\S\ref{sss:infin renorm}. Then for every sufficiently small $\delta>0$ there is a $C_\delta = C_{\delta, \overline p}>1$ with
\[C_\delta\to \infty\sp \sp \text{ as }\sp\sp \delta\to 0\]
such that the following holds for every $\psi^\bullet$-ql map $F$, and its $\psi^\bullet$-ql renormalizations $\ F_1=\RR^\bullet (F),\ F_2=\RR^{2\bullet }(F)$, see~\eqref{RR bullet f}.

Suppose that $R$ with $\Width(R)\gg_{\delta, \overline p}1 $ is a $\delta$-almost periodic rectangle (see Definition~\ref{dfn:s delta inv rect}) in the dynamical plane of $F$ between bushes $\bush_{a}^{[1]}$ and $\bush_{a+1}^{[1]}$. Then 
\begin{equation}
\label{eq:1:thm:sat pull off}
\text{either }\sp \sp\sp\Width_\bullet(F_2)\ge C_\delta K \sp\sp\text{ or }\sp\sp  \Width_\bullet(F)\ge C_\delta K.
\end{equation}
Moreover, if $F=\RR^{n\bullet} (G)$, then we also have the alternative
\begin{equation}
\label{eq:2:thm:sat pull off}
\text{either }\sp \sp\sp\Width_\bullet\left[ \RR^{n+2 \bullet} (G)\right]\ge C_\delta K \sp\sp\text{ or }\sp\sp  \Width_\bullet(F)\ge C_\delta K
\end{equation}
(independent of $n$).
\end{thm}
\subsection{Proof of Theorem~\ref{thm:sat pull off}}
Let $\Pi\subset \bush_F$ be the geodesic continuum between $\bush^{[n]}_a$ and $\bush^{[1]}_{a+1}$. We will use notations from Definition~\ref{dfn:s delta inv rect} such as $\FamG, \FamG^s, R^{(s)}$.

\subsubsection{Spiraling Numbers}First we will introduce the spiraling parameters rel $\Pi$ for a curve $\gamma\in R$ shortly summarized in the following remark:

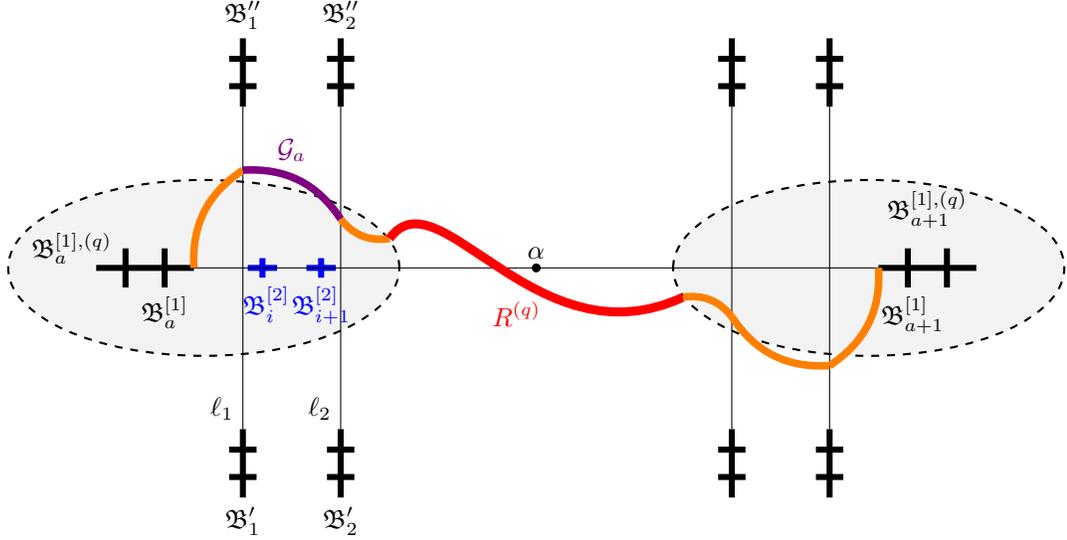
\begin{figure}[t!]
\[\begin{tikzpicture}[scale=1.3]

\begin{scope}[rotate =0]
\draw[line width=0.8mm] (-0.5,0)--(0.5,0); 
\draw[line width=0.8mm] (-0.2,-0.2)--(-0.2,0.2);
\draw[line width=0.8mm] (0.2,-0.2)--(0.2,0.2);
\coordinate (ba) at (0.5,0);
\node[below] at  (0.2,-0.2) {$\bush^{[1]}_{a}$};

 \draw[dashed,line width=0.8,fill, fill opacity=0.05] (0.6,0) ellipse (2cm and 0.9cm);
\node[] at (-0.75,0.2){$\bush^{[1], (q)}_{a}$};

\begin{scope}[shift={(1,2)},rotate=90,scale=0.7]
\draw[line width=0.8mm] (-0.5,0)--(0.5,0); 
\draw[line width=0.8mm] (-0.2,-0.2)--(-0.2,0.2);
\draw[line width=0.8mm] (0.2,-0.2)--(0.2,0.2);
\coordinate (xup) at (-0.5,0);

\node[above] at (0.5,0) {$\bush''$};
\end{scope}
\begin{scope}[shift={(1,-2)},rotate=-90,scale=0.7]
\draw[line width=0.8mm] (-0.5,0)--(0.5,0); 
\draw[line width=0.8mm] (-0.2,-0.2)--(-0.2,0.2);
\draw[line width=0.8mm] (0.2,-0.2)--(0.2,0.2);
\coordinate (xdown) at (-0.5,0);

\node[left] at (-0.8,0) {$\ell_1$};
\node[below] at (0.5,0) {$\bush'$};
\end{scope}

\draw[] (xup)  -- (xdown); 

\begin{scope}[shift={(-0.05,0)}]
\draw [blue, line width=0.8mm] (1.1,0)--(1.4,0)
(1.25,-0.1)  -- (1.25,0.1) ;
\node[blue, above] at(1.38,0.) {$\bush_{i_1}^{[2]}$};
\end{scope}
%xxxx

\begin{scope}[shift={(0.55,0)}]
\draw [blue, line width=0.8mm] (1.1,0)--(1.4,0)
(1.25,-0.1)  -- (1.25,0.1) ;
\node[blue, below] at(1.2,-0.1) {$\bush_{i_2}^{[2]}$};
\end{scope}

\begin{scope}[shift={(1,0)}]
\begin{scope}[shift={(1,2)},rotate=90,scale=0.7]
\draw[line width=0.8mm] (-0.5,0)--(0.5,0); 
\draw[line width=0.8mm] (-0.2,-0.2)--(-0.2,0.2);
\draw[line width=0.8mm] (0.2,-0.2)--(0.2,0.2);
\coordinate (xup) at (-0.5,0);
\node[above] at (0.5,0) {$\bush'''$};

\end{scope}
\begin{scope}[shift={(1,-2)},rotate=-90,scale=0.7]
\draw[line width=0.8mm] (-0.5,0)--(0.5,0); 
\draw[line width=0.8mm] (-0.2,-0.2)--(-0.2,0.2);
\draw[line width=0.8mm] (0.2,-0.2)--(0.2,0.2);
\coordinate (xdown) at (-0.5,0);
\node[below] at (0.5,0) {$\bush^\rn{4}$};

\node[left] at (-0.8,0) {$\ell_2$};

\end{scope}

\draw[] (xup)  -- (xdown); 

\end{scope}

\end{scope}

\begin{scope}[shift={(8,0)},xscale=-1]
\draw[line width=0.8mm] (-0.5,0)--(0.5,0); 
\draw[line width=0.8mm] (-0.2,-0.2)--(-0.2,0.2);
\draw[line width=0.8mm] (0.2,-0.2)--(0.2,0.2);
\coordinate (bb) at (0.5,0);
\node[below] at  (0.15,-0.2) {$\bush^{[1]}_{a+1}$};

 \draw[dashed,line width=0.8,fill, fill opacity=0.05] (0.6,0) ellipse (2cm and 0.9cm);
\node[] at (-0,0.6){$\bush^{[1], (q)}_{a+1}$};

\begin{scope}[shift={(1,2)},rotate=90,scale=0.7]
\draw[line width=0.8mm] (-0.5,0)--(0.5,0); 
\draw[line width=0.8mm] (-0.2,-0.2)--(-0.2,0.2);
\draw[line width=0.8mm] (0.2,-0.2)--(0.2,0.2);
\coordinate (xup) at (-0.5,0);
\end{scope}
\begin{scope}[shift={(1,-2)},rotate=-90,scale=0.7]
\draw[line width=0.8mm] (-0.5,0)--(0.5,0); 
\draw[line width=0.8mm] (-0.2,-0.2)--(-0.2,0.2);
\draw[line width=0.8mm] (0.2,-0.2)--(0.2,0.2);
\coordinate (xdown) at (-0.5,0);
\end{scope}

\draw[] (xup)  -- (xdown); 

\begin{scope}[shift={(1,0)}]
\begin{scope}[shift={(1,2)},rotate=90,scale=0.7]
\draw[line width=0.8mm] (-0.5,0)--(0.5,0); 
\draw[line width=0.8mm] (-0.2,-0.2)--(-0.2,0.2);
\draw[line width=0.8mm] (0.2,-0.2)--(0.2,0.2);
\coordinate (xup) at (-0.5,0);
\end{scope}
\begin{scope}[shift={(1,-2)},rotate=-90,scale=0.7]
\draw[line width=0.8mm] (-0.5,0)--(0.5,0); 
\draw[line width=0.8mm] (-0.2,-0.2)--(-0.2,0.2);
\draw[line width=0.8mm] (0.2,-0.2)--(0.2,0.2);
\coordinate (xdown) at (-0.5,0);
\end{scope}

\draw[] (xup)  -- (xdown); 

\end{scope}

\end{scope}

\draw[] (ba)--(bb);

\draw[orange,line width=1mm] (0.5,0) edge[bend left]  (1,1);

\draw[violet,line width=1mm] (1,1) edge[bend left]  (2,0.5);

\node[above,violet] at (1.5,1) {$\GG_a$};

\draw[orange,line width=1mm] (2,0.5) edge[bend right]  (2.5,0.3);

\draw[red,line width=1.2mm] (2.5,0.3) edge[bend right]  (2.5,0.3)
 .. controls  (3, 1) and (4,-1) .. 
(5.5,-0.3)
;

\filldraw  (4,0) circle (0.04 cm);
\node[above] at (4,0) {$\alpha$};
\node[red,below]at(3.8,-0.25) {$ R^{(q)}$};

\begin{scope}[shift={(8,0)},scale=-1]
\draw[orange,line width=1mm] (0.5,0) edge[bend left]  (1,1);

\draw[orange,line width=1mm] (1,1) edge[bend left]  (2,0.5);

\draw[orange,line width=1mm] (2,0.5) edge[bend right]  (2.5,0.3);
\end{scope}

\end{tikzpicture}\]

\caption{The non-spiraling case. The rectangle $R$ (orange, violet, red) connects $\bush^{[1]}_a$ and $\bush^{[1]}_{a+1}$. Between $\bush^{[1],(q)}_a$ and $\bush^{[1],(q)}_{a+1}$ most of the rectangle $R$ travels through its lift $ R^{(q)}$ (red). Therefore, the part $\GG_a$ (violet) of $R$ between $\ell_1$ and $\ell_2$ is tremendously wide. The bushes $\bush',\bush',\bush''',\bush^\rn{4}$ are lifts of $\bush^{[1]}_{a+1}$, see Figure~\ref{Fg:lem:bush:prim comb}.}
\label{Fg:no spiral}
\end{figure}

\begin{rem}
Recall the the pure Mapping Class Group $\operatorname{MCG}$ of a $3$-punctured sphere $S^2\setminus \{a,b,\infty\}$ is trivial. By replacing $a,b$ with small open Jordan disks $D_a,D_b$, we obtain the new group $\operatorname{MCG}(S^2\setminus (D_a\cup D_b\cup \{\infty\}))\simeq \Z^2$ consisting of Dehn-twists around $D_a,D_b$, see~\cite{MCG}. This fact has the following implication.

Let $O$ be a small disk-neighborhood of $\bush^{[1]}_a\cup \Pi \cup \bush^{[1]}_{a+1}$. Then $O'\coloneqq O\setminus (\bush^{[1]}_a\cup \bush^{[1]}_{a+1})$ is topologically a disk with two holes. Therefore, a path (simple curve) $\beta \subset O'$ from $\bush^{[1]}_a$ to $\bush^{[1]}_{a+1}$ has a simple description: it first spirals $\tilde t_a\in \Z$ times around $\bush^{[1]}_a$ and then spirals $\tilde t_{a+1}\in \Z$ times around $\bush^{[1]}_{a+1}$, where $\Pi$ can be used as a reference ``path'' of zero spiralings. Below we will ignore the spiraling orientation and introduce the absolute quantity $t_a=|\tilde t_a|$.
\end{rem}

Recall \S\ref{ss:WAD} that $\gamma$ lands at the ideal boundary of $V\setminus \bbush^{[1]}$. Let $\gamma_\Pi\subset V\setminus \bbush^{[1]}$ be a path homotopic rel the endpoints to $\gamma$ in $V\setminus \bbush^{[1]}$ such that $\gamma_\Pi\cap \Pi$ is the minimal possible number. In other words, $\gamma_\Pi$ and $\Pi$ are in the minimal position -- this is well defined because $\Pi$ has infinitely many cut points. 

Consider the components $\gamma_0,\gamma_1,\dots, \gamma_{f}, \gamma_{f+1}$ of $\gamma_\Pi\setminus \Pi.$  Since $\gamma$ is properly homotopic into (a small neighborhood of) $\Pi$, there is a $t_a\le f$ such that
\begin{itemize}
\item $\gamma_t\cup \Pi$ surrounds $\bush^{[1]}_a$ for $t\in \{1,2,\dots, t_a\}$;
\item $\gamma_t\cup \Pi$ does not surround $\bush^{[1]}_{a}$ for $t\in \{t_a+1,\dots ,f\}$.
\end{itemize}
(It follows that $\gamma_t\cup \Pi$ surrounds $\bush^{[1]}_{a+1}$ for $t>t_a$, and only $\gamma_{t_a}\cup \Pi$ can surround both $\bush^{[1]}_a$ and $\bush^{[1]}_{a}$.) 

We say that $t_a=t_a(\gamma)$ is the \emph{spiraling number of $\gamma$ around $\bush^{[1]}_a$.} Similarly,  the \emph{spiraling number of $\gamma\in R^{(s)}$ around $\bush^{[1],(sp_1)}_a$} is introduced. Below are some properties of spiraling numbers:
\begin{enumerate}[label=\text{(\Alph*)},font=\normalfont,leftmargin=*]
\item \label{prop:A:thm:sat} The spiraling numbers of $\gamma_1,\gamma_2\in R$ differ by at most $1$. 
\item  \label{prop:B:thm:sat} If $\gamma_s\in R^{(s)}$ is the lift of $\gamma\in R$ into $R^{(s)}$, then $t_a(\gamma_s)\le \frac{1}{2^s} t_a(\gamma)$.  
\end{enumerate}
Indeed,~\ref{prop:A:thm:sat} follows from the fact (the disjoint curves) $\gamma_1,\gamma_2$ can be simultaneously put into the minimal position with $\Pi$. And~\ref{prop:B:thm:sat} follows from the fact that $f^{sp_1}\colon\bush^{[1],(sp_1)}_a\to \bush^{[1]}_a $ has degree $2^s$.

\subsubsection{The non-spiraling case:} $t_a(\gamma)\le 4$ for all $\gamma\in R$; see Figure~\ref{Fg:no spiral} for illustration. Write $q\coloneqq 10 \overline p$ and note that $qp_1>10 p_2$.

Let us consider the objects introduced in Lemma~\ref{lem:bush:prim comb}. It follows from~\ref{prop:B:thm:sat} that curves in $R^{(3)}$ do not spiral around $\bush^{[1],(3)}_a$. Therefore, $\ell_2$ separates the base of $R^{(q)}$ from $\bush_a^{[1]}$:  the base $\partial ^{h,0} R^{(q)}$ and $\bush_a^{[1]}$ are in different components of $\bush^{[1],(q)}_a\setminus \ell_2$.

Let $\GG_a$ be the restriction of $\FamG^{s}$ between $\ell_1$ and $\ell_2$ -- this lamination consists of the first shortest subarcs $\gamma'\subset \gamma$ between $\ell_1$ and $\ell_2$ for all $\gamma\in \FamG^s$. Since $\FamG^s$ consequently overflows $\FamG_a$ and then $R^{(q)}$, we obtain from the Gr\"otzsch inequality that
\begin{equation}
\label{eq:Grotzsch:ineq}
  (1-\delta)\Width(R)\le \Width(\FamG^{q})\le \Width( R^{(q)})\oplus \Width(\GG_a)=\Width( R)\oplus \Width(\GG_a),
\end{equation}
where $\Width( R^{(q)})=\Width(R)$ because $ R^{(q)}$ is a lift of $R$. Therefore, $\Width(\GG_a)\ge C_{0,\delta} \Width(R)\ge C_{0,\delta} K,$ where $C_{0,\delta}\to \infty$ as $\delta\to 0$.

There are two possibilities. If a substantial part of $\FamG_a$ travels through both $\bush^{[2]}_{i_1}, \bush^{[2]}_{i_2}$, then this substantial part of  $\FamG_a$ restricts to a wide lamination $\LL$  between $\bush^{[2]}_{i_1}, \bush^{[2]}_{i_2}$. Pushing $\LL$ with respect to $f_*$ and $\iota^*$ and using Lemma~\ref{generaltransform rules} (the number of iterations and the degree is bounded in terms of $\overline p$), we obtain a family $\Fam$ of non-trivial proper curves in $V\setminus \bbush^{[2]}$ starting at $\bush^{[2]}_0$ such that $\Width(\Fam)\succeq_{\overline p}  C_{0,\delta} K$. If a substantial part of $\Fam$ is vertical, then we obtain the second estimate in~\eqref{eq:1:thm:sat pull off},~\eqref{eq:2:thm:sat pull off}. If a substantial part of $\Fam$ is horizontal, then we obtain the first estimate in~\eqref{eq:1:thm:sat pull off}; applying Lemma~\ref{lem:reverse to sup chain rule} to $\Fam$, we obtain the first estimate in~\eqref{eq:2:thm:sat pull off}.

If a substantial part of $\GG_a$ omits either $\bush^{[2]}_{i_1}$ or $\bush^{[2]}_{i_2}$, then we obtain a wide wave above one of the sides of $\bush^{[2]}_{i_1}, \bush^{[2]}_{i_2}$; Wave Lemma~\ref{lem:Wave} implies that $\Width_\bullet(F)\ge C_\delta K$.

\begin{figure}[t!]
\[\begin{tikzpicture}[scale=1.3]

\begin{scope}[rotate =0]
\draw[line width=0.8mm] (-0.5,0)--(0.5,0); 
\draw[line width=0.8mm] (-0.2,-0.2)--(-0.2,0.2);
\draw[line width=0.8mm] (0.2,-0.2)--(0.2,0.2);
\coordinate (ba) at (0.5,0);
\node[below] at  (0.2,-0.2) {$\bush^{[1]}_{a}$};

 \draw[dashed,line width=0.8,fill, fill opacity=0.05] (0.6,0) ellipse (2cm and 0.9cm);
\node[] at (-0.85,0.1){$\bush_{a}^{[1],(1)}$};

\end{scope}

\begin{scope}[shift={(8,0)},xscale=-1]
\draw[line width=0.8mm] (-0.5,0)--(0.5,0); 
\draw[line width=0.8mm] (-0.2,-0.2)--(-0.2,0.2);
\draw[line width=0.8mm] (0.2,-0.2)--(0.2,0.2);
\coordinate (bb) at (0.5,0);
\node[below] at  (0.15,-0.2) {$\bush^{[1]}_{a+1}$};

 \draw[dashed,line width=0.8,fill, fill opacity=0.05] (0.6,0) ellipse (2cm and 0.9cm);
\node[] at (-0.95,-0.1){$\bush_{a+1}^{[1],(1)}$};

\end{scope}

\draw[] (ba)--(bb);

\draw[orange,line width=1mm] (0.2,-0.2)edge[bend right]  (0.8,0) ;
\draw[violet,line width=1mm] (0.8,0) edge[bend right=60]  (-1.5,0);

\draw[violet,line width=1mm] (-1.5,0) edge[bend right=40] (1.5,-0.5) ;

\draw[violet,line width=1mm](1.5,-0.5) edge[bend left=10]  (2,0);

\node[above,violet] at (0.7,0.3){$\GG_a$};

\draw[orange,line width=1mm](2,-0) edge[bend left=10]  (2.5,0.3);

\draw[red,line width=1.2mm](2.5,0.3) edge[bend right=50]  (-1.5,0.5) ;

\draw[red,line width=1.2mm] (-1.5,0.5) edge[bend right=30]  (-1.5,-0.5) ;

\draw[red,line width=1.2mm]  (-1.5,-0.5) edge[bend right=40]  (5.5,-0.3) ;

\node[red,above]at(3.8,-1.2) {$R^{(1)}$};

\begin{scope}[shift={(8,0)},scale=-1]
\draw[orange,line width=1mm] (0.5,0) edge[bend left]  (1,1);

\draw[orange,line width=1mm] (1,1) edge[bend left]  (2,0.5);

\draw[orange,line width=1mm] (2,0.5) edge[bend right]  (2.5,0.3);
\end{scope}

\filldraw  (4,0) circle (0.04 cm);
\node[above] at (4,0) {$\alpha$};

\end{tikzpicture}\]
\caption{The spiraling case. Since $R$ and its lift  $R^{(1)}$  have different spiraling numbers around $\bush^{[1]}_a$ and  $\bush^{[1]}_{a+1}$ respectively, (most of) $R$ spirals first around $\bush^{[1]}_a$ before it continues as $R^{(1)}$ (red). The part $\GG_a$ (violet) of $R$ before $R^{(1)}$ is tremendously wide.}
\label{Fg:spiral}
\end{figure}
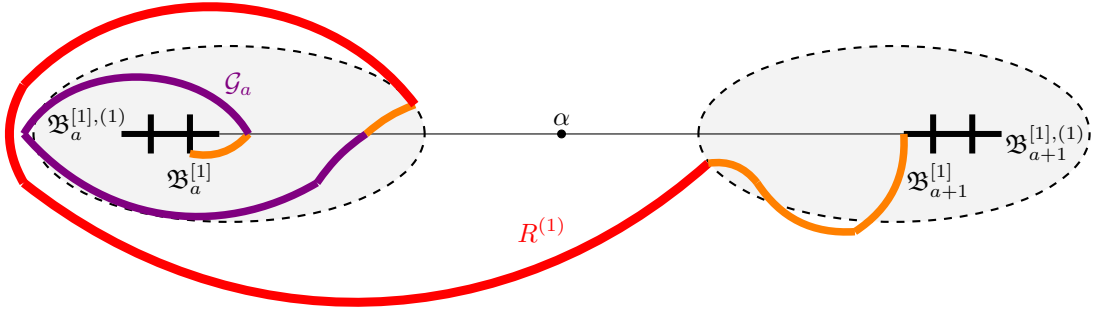

\subsubsection{The spiraling case:} $t_a(\gamma)\ge 4$ for all $\gamma\in R$; see Figure~\ref{Fg:spiral} for illustration. It follows from~\ref{prop:B:thm:sat} that the spiraling number of any curve in $\FamG^1$ around $\bush^{[1]}_a$ is strictly less than the spiraling number of any curve in $R^{(1)}$ around $\bush^{[1],(1)}_a$. Therefore, curves in $\FamG^1$ first spirals around $\bush_a$ before they continue within $\Fam^\full( R^{(1)})$. Applying the Gr\"otzsch inequality as in~\eqref{eq:Grotzsch:ineq}, we obtain a wide lamination $\GG_a$ with $\Width(\GG_a)\ge C_{2,\delta} \Width(R)$ such that $\GG_a$ creates a wave above one of the sides of $\bush_a$. Wave Lemma~\ref{lem:Wave} implies that $\Width_\bullet(f)\ge C_\delta K$.\qed

\begin{rem}\label{rem:about thm:sat case}The argument of this section will be repeated in~\S\ref{sss:prf:eq:prf:thm:psi bounds} with $\bush^{[1]}_{a+1}$ being replaced by a little Julia set $\filled^{[1]}_0$. Observe that no specific properties of $\bush^{[1]}_{a+1}$ has been used in this section (only that it is a forward invariant set that is disjoint from $\bush^{[1]}_a$).
\end{rem}

\section{Conclusions}\label{s:conclusions} In this section, we will now deduce the main theorems.

\begin{thm}[A priori beau $\psi^\bullet$-bounds]
\label{thm:psi bullet bounds}
For any combinatorial bound $\bar p$, there is an $n>1$ and $K_{\bar p}>1$ such that the following holds. If $F$ is an infinitely renormalizable $\psi^\bullet$-ql map of bounded type $\bar p$, then 
\[ \Width_\bullet \left[ (\RR^{n\bullet})^m (F)\right]\le K_{\overline p} \sp\sp\sp \text{ for }\sp m\ \gg_{\Width_\bullet(F)}\ 1,\] 
$\RR^{n \bullet}$ is a $\psi^\bullet$-ql renormalization~\eqref{RR bullet f}.
\end{thm}

\begin{proof}
Let us choose a sufficiently small $\delta>0$ such that  $C_\delta\gg \Delta_{\overline p}$,
where
 $C_\delta$ is from Theorem~\ref{thm:sat pull off} and $\Delta_{\overline p}$ is from Proposition~\ref{prop:Teichm contr}. We next choose a sufficiently big $n\gg _{\overline p, \delta} 1$ such that Theorem~\ref{thm:prim pull-off} is applicable as follows: if 
 \[\Width_\bullet \left[ (\RR^{n\bullet})^m (F)\right] = K \gg_{\overline p, \delta, n} 1,\]
 then 
\begin{enumerate}[label=\text{(\Roman*)},font=\normalfont,leftmargin=*]
\item\label{case:1:final}  either $\Width_\bullet \left[ (\RR^{n\bullet})^{(m-1)} (F)\right] \ge 2 K$; 
\item\label{case:2:final} or $\Width_\bullet \left[ \RR^{n+1\bullet} \circ (\RR^{n\bullet})^{m-1} (F)\right] \ge C_\delta K$,
\end{enumerate}
where~\ref{case:2:final} follows from Case~\ref{case:S:thm:prim pull-off} of Theorem~\ref{thm:prim pull-off} combined with Theorem~\ref{thm:sat pull off}. 

Write $G\coloneqq (\RR^{n\bullet})^{(m-1)} (F).$ Reapplying Theorem~\ref{thm:prim pull-off} for $G$ and its renormalization $\RR^{n+1 \bullet }G$, we obtain the alternative ``\ref{case:1:final} vs \ref{case:2:b:final}'', where
\begin{enumerate}[label=\text{(\Roman*')},font=\normalfont,start=2, leftmargin=*]
\item\label{case:2:b:final} or $\Width_\bullet \left[ \RR^{n+2\bullet} (G)\right] \ge C^2_\delta K$.
\end{enumerate}
Repeating the argument, we eventually obtain the alternative ``\ref{case:1:final} vs \ref{case:2:c:final}'', where
 \begin{enumerate}[label=\text{(\Roman*'')},font=\normalfont,start=2, leftmargin=*]
\item\label{case:2:c:final} or $\Width_\bullet \left[ \RR^{2n\bullet} (G)\right] \ge C^n_\delta K$.
\end{enumerate}

Therefore, if there are no a priori beau $\psi^\bullet$-bounds, then we obtain
\[\Width_\bullet \left[ (\RR^{n\bullet})^{m} (F)\right]  \succeq_F \ C^{n(m-1)}_\delta K.\]
This contradicts the Teichm\"uller contraction: by Proposition~\ref{prop:Teichm contr}, see also Remark~\ref{rem:prop:Teichm contr}, the sequence $(\RR^{n\bullet})^m (F)$ restricts to ql maps $f_{nm}:X_{nm}\to Y_{nm}$ such that  
\[ \Width\big(Y_{nm}\setminus X_{nm}\big) =O_F(\Delta_{\overline p}^{nm}) ,\]
where $\Delta_{\overline p}\ll C_\delta.$
\end{proof}

\begin{thm}[A priori beau ql-bounds]
\label{thm:psi bounds}
For any combinatorial bound $\bar p$, there is a $K_{\bar p}>1$ such that the following holds. If $f\colon U\to V$ is an infinitely renormalizable ql map of bounded type $\bar p$, then for $n\gg_{\Width(f)} 1$, the map $f_n$ has a ql restriction
\[ f_n\colon X_n\to Y_n\sp\sp\sp\text{ with }\sp\sp \Width(Y_n\setminus X_n)\le K_{\overline p}.\] 
\end{thm}
\begin{proof}Write $F\coloneqq f$, and consider the sequence $G_m\coloneqq (\RR^{n\bullet})^m (F)$ from Theorem~\ref{thm:psi bullet bounds}. We have 
\begin{equation}
\label{eq:prf:thm:psi bounds:0}
\Width_\bullet(G_m)\le K_{\overline p}\hspace{1cm}\text{ for }m\gg_{\Width_\bullet(F)} 1
\end{equation}
for $K_{\overline p}$ from Theorem~\ref{thm:psi bullet bounds}.

For $m$ satisfying~\eqref{eq:prf:thm:psi bounds:0}, consider the dynamical plane of $G_m\colon U\to V$. Let $\bbush^{[1]}$ be the cycle of little bushes in the dynamical plane of $G_m$. We {\bf~claim} that there is a space between $\filled^{[1]}_0$ and $\bbush^{[1]}\setminus \bush^{[1]}_0$; i.e., that there is a $K_2>0$ depending on $K_{\overline p}, \overline p$ such that
\begin{equation}
\label{eq:prf:thm:psi bounds}
\Width\left(V\setminus \bigcup_{i\not= 0} \bush^{[1]}_i,\ \filled^{[1]}_0\right) \le K_2.
\end{equation}
Then~\eqref{eq:prf:thm:psi bounds} together with~\S\ref{sss: psi^b to psi}, Lemma~\ref{lem:from psi-ql to ql}, and Remark~\ref{rem:lem:from psi-ql to ql}
 will imply that $G_m$ has a ql renormalization around $\filled_0^{[1]}$ with a definite modulus. 
 
\subsubsection{Proof of~\eqref{eq:prf:thm:psi bounds}}\label{sss:prf:eq:prf:thm:psi bounds}Assuming converse, the associated degeneration is arbitrary big:
\[ K\coloneqq \Width\left(\FamG\right)\gg_{K_{\overline p}} 1, \hspace{0.7cm}\text{ where }\sp \FamG\coloneqq \Fam\left(V\setminus \bigcup_{i\not= 0} \bush^{[1]}_i,\ \filled^{[1]}_0\right).\]
 We will now argue that a substantial part of $\FamG$ travels through level two little bushes; this will lead to a contradiction to the $\psi^\bullet$-bounds of Theorem~\ref{thm:psi bounds}.

Write $\bUpsilon\coloneqq \bbush^{[1]}\cup \filled^{[1]}_0$, and consider the horizontal and vertical WAD $\AA^{[1],k}_\hor, \AA^{[1],k}_\ver$ of $U^k\setminus \bUpsilon$. By~\eqref{eq:prf:thm:psi bounds:0}, 
\begin{equation}
\label{eq:prf of the claim: last thm}
\Width(\AA^{[1],k}_\ver)= O_{k,\overline p}(K_{\overline p}) \hspace{0.5cm}\text{ and hence }\hspace{0.4cm}\Width(\AA^{[1],k}_\hor)=K - O_{k,\overline p}(K_{\overline p}).
\end{equation}

As in the proof of Theorems~\ref{thm:prim pull-off}, see~\S\ref{sss:alignment with HT}, the arc diagrams $\AA^{[1],k p_1}_\hor$ eventually stabilize:
\[ \AD \left( \AA^{[1],(t+1) p_1}_\hor -C^{t+1}  \right)= \AD \left(  \AA^{[1],t p_1}_\hor -C^{t} \right) \sp\sp\sp\text{ for }\sp  t\le 3p_1, \]
where $C\gg_{p_1} 1$ is fixed. By  Lemma~\ref{lem:Alignment:Ups}, $H=\AD \left( \AA^{[1],(t+1) p_1}_\hor -C^{t+1}  \right)$ is invariant under $f^{p_1}$. Since most of the horizontal curves restrict to horizontal curves (by~\eqref{eq:prf of the claim: last thm}) the same argument as in~\S\ref{prf:P:thm:prim pull-off} provides a $\delta$-almost periodic rectangle $R\subset {V\setminus \bUpsilon},\  \Width(R)\asymp K$ between $\filled^{[1]}_0$ and some $\bush^{[1]}_a$, where $\delta$ is sufficiently small and $\filled^{[1]}_0$ replacing $\bush^{[1]}_{a+1}$ in Definition~\ref{dfn:s delta inv rect}. By Lemma~\ref{lem:per arcs:ups}, the first renormalization of $G_m$ is satellite.

We can now repeat the argument of Theorem~\ref{thm:sat pull off}, see Remark~\ref{rem:about thm:sat case}, with $\filled^{[1]}_0$ replacing $\bush^{[1]}_{a+1}$ to obtain 
\[\text{either }\sp \sp\sp\Width_\bullet\left[ \RR^{2 \bullet} (G_m)\right]\succeq_{\overline p} K \sp\sp\text{ or }\sp\sp  \Width_\bullet( G_m)\succeq_{\overline p}  K;\]
both estimates leading to a contradiction with $\psi^\bullet$-bounds of Theorem~\ref{thm:psi bullet bounds}.
\end{proof}

\end{document}